\documentclass[10pt,reqno]{amsart}
\usepackage{amsfonts}
\usepackage{amssymb,latexsym}
\usepackage{amsmath}
\usepackage{amsthm}
\usepackage{amssymb}
\usepackage{enumerate}
\usepackage{hyperref}
\usepackage{xcolor}

\usepackage{array}
\usepackage{delarray}
\usepackage{amsthm}
\usepackage{mathtools}
\usepackage{hyperref}

\usepackage{cleveref}

\numberwithin{equation}{section}
\newtheorem{theorem}{Theorem}[section]
\newtheorem{lemma}[theorem]{Lemma}
\newtheorem{remark}[theorem]{Remark}
\newtheorem{definition}[theorem]{Definition}
\newtheorem{proposition}[theorem]{Proposition}
\newtheorem{corollary}[theorem]{Corollary}

\def\XXint#1#2#3{{\setbox0=\hbox{$#1{#2#3}{\int}$ }
\vcenter{\hbox{$#2#3$ }}\kern-.6\wd0}}

\newcommand{\ba}{\mathbf{a}}
\newcommand{\mC}{{\mathcal C}}

\newcommand{\Si}{\Sigma}

\title{Mean curvature flow converging to an minimizing cone and its Hardt-Simon foliation}

\author{Jiuzhou Huang}
\address{School of Mathematics, Korea Institute for Advanced Study, 85 Hoegiro, Dongdaemun-gu, Seoul 02455, Republic of Korea}
\email{jiuzhou@kias.re.kr}
 
\date{}

\begin{document}
\allowdisplaybreaks
\begin{abstract}
    In this paper, we construct a family of mean curvature flow which converges to an area minimizing, strictly stable hypercone $\mC$ after type I rescaling, and converges to the Hardt-Simon foliation of the cone after a type II rescaling provided the cone satisfies some technique conditions. The difference from Vel\'azquez's previous results is that we drop the symmetry condition on the cone.
\end{abstract}
\maketitle

\section{Introduction}
A family of smooth hypersurfaces $\{\Sigma_t\}_{t_0\leq t<0}\subset\mathbb R^{n+1}$ immersed in $\mathbb R^{n+1}$ is called a mean curvature flow (MCF) if it evolves by
\begin{equation}\label{eq mcf}
     F_t=-  H\nu,
\end{equation}
where $F$, $ H, \nu$ are the position vector, mean curvature, and unit normal vector of $\Sigma_t$ respectively. 
If the norm of the second fundamental form $| A(t)|$ of $\Sigma_t$ blows up at time $t=0$, i.e. $\limsup_{t\nearrow 0} |A(t)|=\infty$, then we say $\Sigma_t$ develops a singularity at $t=0$. When a singularity appears, we can rescale the flow to study the structure of the singularities. The type I rescaled flow $\{\Pi_s:=e^{\frac{1}{2}s}\Sigma_{-e^{-s}}\}_{s_0\leq s<\infty}$, ($t=-e^{-s}$, $s_0=-\ln|t_0|$) satisfies 
\begin{equation}\label{eq typie I flow}
\tilde F_s=\frac{1}{2}\tilde F -\tilde{H}\tilde\nu,   
\end{equation}
where $\tilde F$, $\tilde H,\tilde \nu$ are the position vector, mean curvature, and unit normal vector of $\Pi_s$ respectively. If $\{\Pi_s\}_{s_0\leq s<\infty}$ has bounded second fundamental form, then we say the singularity is type I, otherwise we say the singularity is type II. 

Both Type I and type II singularities are very common in the study of MCF. For example, let $\Xi\subset\mathbb R^{n+1}$ be a self-shrinker. That is, $\Xi$ satisfies the equation
\begin{equation}
   \frac{1}{2}\langle  F,\nu\rangle= H,
\end{equation}
where $F,\nu,H$ are the position vector, unit outer normal, and mean curvature of $\Xi$ respectively. Then, $\{F_t:=\sqrt{|t|}\Xi\}_{-\infty\leq t<0}$ is a MCF. If $\Xi$ is smooth, then $F_t$ has type I singularity at $t=0$, otherwise, $F_t$ has type II singularity at $t=0$. When $n=2$, Illmanen \cite{I95} showed that all shrinkers are smooth. When $n\geq 3$, we do have singular shrinkers, like the the minimal cone $\mathcal C$ generated by Colifford torus $\{x_1^2+x_2^2=x_3^2+x_4^2\}\subset \mathbb S^3$ in $\mathbb R^4$. 

When a type II singularity appears, we want to explore the blowing-up rate of $|A(t)|$ near $t=0$. That is, we want to find a positive scaling function $S(t)$, $t_0\leq t<0$; such that the rescaled flow $\{S(t)\Sigma_t\}_{t_0\leq t<0}$ converges to some smooth non-planar smooth hypersurface as $t\to 0$. Equivalently, we want $\lim_{t\to 0}S(t)=\infty$, and $0<\limsup_{t\to 0} S(t)\sup_{\Sigma(t)}|A|<\infty$. In some cases, we can get an explicit expression of the function $S$. For instance, Vel\'azque \cite{V94} constructed a family of MCF $\{\Sigma_t^{l}\}_{t_0\leq t<0;l\geq 2}$ in $\mathbb R^{2n}$ for $n\geq 4$, which has $O(n)\times O(n)$ symmetry. In Vel\'azquez' construction, we can take $S(t)=|t|^{-\frac{1}{2}-\sigma_l}$, where $\sigma_l=\sigma_l(n)$ is some positive constant. To be more precise, Vel\'azquez's solution converges to the Simons' cone in any fixed annulus centered at $O$ as $s\to \infty$ after type I rescaling; and converges in $C^0$ sense to a smooth minimal hypersurface tangent to Simons' cone at infinity after a type II rescaling by multiplying $S(t)$. Vel\'azquez's solution was studied throughly by Guo-Sesum in \cite{GS} in which they proved that the Type II rescaled flow actually converges locally smoothly to the minimal hypersurface. Moreover,  a sub-family of Vel\'azquez's solutions have mean curvature blowing up near the origin at a smaller rate than that of the second fundamental form. On the other hand, Stolarski \cite{S} showed that many other Vel\'azquez's solutions have bounded mean curvature. 

Besides the Simons' cone case, Liu \cite{Liu} recently constructed a compact mean curvature flow which converges to a quadratic cone $\mC_{p,q}$ in $\mathbb R^{n+1}$ ($p,q\geq 2$ with $p+q+1=n\geq 7$, and $p,q\geq 3$ when $n=7$) after type I rescaling, and converges to a smooth minimal hypersurface tangent to $\mC_{p,q}$ at infinity after a type II rescaling with $S(t)=|t|^{-\frac{1}{2}-\sigma_l}$. Moreover, for $l$ sufficiently large, the mean curvature of the flow remains bounded. Note that Simons' cone is a quadratic cone $\mC_{p,q}$ with $p=q$. Thus, Liu's result is a generalization of  Vel\'azque and Stolarski's. 

On the other hand, $\mC_{p,q}$ are invariant by the action of $O(p)\times O(q)$, thus still can be parametrized by a single parameter. So Vel\'azquez and Liu's solution can are based on the analysis of ODEs. A natural question is whether we can consider the case when the cone is non-symmetric. This is the main result of this paper. 
To be more precise, we show that
\begin{theorem}\label{thm main aim intro}
    Let $n\geq 7$ be an integer, $\mC\subset \mathbb R^{n+1}$ be a regular  minimizing, strictly stable hypercone with isolated singularity. Let $\mathcal L_{\mC}$ be the Jacobi operator, $S_+$ be the Hardt-Simon foliation. If $\mathcal L_{\mC}$ has eigenvalue $\lambda_l$ satisfies the condition \eqref{eq lambda l} for some integer $l$, and the constant $\alpha,\tilde\alpha$ in \eqref{eq psi asy} satisfies \eqref{eq alpha ine}\footnote{Note that the condition on $\mC$ is satisfied for Simons' cone for $n$ large (see Remark \ref{rem nonemp} for more details).}. Then for $|t_0|\ll 1$ (depending on $n,\Sigma,l$), there is a MCF $\{\Sigma_t\}_{t_0\leq t<0}$ for which, the type I rescaled hypersurface $\{\Pi_s\}$ converges to $\mathcal C$ locally smoothly, i.e. for any $0<r<R<\infty$
\begin{align*}
    \Pi_s\to \mathcal C\quad \text{in }C^\infty(B(O,R)\setminus B(O,r)).
\end{align*}
as $s\to\infty$. Moreover, the type II rescaled hypersurfaces $\{\Gamma_\tau\}_{\tau_0\leq \tau<\infty}$ converges to $S_{\kappa,+}$ locally smoothly, i.e.
\begin{align*}
    \Gamma_\tau\to S_{\kappa,+} \text{ in }C_{loc}^\infty (S_{\kappa,+}).
\end{align*}
as $\tau\to\infty$. Here $S_{\kappa,+}$ is one piece of the Hardt-Simon foliation tangent to $\mC$ at infinity.
\end{theorem}
In fact, the flow we get is admissible which means that it is a normal graph over $S_{\kappa,+}$ in the tip region and a normal graph over $\mC$ outside the tip region after rescaling (see Section \ref{sec adm} for the precise definition of admissible flows). Moreover, we have more detailed description about the flow $\{\Sigma_t\}$ in Theorem \ref{thm main aim}. 

Now, let's recall some terminology in Theorem \ref{thm main aim intro}. Here, a hypercone $\mC\subset \mathbb R^{n+1}$ is called regular if it has an isolated singularity at the origin, or equivalently, the link $\Sigma:=\mC\cap \mathbb S^{n}$ of $\mC$ is a smooth hypersurface in $\mathbb S^{n}$. The strict stability of a hypercone $\mC$ was introduced by Hardt-Simon \cite{HS}, where they also introduced the foliation of $\mathbb R^{n+1}\setminus \mC$, which is called Hardt-Simon foliations now. The reader can refer to Appendix \ref{sec HS foliation} for more details about Hardt-Simon foliations. A regular cone $\mC$ is called strictly stable if the first eigenvalue $\mu_1$ of 
\begin{equation}\label{eq Lsigma def}
\mathcal L_\Sigma:= \Delta_\Sigma+|A_\Sigma|^2    
\end{equation}
satisfies 
\begin{equation}
    \mu_1>-\frac{(n-2)^2}{4}.
\end{equation}

Since the quadratic cone $\mC_{p,q}$ is a minimizing and strictly stable hypercone when $p, q \ge 2$ with $p + q + 1 = n \ge 7$, and when $p, q \ge 3$ in the borderline case $n = 7$ (see Bombieri--De~Giorgi--Giusti~\cite{BDG}, Lawson~\cite{La}, and Sim\~oes~\cite{SP}), Theorem~\ref{thm main aim intro} generalize of the results of Vel\'azquez and Liu in higher dimensions in some sense.

The construction of $\{\Sigma_t\}$ in Theorem~\ref{thm main aim intro} largely follows the simplified version of Vel\'azquez's construction developed by Guo--\v{S}e\v{s}um~\cite{GS}. A key distinction between our setting and those considered by Vel\'azquez, Guo--\v{S}e\v{s}um, and Liu is that we must deal with a partial differential equation rather than an ordinary differential equation, since the absence of symmetry in $\mC$ forces us to account for the non-radial variables. Fortunately, most of the estimates required in the symmetric case extend to the non-symmetric setting after suitable modifications. 

A more essential difference lies in the lack of a global graphical representation of the flow over the plane. In the symmetric case, one only needs to study a single rotationally symmetric slice of the flow, which can be written as an entire graph over a half-line. In contrast, in our setting the (rescaled) flow can be written as a graph over $S_{\kappa,+}$ only in the tip region, while in other regions it must be expressed as a graph over a portion of the cone~$\mC$. This geometric restriction introduces additional difficulty in constructing barriers for the $C^0$ estimate near the tip. To overcome this obstacle, we make use of the strict stability of~$\mC$ together with the Hardt--Simon foliation of minimal hypersurfaces asymptotic to~$\mC$~\cite{HS}. More precisely, by rescaling this foliation at two distinct spatial scales and exploiting the positivity of the first eigenvalue of the Jacobi operator of $S_{\kappa,+}$ on suitable local domains, we construct upper and lower barriers that yield the desired control. Further details are given in Subsection~\ref{sec inner reg}.

We should emphasize that both the minimizing property and the strict stability assumption in Theorem \ref{thm main aim intro} on $\mC$ are heavily used in our proof. More concretely, the strict stability is used to deduce the coercivity of the linearized operator of \eqref{eq typie I flow} in section \ref{sec pre}, which is the foundation of the spectral method we employed in the analysis of our solution in the intermediate region. The area minimizing property of $\mC$ are used in two different places. First, we need the minimizing property of $\mC$ to provide the Hardt-Simon foliation described by \cite{HS}, which is the smooth limit of our type II rescaled flow. Secondly, we need the minimizing property to deduce that the foliation is also minimizing, thus a locally stable minimal hypersurface so that its Jacobi operator has positive fist eigenvalue on every compact domain of the hypersurface. The latter is the key to construct our upper barrier for the flow in tip region described above.  

The rest of the paper is organized as follows. In Section \ref{sec pre}, we prove some prelimilary results about the spectrum of the linearized operator of the flow \eqref{eq typie I flow} which is the base for the analysis of the constructions. Then we define admissble mean curvature flows in Section \ref{sec adm}. Our solution is constructed under this frame work. In Section \ref{sec outline}, we describe the ideas to construct our solutions, which includes the degree method, and the statement of $C^0$ and $C^2$ estimates needed for the degree method. Assuming these estimates, we prove Theorem \ref{thm main aim intro} at the end of Section \ref{sec outline}. The $C^0$ and $C^2$ estimates are proved in Section \ref{sec C0 est}, \ref{sec higher est}, respectively. The proofs are divided into outer, intermediate and inner (tip) regions in different scales, and mostly follow that of \cite{GS} except for the two differences described above
described above. We omit details of the proof if it is very similar to that of the symmetric case. Since the flow we constructed are normal graphs over cone $\mC$ or a smooth hypersurface $S_{\kappa,+}$, we first collect some facts about the normal graphs and evolution equations along MCF in Appendix \ref{sec graph evo}. In Appendix \ref{sec cone prea}, we collect some facts about the geometry of cones, and prove a Morrey type inequality on cones. At last, we collect Hardt-Simon's results on strictly stable, minimizing hypercones in Appendix \ref{sec HS foliation}. These results are important for the construction of barriers and proof of convergence in the tip region.
 \section*{Acknowledgement}
 The author thanks Professor Herrero and Professor Vel\'azquez for kindly sending their unpublished notes \cite{HV}. He is grateful to Professor Guo for answering questions about the paper \cite{GS}. At last, the author is deeply indebted to Professor Choi for many helpful discussions and his constant support. The author is supported by KIAS individual Grant MG088502.

\section{Preliminaries}\label{sec pre}
In this section, we prove some preliminary properties of the linearized operator corresponding to \eqref{eq typie I flow}. These properties are the foundations of the spectral method to construct the solution. 

Let $Y:=\mC=\mathbb R_+\times\Sigma=\{(y,\theta)|y\in\mathbb R_+,\theta\in\Sigma\}\subset\mathbb R^{n+1}$ be a regular hypercone in $\mathbb R^{n+1}$, where $\Sigma=\mC\cap\mathbb S^{n}$ is a smooth hypersurface in $\mathbb S^{n-1}$. Let 
\begin{equation}\label{eq lin res}
    L_{\mC}v:=\Delta v+|A|^2v+\frac{1}{2}(v-\langle Y,\nabla v\rangle),
\end{equation}
with $\Delta, A$ the Laplacian operator and length square of the second fundamental form of $\mC$, $E(v)
$ is given in \eqref{eq error outer}. Since $\mC$ is a cone, \eqref{eq lin res} becomes
\begin{equation}
    L_{\mC}v=\Delta v+|A|^2v+\frac{1}{2}(v-yv_y).
\end{equation}
\subsection{A coercivity lemma}
\begin{lemma}\label{lem cocer}
Suppose that $\mC$ is a strictly stable minimal hypercone in $\mathbb R^{n+1}$, then there exists a constant $C(n)>0$, ${\tilde \varepsilon} (\Sigma)\in(0,1)$ depending only on $n,\Sigma$,  small such that
    \begin{equation}
        \int_{\mC} -(L_{\mC} uu) e^{-\frac{|Y|^2}{4}}dvol_{\mC}\geq {\tilde \varepsilon}  \int_{\mC} |\nabla_{\mC} u|^2e^{-\frac{|Y|^2}{4}}dvol_{\mC} -C\int_{\mC} u^2  e^{-\frac{|Y|^2}{4}}dvol_{\mC}
    \end{equation}
    for every $u\in C_c^\infty(\mC)$.
\end{lemma}
\begin{proof}
Note that for any $u\in C_c^\infty(\mC)$, $|\nabla_{\mC}u|^2=|\partial_ yu|^2+\frac{|\nabla_\Sigma u|^2}{y^2}$, and $|A_{\mC}|^2=\frac{|A_\Sigma|^2}{y^2}$. Thus, for any  $u\in C_c^\infty(\mC)$, we have 
 \begin{align*}
        &\int_{\mC} -(L_{\mC}uu) e^{-\frac{|Y|^2}{4}}dvol_{\mC}
        =\int_{\mC} |\nabla_{\mC} u|^2-(|A_{\mC}|^2+\frac{1}{2})u^2  e^{-\frac{|Y|^2}{4}} dvol_{\mC}\\
        =&\int_{\mC} \left(|\partial_y u|^2+\frac{|\nabla_\Sigma u|^2-|A_\Sigma|^2u^2}{y^2}-\frac{1}{2}u^2\right)e^{-\frac{|Y|^2}{4}}dvol_{\mC}=I_1+I_2+I_3,
\end{align*}
where
\begin{align*}
    I_1=&\varepsilon \int_{\mC} \left(|\nabla_{\mC} u|^2-\frac{|A_\Sigma|^2}{y^2}u^2\right)e^{-\frac{|Y|^2}{4}}dvol_{\mC},\\
    I_2=&(1-\varepsilon)\int_{\mC}\left(|\partial_y u|^2+\frac{|\nabla_\Sigma u|^2-|A_\Sigma|^2u^2}{y^2}\right)e^{-\frac{|Y|^2}{4}}dvol_{\mC},\quad I_3=-\frac{1}{2}\int_{\mC}u^2e^{-\frac{|Y|^2}{4}}dvol_{\mC},
\end{align*}
with $\varepsilon\in (0,1)$ to be determined. In polar coordinates, $dvol_{\mC}=y^{n-1}d\theta dy$, we have
\begin{align*}
    I_2=(1-\varepsilon)\left[\int_\Sigma\int_0^\infty |\partial_y u|^2y^{n-1}e^{-\frac{y^2}{4}}dy d\theta+\int_0^\infty \int_\Sigma\left(\frac{|\nabla_\Sigma u|^2-|A_\Sigma|^2u^2}{y^2}\right)y^{n-1}e^{-\frac{|y|^2}{4}}d\theta dy\right].
\end{align*}
Let $\mu_1$ be the first eigenvalue of $\mathcal L_\Sigma$ defined in \eqref{eq Lsigma def}, i.e.  $\mu_1=\inf_{v\in C^\infty(\Sigma),v\neq 0}\frac{\int_\Sigma-\mathcal L_\Sigma v vd\theta }{\int_\Sigma v^2d\theta}$. Thus,
\begin{align*}
    I_2\geq (1-\varepsilon)\int_\Sigma \int_0^\infty \left(\frac{\mu_1}{y^2}u^2+|\partial_yu|^2\right)y^{n-1}e^{-\frac{y^2}{4}}dyd\theta
\end{align*}
By the same argument as in Herrero-Vel\'azquez (step 2 of proof of Lemma 2.3 in \cite{HV}),  there exists a constant $C'(n)$ depending only on $n$ such that, for any fixed $\theta\in\Sigma$,
    \begin{equation}\label{eq Hardy type HV}
         \int_0^\infty  \left(|\partial_yu(\cdot,\theta)|^2-\frac{(n-2)^2}{4y^2}u(\cdot,\theta)^2\right)y^{n-1}e^{-\frac{r^2}{4}} dy\geq -C'\int_0^\infty u(\cdot,\theta)^2 y^{n-1}e^{-\frac{y^2}{4}}dy.
    \end{equation}
    Thus,
\begin{align*}
    I_2\geq (1-\varepsilon)\Big[-C'\int_\Sigma \int_0^\infty u^2y^{n-1}e^{-\frac{y^2}{4}}dyd\theta+\big(\mu_1+\frac{(n-2)^2}{4}\big)\int_\Sigma\int_0^\infty \frac{u^2}{y^2}y^{n-1}e^{-\frac{y^2}{4}}dyd\theta\Big],
\end{align*}
and 
\begin{align*}
    I_1+I_2+I_3\geq  \varepsilon\int_{\mC}\left(|\nabla_{\mC}u|^2+\frac{\frac{(1-\varepsilon)}{\varepsilon}\big(\mu_1+\frac{(n-2)^2}{4}\big)-|A_\Sigma|^2}{y^2}u^2\right)e^{-\frac{|Y|^2}{4}}dvol_{\mC}-(C'+\frac{1}{2})\int_{\mC}u^2e^{-\frac{|Y|^2}{4}}dvol_{\mC}
\end{align*}
    Since $\mC$ is strictly stable, we have $\mu_1+\frac{(n-2)^2}{4}>0$. On the other hand, $\Sigma$ is compact, we can take $\varepsilon={\tilde \varepsilon} (\Sigma)\in(0,1)$ sufficiently small such that
    \begin{align*}
        \frac{(1-{\tilde \varepsilon} )}{{\tilde \varepsilon} }(\mu_1+\frac{(n-2)^2}{4})-\sup_\Sigma|A_\Sigma|^2\geq 0. 
    \end{align*}
    Then we get
    \begin{equation}
        \int_{\mC} -(L_{\mC}uu) e^{-\frac{|Y|^2}{4}}dvol_{\mC}\geq {\tilde \varepsilon}  \int_{\mC} |\nabla_{\mC} u|^2e^{-\frac{|Y|^2}{4}}dvol_{\mC}-(C'+\frac{1}{2})\int_{\mC} u^2 e^{-\frac{|Y|^2}{4}}dvol_{\mC}.
    \end{equation}
    We can then take $C=C'(n)+\frac{1}{2}$. 
\end{proof}
\subsection{Functional analysis on $\mC$}
Similar to the symmetric case \cite{V94} (see also \cite{HV}), we define the following functional space to facilitate our analysis of the operator $\mathcal L$ on $\mC$. Let 
\begin{equation}
    \begin{aligned}
        &\langle f,g\rangle_W:=\int_{\mC} fg e^{-\frac{1}{4}|Y|^2}dg\\
        &L^2_W(\mC):=\{f:\mC\to\mathbb R|\|f\|_W:=\langle f,f\rangle_W^{\frac{1}{2}}<\infty\}\\
        &H_W^k(\mC):=\{f:\mC\to\mathbb R|\|f\|_{W,k}:=\|f\|_W+\|\nabla_{\mC}f\|_W+\cdots+\|\nabla_{\mC}^k f\|_W<\infty\},\quad k\geq 1.
    \end{aligned}
\end{equation}
be the completetion of $C_c^\infty(\mC)$ under the norm $\|\cdot\|_W$, $\|\cdot\|_{W,k}$ respectively. Then we can do the Friedrichs extension as in \cite{HV} to extend the operator $L_{\mC}$ as a self-adjoint operator (still denoted by $L_{\mC}$) whose domain 
\begin{equation}\label{eq def H}
    D(L_{\mC})\subset H^1_W(\mC)=:\mathbf H. 
\end{equation}
Moreover, we can prove that $(L_{\mC}-\lambda)^{-1}:L^2_W(\mC)\to L_W^2(\mC)$ is locally compact and compact for $\lambda>0$ large by the same argument of Lemma 2.3 in \cite{HV}, and show that there are eigenfunctions $\{\varphi_i\}_{i=1}^\infty\subset H^1_W(\mC)$ of $L_{\mC}$ such that 
\begin{equation}\label{eq eig}
    L_{\mC}\varphi_i=-\lambda_i\varphi_i,
\end{equation} 
which forms a basis of $L^2_W(\mC)$, with $\lambda_i\to\infty$ as $i\to\infty$. By elliptic theory, $\varphi_i\in C^\infty(\mC)$. 

To solve out $\varphi_i$, we write $\varphi_i=\sum_{j=1}^{\infty}\varphi_{ij}(y)\omega_j$, where $\varphi_{ij}(y)=\int_\Sigma\varphi(y,\theta)\omega_j(\theta)d\theta$, and $\{\omega_j\}_{j=1}^\infty$ is an $L^2$ orthonormal basis of $\mathcal L_\Sigma$ defined in \eqref{eq Lsigma def}, with 
$$\mathcal L_\Sigma\omega_j=-\mu_j\omega_j.$$
Note that 
$$L_{\mC}u=\frac{1}{y^{n-1}}\frac{\partial}{\partial y}(y^{n-1}\frac{\partial u}{\partial y})+\frac{\mathcal L_\Sigma u}{y^2}+\frac{1}{2}(u-yu_y).$$
Thus, we have
\begin{equation}\label{eq eig rai}
    \varphi_{ij}''+\frac{n-1}{y}\varphi_{ij}'-\frac{\mu_j}{y^2}\varphi_{ij}+\frac{1}{2}(\varphi_{ij}-y\varphi_{ij}')=-\lambda_{i}\varphi_{ij}.
\end{equation}
Here, prime means partial derivative with respect to $y$. Near $y=0$, the solution of \eqref{eq eig rai} behaves like 
$y^{\alpha_j^\pm}$, where 
\begin{equation}
    \alpha_j^{\pm}=\frac{-(n-2)\pm \sqrt{(n-2)^2+4\mu_j}}{2}
\end{equation}
is the solution of
\begin{align*}
    \alpha^2+(n-2)\alpha-\mu_j=0.
\end{align*}
In fact, let $\varphi_{ij}=y^{\alpha_j^\pm}\phi(\frac{y^2}{4})=y^{\alpha_j^\pm}\phi(\eta)$, we have $\phi(\eta)$ satisfies the equation
\begin{align*}
    \eta \phi''(\eta)+(\alpha_j^\pm+\frac{n}{2}-\eta)\phi'(\eta)-(-\frac{1}{2}(1-\alpha_j^\pm)-\lambda_{i})\phi(\eta)=0
\end{align*}
whose solution is given by $M(-\lambda_{i}-\frac{1}{2}(1-\alpha_j);\alpha_j^\pm+\frac{n}{2};\frac{r^2}{4})$. Here $M$ is the Kummer's function 
 defined by
 $$
 M(a;b;\xi)=1+\sum_{k=1}^{\infty}\frac{a(a+1)\cdots(a+k-1)}{b(b+1)\cdots(b+k-1)}\frac{\xi^k}{k!}
 $$
 and satisfies
 $$\xi M(a;b;\xi)+(b-\xi)\partial_\xi M(a;b;\chi)-a M(a;b;\xi)=0.$$
 Note that if $a\neq0,-1,-2,\cdots$, then $M(a,b;0)=1$; $M(a;b;\xi)\sim\frac{\Gamma(b)}{\Gamma(a)}e^\xi\xi^{a-b}$ as $\xi\to\infty$ if $a\neq 0,-1,-2,\cdots$, where $\Gamma(s)$ denotes the standard Euler's gamma function (see page 25 of \cite{HV}). This means that $$\varphi_{ij}\sim y^{\alpha_j}\text{ as }y\to0 ^+,$$ 
 and $$\varphi_{ij}\sim\frac{\Gamma(\alpha_j+\frac{n}{2})}{\Gamma(-\lambda_i-\frac{1}{2}(1-\alpha_j))}e^{\frac{y^2}{4}}(\frac{y^2}{4})^{-\frac{\alpha_j}{2}-\frac{n+1}{2}-\lambda_i}\text{ as }y\to\infty\text{ if }-\lambda_i-\frac{1}{2}(1-\alpha_j^\pm)\neq 0,-1,-2,\cdots.$$ 
To ensure that $\varphi_i\in H^1_W(\mC)$, we need $$2\alpha_j^\pm-2>-n,$$
and $$-\frac{1}{2}(1-\alpha_j^\pm)-\lambda_{i}=a=0,-1,-2,\cdots.$$ 
Thus we have to take positive sign in $\alpha_j^\pm$, 
$$\lambda_{i}=-\frac{1}{2}(1-\alpha_j^+)+i,\quad i=0,1,2,\cdots,\quad j=1,2\cdots.$$
and 
$$\varphi_{ij}=y^{\alpha_{j}}M(-i;\alpha_j^++\frac{n}{2};\frac{y^2}{4}), \quad i=0,1,2,\dots, \quad j=1,2\cdots.$$ 
To simplify notations, we omit $+$ in $\alpha_j^+$ and use $\alpha_j$ in the rest of the paper. 

Since the eigenfunction $\varphi_{ij}\omega_j$ (no summation in $j$) and the corresponding eigenvalue $-\frac{1}{2}(1-\alpha_j)+i$, $i=0,1,2,\cdots$, $j=1,2\cdots$ depend on both $i$ and $j$, we use double sub-index $i,j$ to number them. That is, $L_{\mC}$ has eigenfunction and eigenvalues  
\begin{equation}\label{eq phiij eq}
\begin{aligned}
    \varphi_{ij}=&\tilde \varphi_{ij}\omega_j\quad(\text{no summation in } j)\\
    \lambda_{ij}=&-\frac{1}{2}(1-\alpha_j)+i;\quad i=0,1,2,\cdots;\, j=1,2\cdots,
\end{aligned}
\end{equation}
where
\begin{equation}\label{eq tilde phiij}
\begin{aligned}
     \tilde \varphi_{ij}
    =&c_{ij}y^{\alpha_{j}}(1+\sum_{m=1}^i(-1)^mK_{mij}y^{2m}) 
\end{aligned}
\end{equation}
with $K_{mij}=\frac{(-1)^m(-i)^{(m)}}{(\alpha_j+\frac{n}{2})^{(m)}4^mm!}>0$, and $a^{(m)}:=a(a+1)\cdots (a+m-1)$ for $m\geq 1$; and $c_{ij}$ is the normalization constant such that $\|\varphi_{ij}\|_W=1$. 

We can order the eigenvalues by magnitudes. In this case, we use a single index to mark the eigenfunctions and eigenvalues. For instance,  $\varphi_k$, $\lambda_k$ denotes the $k-th$ eigenfunction and eigenvalue of $L_{\mC}$ ($\lambda_1<\lambda_2\leq \lambda_3\leq \cdots$). We always use this convention when we use a single index to mark the eigenfunctions and eigenvalues. 

For simplicity, we will write 
$$\alpha_1=\alpha.$$ 
We choose $l\in \mathbb N_+$ such that there exists $i_1\geq 0$ s.t. 
$$\lambda_{i_11}=\lambda_l>0.$$ 
That is, the $l$-th eigenvalue $\lambda_{l}$ comes from $\omega_1$ and is positive. Moreover, we need that there exists $\delta_l>0$ such that 
\begin{equation}\label{eq lambda l}
    \lambda_{l+1}\geq \lambda_l+\delta_l.
\end{equation}
For such a fixed $l$, define $m=m(l)$ as 
$$m=\sup\{m'|\lambda_{0m'}\leq  \lambda_l\}.$$ 
For $2\leq k\leq m$, define 
$$i_k=\max\{i|\lambda_{ik}<\lambda_{l}\},\quad k\geq 2. $$ 
For later use, we also define  
\begin{equation}\label{eq sigmal def}
    \sigma_l=\frac{\lambda_l}{1-\alpha}.
\end{equation}

\section{Admissible flows}\label{sec adm}
In this section, we define admissible mean curvature flows. It will be the main objects we will consider in the following sections. Let $n\in\mathbb N_+$ be a large integer,
\begin{equation}\label{eq para}
    l\in \mathbb N_+,\,\Lambda=\Lambda(n)\gg 1,\, 0<\rho\ll 1\ll\beta,\,  |t_0|\ll1.
\end{equation}
 be constants to be determined, $t_0<t^\circ<0$. Assume there is a one-parameter family of smooth hypersurfaces $\{\Si_t\}_{t_0\leq t\leq t^\circ}$ in $\mathbb R^{n+1}$ moving by mean curvature. That is, the position vector $F$ of $\Sigma_t$ satisfies \eqref{eq mcf}.
We say $\{\Si_t\}_{t_0\leq t\leq t^\circ}$ is admissible if 

(1) The flow is a normal graph over $|t|^{\frac{1}{2}+\sigma_l}S_{\kappa,+}$ inside $B(O;2\beta^2|t|^{\frac{1}{2}+\sigma_l})$ with profile function $\hat u(x,t)$, where $\sigma_l$ is defined in \eqref{eq sigmal def}, and $S_{\kappa,+}={\kappa}^{\frac{1}{1-\alpha}}S_+$ is one piece of the Hardt-Simon foliation with $\kappa \approx 1$ is defined in \eqref{eq k def}.

(2) The flow is a normal graph over $X:=\mC=\{(x,\theta)|x\in\mathbb R_+,\theta\in\Sigma\}$ outside $B(O;\beta|t|^{\frac{1}{2}+\sigma_l})$. In other words, we can parametrize $\Si_t$ by
\begin{equation}\label{eq Sigma t gra eq}
    F(x,\theta,t)=X(x,\theta)+u(x,\theta,t)\nu(x,\theta)
\end{equation}
for $(x,\theta)\in [\beta|t|^{\frac{1}{2}+\sigma_l},\infty)\times \Sigma$, $t_0\leq t\leq t^\circ$, with $u(x,\theta,t)$ satisfying the equation 
\begin{equation}\label{eq flow outer}
    u_t=-\frac{\bar{H}}{\nu\cdot\bar{\nu}}
    =\mathcal L_{\mC}u+E(u)
\end{equation}
where $\nu$ is the unit normal of $X$, $\bar H,\bar \nu$ are the mean curvature and unit normal of $F$, respectively; $\mathcal L_{\mC}$ is defined in \eqref{eq mathcal Lcu} , and $E(u)$ is defined in $\eqref{eq error outer}$.

(3) For the function $u(x,\theta,t)$, there holds 
\begin{equation}\label{eq adm x}
    x^{|\gamma|}|\nabla^{\gamma}_{(x,\theta)}u(x,\theta,t)|<\Lambda(|t|^{i_1}x^\alpha+x^{2\lambda_l+1}),\quad |\gamma|\in\{0,1,2\}
\end{equation}
for $(x,\theta)\in [\beta|t|^{\frac{1}{2}+\sigma_l}, \rho]\times\Sigma$, $t_0\leq t\leq t^\circ$. Here, $\gamma$ is a multi-index and $|\gamma|$ is the length of $\gamma$.

We can divide the admissible flows into three regions and do rescalings in the corresponding region.
\begin{itemize}
    \item The outer region: $\Sigma_t\setminus B(O;\sqrt{|t|})$. In this region, we parametrize the flow $\Si_t$ by \eqref{eq Sigma t gra eq} with $X=\mC$ and $u(x,\theta,t)$ satisfies the equation \eqref{eq flow outer}.

\item The intermediate region: $\Sigma_t\cap\Big(B(O;\sqrt{|t|})\setminus B(O;\beta|t|^{\frac{1}{2}+\sigma_l})\Big)$: we do the Type I rescaling
    \begin{equation}
        \Pi_s=\frac{1}{\sqrt{|t|}}\Si|_{t=-e^{-s}}
    \end{equation}
    By this rescaling, the intermediate region is dilated to $\Pi_s\cap(B(O;1)\setminus B(O;\beta e^{-\sigma_l s}))$ for $s_0\leq s\leq s^\circ$, where $s_0=-\ln|t_0|$ and $s^\circ=-\ln(|t^\circ|)$. Let $x=e^{-\frac{s}{2}}y$, and
    \begin{equation}\label{eq vu}
        v(y,\theta,s)=e^{\frac{s}{2}}u(e^{-\frac{s}{2}}y,\theta,-e^{-s}).
    \end{equation}
We can parametrize the rescaled hypersurface $\Pi_s$ in the intermediate region by    
    \begin{equation}\label{eq v eq}
    \begin{aligned}
    \tilde F(y,\theta,s)=Y(y,\theta)+v(y,\theta,s)\nu(y,\theta),
    \end{aligned}     
    \end{equation}
where $Y(y,\theta)=e^{\frac{s}{2}}X(e^{-\frac{s}{2}}y,\theta)$ is the rescaled cone (thus still a cone) in $(y,\theta)$ coordinates. By \eqref{eq flow outer} and \eqref{eq vu}, $v$ satisfies
$$v_s =\frac{1}{2}v-\frac{1}{2} yv_y-\frac{\tilde H}{\nu\cdot\tilde\nu}=L_{\mC}v+E(v)$$
with   
\begin{equation}\label{eq LmC def}
    L_{\mC}v:=\Delta_{\mC} v+|A_{\mC}|^2v+\frac{1}{2}(v-yv_y),
\end{equation}
and $E(v)$ is defined in $\eqref{eq error outer}$. Here $\tilde H,\tilde\nu$ are the mean curvature and unit outer normal of $\tilde F$ respectively.

Since $v$ satisfies \eqref{eq vu}, the admissible condition \eqref{eq adm x} is rescaled to
\begin{equation}\label{eq adm y}
    y^{|\gamma|}|\nabla^\gamma_{(y,\theta)}v(y,\theta,s)|<\Lambda e^{-\lambda_ls}(y^\alpha+y^{2\lambda_l+1}),\quad |\gamma|\in \{0,1,2\}
\end{equation}
for $(y,\theta)\in [\beta e^{-\sigma_ls}, \rho e^{\frac{s}{2}}]\times \Sigma$, $s_0\leq s\leq s^\circ$.

\item the tip region $\Si_t\cap B(O;\beta|t|^{\frac{1}{2}+\sigma_l})$, we do the Type II rescaling
\begin{equation}\label{eq type ii rescaling graph}
    \Gamma_\tau=\frac{1}{|t|^{\frac{1}{2}+\sigma_l}}\Si_t|_{t=-(2\sigma_l\tau)^{\frac{-1}{2\sigma_l}}}.
\end{equation}
By this rescaling, the tip region is dilated to $\Gamma_\tau\cap B(O;\beta)$ for $\tau_0\leq \tau\leq \tau^\circ$, where $\tau_0=(2\sigma_l)^{-1}|t_0|^{-2\sigma_l}$ and $\tau^\circ=(2\sigma_l)^{-1}|t^\circ|^{-2\sigma_l}$. Let $z=e^{\sigma_ls}y=(2\sigma_l\tau)^{\frac{1}{2}}y$, and
\begin{equation}\label{eq de w vu}
    \begin{aligned}
        w(z,\theta,\tau)
        =&|t|^{-\frac{1}{2}-\sigma_l}u(|t|^{\frac{1}{2}+\sigma_l}z,\theta,t)|_{t=-(2\sigma_l\tau)^{-\frac{1}{2\sigma_l}}}=e^{\sigma_ls}v(e^{-\sigma_ls}z,\theta,s)|_{s=\frac{1}{2\sigma_l}\ln(2\sigma_l\tau)}.
    \end{aligned}
\end{equation} 
We can parametrize $\Gamma_\tau$ outside $B(O,\beta)$ by
\begin{equation}\label{eq til til F eq}
    \begin{aligned}
        \hat{F}(z,\theta,\tau)=Z(z,\theta)+w(z,\theta,\tau)\nu(z,\theta),
    \end{aligned}
\end{equation}
where $Z(z,\theta)=(2\sigma_l\tau)^{c_l}X((2\sigma_l\tau)^{-c_l}z,\theta)$ is the cone in $(z,\theta)$ coordinates, $c_l=\frac{1}{2}+\frac{1}{4\sigma_l}>0$. By \eqref{eq flow outer} and \eqref{eq de w vu}, $w$ satisfies
\begin{equation}
w_\tau=c_l\tau^{-1}(w-zw_z)-\frac{\hat{H}}{\nu\cdot\hat \nu}=\hat L_{\mC}w+E(w)
\end{equation}
where
\begin{equation}
    \hat L_{\mC}w=c_l\tau^{-1}(w-zw_z)+\Delta_{\mC}w+|A_{\mC}|^2w
\end{equation}
and $E(w)$ is defined in $\eqref{eq error outer}$.  Here $\hat H,\hat\nu$ are the mean curvature and unit outer normal of $\hat F$ respectively.

Since $w$ satisfies \eqref{eq de w vu}, the admissible condition \eqref{eq adm x} is rescaled to
\begin{equation}\label{eq adm z}
    z^{|\gamma|} |\nabla^\gamma_{(z,\theta)}w(z,\theta,\tau)|<\Lambda (z^\alpha+\frac{z^{2\lambda_l+1}}{(2\sigma_l\tau)^{i_1}}),|\gamma|\in\{0,1,2\}
\end{equation}
for $(z,\theta)\in [\beta, \rho (2\sigma_l\tau)^{\frac{1}{2}+\frac{1}{4\sigma_l}}]\times \Sigma$, $\tau_0\leq \tau\leq \tau^\circ$.

By the first admissible condition, in the region $\Gamma_\tau\cap B(O;2\beta^2)$, we can parametrize the $\Gamma_\tau$ as a graph over $S_{\kappa,+}=\{(\tilde z,\theta)|\tilde z\geq \tilde z_0,\theta\in \Sigma\}$ with some profile function $\hat w(\tilde z,\theta,\tau)$ (see the last paragraph of Appendix \ref{sec HS foliation} for the definition of global coordinates $(\tilde z,\theta)$ ($(\tilde r,\theta)$ there) on $S_{\kappa,+}$). That is
\begin{equation}\label{eq hat w def}
    \hat F(\tilde z,\theta,\tau)=S_{\kappa,+}(\tilde z,\theta)+\hat w(\tilde z,\theta,\tau)\nu_{S_{\kappa,+}}(\tilde z,\theta).
\end{equation}
On the other hand, by \eqref{eq mcf} and $\hat F(\tau)=(2\sigma_l\tau)^{c_l}F(-(2\sigma_l\tau)^{-\frac{1}{2\sigma_l}})$, 
we have
\begin{equation}\label{eq rescaled hat flow}
 \hat F_\tau
 =c_l\tau^{-1}\hat F-\hat H\hat\nu.
\end{equation}
Let $\nu_+(\tilde z,\theta)=\nu_{S_{\kappa,+}}(\tilde z,\theta)$, we obtain
\begin{equation}\label{eq rescaled hat graph}
\begin{aligned}
    \hat w_\tau=&c_l\tau^{-1}\frac{(S_{\kappa,+}+\hat w\nu_+)\cdot\hat\nu}{\nu_+\cdot \hat\nu}-\frac{\hat  H}{\nu_+\cdot \hat\nu}=c_l\tau^{-1}\frac{S_{\kappa,+}\cdot\hat\nu}{\nu_+\cdot \hat\nu}+c_l\tau^{-1}\hat w-\frac{\hat  H}{\nu_+\cdot \hat\nu}
\end{aligned}
\end{equation}
for $(\tilde z,\theta)\in S_{\kappa,+}\cap B(O,2\beta^2)$, $\tau_0\leq \tau\leq\tau^\circ$.
\end{itemize}
\section{Constructions of the flow}\label{sec outline}
In this section, we will construct an admissible solution by degree method following \cite{GS}. The method is based on the a prior estimates in Proposition \ref{prop C0 est} and Proposition \ref{prop higher est}, whose proof will be given in Section \ref{sec C0 est}, \ref{sec higher est}, respectively. Assuming Proposition \ref{prop C0 est}, \ref{prop higher est}, we construct the solution in Theorem \ref{thm main aim}. 

The idea to apply the degree method is to show that we can choose a "good" initial hypersurface $\{\Sigma_{t_0}^{\mathbf a}\}$ by choosing a parameter $\mathbf a\in B^{l-1}(O)$ which is close to the origin $O$ for some $l\in\mathbb N_+$ and $0<-t_0\ll 1$ small. Here "good" means that if we evolves $\{\Sigma_{t_0}^{\mathbf a}\}$ by MCF, then the flow (denoted by $\{\Sigma_{t}^{\mathbf a}\}$) exists and is admissible up to $t<0$. To achieve this, for each $t^\circ\in [t_0,0)$, we show that there is a parameter $\mathbf a_{t^\circ}\in B^{l-1}(O)$ close to the origin, such that the flow $\{\Sigma_{t}^{\mathbf a_{t^\circ}}\}$ exists and is admissible up to time $t^\circ$. More importantly, we can derive uniform estimates (Proposition \ref{prop C0 est}, \ref{prop higher est}) for $\{\Sigma_{t}^{\mathbf a_{t^\circ}}\}$, which can be used to take a limit as $t^\circ\to 0$ to get a limit flow $\{\Sigma_t\}$ which exists and is admissible on $[t_0,0)$. Moreover, these estimates also imply that $\{\Sigma_t\}$ converges to $\mC$ in $C^\infty_{loc}(\mC)$ after type I rescaling, and converges to $S_{\kappa,+}$ in $C^{\infty}_{loc}(S_{\kappa,+})$ after type II rescaling for some $\kappa\approx 1$.  

Let's first construct the initial value of the MCF.
\subsection{Initial values}
Let ${\tilde \alpha} ={\tilde \alpha} (\Sigma)>0$ be the constant defined in \eqref{eq psi asy}, for
\begin{equation}\label{eq a def}
    \ba=(a_1,\cdots,a_{l-1)}\in B^{l-1}(0,\beta^{-\tilde \alpha }),
\end{equation} 
we define

(1) The profile function $v(y,\theta,s_0)=v(y,\theta,s_0;\mathbf a)$ over $\mC$ of the type I rescaled hypersurface
\begin{align*}
    \Pi_{s_0}^{\ba}=\frac{1}{\sqrt{|t_0|}}\Si_{t_0}^{\ba}
\end{align*}
is given by
\begin{equation}\label{eq ini int reg}
\begin{aligned}
    v(y,\theta,s_0;\mathbf a)=&e^{-\lambda_ls_0}\Big(\frac{1}{c_l}\varphi_l(y,\theta)+\sum_{k=1}^{l-1}\frac{a_{k}}{c_k}\varphi_{k}(y,\theta)\Big)\\
    =&e^{-\lambda_ls_0}\Big[y^{\alpha}\omega_1\big((1+\sum_{i=0}^{i_1-1}a_{i1})-(K_{1i_11}+\sum_{i=0}^{i_1-1}a_{i1}K_{1i1})y^2+\cdots+(-1)^{i_1}K_{i_1i_11}y^{2i_1}\big)\\
    &+\sum_{j=2}^my^{\alpha_j}\omega_2\big(\sum_{i=0}^{i_j}a_{ij}-\sum_{i=0}^{i_j}a_{ij}K_{1ij}y^2+\cdots+(-1)^{i_j}a_{i_jj}K_{i_ji_jj}y^{2i_j}\big)\Big]
\end{aligned}
\end{equation}
for $\frac{1}{2}\beta e^{-\sigma_l s_0}\leq y\leq 2\rho e^{\frac{s_0}{2}}$. Here we use $a_k$ with a single sub-index $k$ to indicate $a_k$ corresponds to $\varphi_k$, and two sub-index $i,j$ to indicate that $a_{ij}$ corresponds to the eigenfucntion $\varphi_{ij}$ (see the end of Section \ref{sec pre} for the two different ways to index the eigenfunctions of $L_\mC$).

(2) The function $u(x,\theta,t_0)=u(x,\theta,t_0;\mathbf a)$ of $\Si_{t_0}^{\ba}$ is chosen to be
\begin{align*}
    u(x,\theta,t_0)\approx (-1)^{i_1}K_{i_1i_11}\rho^{2\lambda_l+1}\omega_1,
\end{align*}
for $x\gtrsim\rho$, $\theta\in \Sigma$, such that
\begin{equation}\label{eq ini out noncom reg}
    \begin{aligned}
        x^{-1}|u(x,\theta,t_0)|,|\nabla u(x,\theta,t_0)|\leq \varepsilon_0(\mC),\\
        |\nabla^2u(x,\theta,t_0)|\leq C(n,\rho)
    \end{aligned}
\end{equation}
for $(x,\theta)\in[\frac{1}{6}\rho,\infty)\times \Sigma$, for some $\varepsilon_0(\mC)>0$ small.

(3) The part of the hypersurface $|t_0|^{-\lambda_l-\frac{1}{2}}\Sigma_{t_0}$ in $B(O,2\beta^2)$ is a graph of the function $\hat w(\tilde z,\theta)$ over $S_{\kappa,+}$ with $\kappa\approx1$, and is trapped between in two re-scaled surfaces of Simon's foliation $S_{\kappa_1,+}$ and $S_{\kappa_2,+}$ with 
\begin{equation}\label{eq ini lam1 lam2}
\kappa_1=1-\beta^{-\frac{{\tilde \alpha} }{2}}<\kappa<\kappa_2=1+\beta^{-\frac{{\tilde \alpha} }{2}}.
\end{equation}
Moreover, we need
\begin{equation}\label{eq ini tip reg}
\begin{aligned}
|\nabla \hat w(\tilde z,\theta,\tau_0)|\leq C(n,\Sigma,\Lambda)\beta^{-\frac{{\tilde \alpha} }{2}},\\
|\nabla^2 \hat w(\tilde z,\theta,\tau_0)|\leq  C(n,\Sigma,\Lambda)\beta^{-\frac{{\tilde \alpha} }{2}},
\end{aligned}
\end{equation}
for $(\tilde z,\theta)\in S_{\kappa,+}\cap B(O,2\beta^2)$. 

Next, we prove that the initial value defined in \eqref{eq ini int reg} and \eqref{eq ini out noncom reg} are compatible at the intersection points. Moreover, we can take $\kappa_1<\kappa_2$ as in \eqref{eq ini lam1 lam2} so that \eqref{eq ini int reg} holds.
\begin{lemma}
    \eqref{eq ini out noncom reg} holds for $t=t_0$, by choosing $l(\rho)$ large.
\end{lemma}
\begin{proof}
Note $t=-e^{-s}$, $y=\frac{x}{\sqrt{|t|}}$, $\lambda_l+\frac{1}{2}-\frac{1}{2}\alpha_k=\lambda_{i_11}+\frac{1}{2}-\frac{1}{2}\alpha_k=\frac{1}{2}(\alpha-\alpha_k)+i_1$, $k\geq 1$. Define $\bar x_0=\frac{x}{\sqrt{|t_0|}}$, and using \eqref{eq vu} \eqref{eq ini int reg} is equivalent to 
\begin{equation}\label{eq u t0 eq}
    \begin{aligned}
        u(x,\theta,t_0)
        =&\omega_1\big[(1+\sum_{i=0}^{i_1-1}a_{i1})|t_0|^{i_1}x^{\alpha}-(K_{1i_11}+\sum_{i=0}^{i_1-1}a_{i1}K_{1i1})|t_0|^{i_1-1}x^{\alpha+2}+\cdots+(-1)^{i_1}K_{i_1i_11}x^{2\lambda_l+1}\big]\\
        &+\sum_{j=2}^m\omega_j\big[(\sum_{i=0}^{i_j}a_{ij})|t_0|^{\frac{1}{2}(\alpha-\alpha_j)+i_1}x^{\alpha_j}-\sum_{i=0}^{i_j}a_{ij}K_{1ij}|t_0|^{\frac{1}{2}(\alpha-\alpha_j)+i_1-1}x^{\alpha_j+2}\\
        &+\cdots+(-1)^{i_j}a_{i_jj}K_{i_ji_jj}|t_0|^{\lambda_{i_11}-\lambda_{i_jj}}x^{2\lambda_{i_jj}+1}\big].
            \end{aligned}
\end{equation}
Also, we can write
        \begin{align*}
        u(x,\theta,t_0)=&\omega_1 x^{2\lambda_l+1}((1+\sum_{i=0}^{i_1-1}a_{i1}){\bar x_0}^{-2i_1}-(K_{1i_11}+\sum_{i=0}^{i_1-1}a_{i1}K_{1i1}){\bar x_0}^{-2(i_1-1)}+\cdots+(-1)^{i_1}K_{i_1i_11})\\
        &+x^{2\lambda_l+1}\{\sum_{j=2}^m\omega_j[(\sum_{i=0}^{i_j}a_{ij}){\bar x_0}^{-(\alpha-\alpha_j)-2i_1}-\sum_{i=0}^{i_j}a_{ij}K_{1ij}{\bar x_0}^{-(\alpha-\alpha_j)-2i_1+2}\\
        &+\cdots+(-1)^{i_j}a_{i_jj}K_{i_ji_jj}{\bar x_0}^{-2\lambda_{i_11}+2\lambda_{i_jj}}]\}
    \end{align*}
for $\frac{1}{2}\beta|t_0|^{\frac{1}{2}+\sigma_l}\leq x\leq 2\rho$. Since $\omega_1$ is a positive smooth function on $\Sigma$ and $\Si$ is compact, $\omega_1$ has a positive lower bound on $\Si$, thus
\begin{align*}
    x^{|\gamma|}|\nabla^\gamma_{(x,\theta)}u(x,\theta,t_0)|\leq C(n,\Sigma)(|t_0|^{i_1}x^{\alpha}+x^{2\lambda_l+1}),\quad |\gamma|\in\{0,1,2\}
\end{align*}
for $\frac{1}{2}\beta|t_0|^{\frac{1}{2}+\sigma_l}\leq x\leq 2\rho$, and
\begin{equation}
\begin{aligned}
x^{-1}|u(x,\theta,t_0)|\leq C(n,\Sigma)(\beta^{\alpha-1}+\rho^{2\lambda_l})\leq C(n,\Sigma)\rho^{2\lambda_l}
\end{aligned}
\end{equation}
for $x\in [\frac{1}{6}\rho,2\rho]$ if $|t_0|\ll 1$ (depending on $\beta,\rho)$. Thus, \eqref{eq ini out noncom reg} can be achieved by choosing $\rho$ small (depending on $\varepsilon_0(\mC)$) and $|t_0|\ll1$ small (depending on $\rho,\beta$). 
\end{proof}

For the tip region, we have
\begin{lemma}
    We can construct the tip of $\Gamma_{\tau_0}$, s.t. $\Gamma_{\tau_0}\cap B_{(2\sigma_l\tau_0)^{\frac{1-\vartheta}{2}}} $ is trapped between $S_{\kappa_1,+}$ and $S_{\kappa_2,+}$, with $\kappa_1,\kappa_2$ satisfying \eqref{eq ini lam1 lam2}, and intersect smoothly with the intermediate region.
\end{lemma}
\begin{proof}
By \eqref{eq de w vu}, and the fact $z=e^{\sigma_ls}y=(2\sigma_l\tau)^{\frac{1}{2}}y$, let $\bar z_0=\frac{z}{\sqrt{2\sigma_l\tau_0}}$, 
\eqref{eq ini int reg} is equivalent to
    \begin{align*}
         w(z,\theta,\tau_0)
        =&(2\sigma_l\tau_0)^{\frac{\alpha}{2}}(\frac{1}{c_{l}}\varphi_{l}({\bar z_0},\theta)+\sum_{k=1}^{l-1}\frac{a_k}{c_k}\varphi_k({\bar z_0},\theta))
    \end{align*}
for $\frac{1}{2}\beta\leq z\leq 2\rho(2\sigma_l\tau_0)^{\frac{1}{2}+\frac{1}{4\sigma_l}}$, i.e.
    \begin{equation}\label{eq w ini}
    \begin{aligned}
    w(z,\theta,\tau_0)=&z^{\alpha}\Big\{\omega_1\big[(1+\sum_{i=0}^{i_1-1}a_{i1})-(K_{1i_11}+\sum_{i=0}^{i_1-1}a_{i1}K_{1i1}){\bar z_0}^2+\cdots+(-1)^{i_1}K_{i_1i_11}{\bar z_0}^{2i_1}\big]\\
    &+\sum_{j=2}^m{\bar z_0}^{\alpha_j-\alpha}\omega_j\big[\sum_{i=0}^{i_j}a_{ij}-\sum_{i=0}^{i_j}a_{ij}K_{1ij}{\bar z_0}^2+\cdots+(-1)^{i_j}a_{i_jj}K_{i_ji_jj}{\bar z_0}^{2i_j}\big]\Big\}
    \end{aligned}
    \end{equation}
    for $\frac{1}{2}\beta\leq z\leq 2\rho(2\sigma_l\tau_0)^{\frac{1}{2}+\frac{1}{4\sigma_l}}$, 
    which implies
    \begin{equation}\label{eq w ini half int}
    \begin{aligned}
        w(z,\theta,\tau_0)=&z^{\alpha}\omega_1(1+\sum_{i=0}^{i_1-1}a_{i1}+O(|\mathbf a|{\bar z_0}^{2\delta_\alpha})+O({\bar z_0}^{2}))
    \end{aligned}
    \end{equation}
    for $\frac{1}{2}\beta\leq z\leq (2\sigma_l\tau_0)^{\frac{1}{2}}$, where $\delta_\alpha=\delta_\alpha(\Sigma)=\frac{\alpha_2-\alpha}{2}>0$. By 
    \eqref{eq psi asy} and \eqref{eq a def}, we then get
    \begin{align*}
    &|w(z,\theta,\tau_0)-\psi(z,\theta)|\leq |w(z,\theta,\tau_0)-z^{\alpha}\omega_1|+|z^{\alpha}\omega_1-\psi(z,\theta)|\\
    \leq& [|a_{i_1-1,1}|+\cdots|a_{0,1}|+C(n,l,\Sigma)(|\mathbf a|{\bar z_0}^{2\delta_\alpha}+{\bar z_0}^2+z^{-{\tilde \alpha} })]z^{\alpha}\\
    \leq& C(n)\beta^{-{\tilde \alpha} }z^{\alpha}
    \end{align*}
for $\frac{1}{2}\beta\leq z\leq (2\sigma_l\tau_0)^{\frac{1-\vartheta}{2}}$ provided that $\beta\gg 1$ (depending on $n,\Sigma,l$) and $\tau_0\gg 1$ (depending on $n,\Sigma,l,\beta,\vartheta$).
Note also by \eqref{eq psik asym}
\begin{align*}
    \psi_{1\pm \beta^{-\frac{{\tilde \alpha} }{2}}}(z,\theta)-\psi(z,\theta)=(\pm\beta^{-\frac{{\tilde \alpha} }{2}}+O(z^{-{\tilde \alpha} }))z^{\alpha}.
\end{align*}
for $z\geq (1+\beta^{-\frac{\tilde \alpha}{1}})^{\frac{1}{1-\alpha}}R_s$. Consequently, we get
\begin{align*}
    \psi_{1-\beta^{-\frac{{\tilde \alpha} }{2}}}(z,\theta)<w(z,\theta,\tau_0)<\psi_{1+\beta^{-\frac{{\tilde \alpha} }{2}}}(z,\theta)
\end{align*}
for $\frac{1}{2}\beta\leq z\leq(2\sigma_l\tau_0)^{\frac{1-\vartheta}{2}}$ provided that $\beta\gg 1$ (depending on $n,\Sigma,R_s,l$) and $\tau_0\gg 1$ (depending on $n,\Sigma,l,\beta$). Thus we can choose the tip of $\Gamma_{\tau_0}$, s.t. $\Gamma_{\tau_0}\cap B_{(2\sigma_l\tau_0)^{\frac{1-\vartheta}{2}}} $ is trapped between $S_{\kappa_1,+}$ and $S_{\kappa_2,+}$, with $\kappa_1,\kappa_2$ satisfying \eqref{eq ini lam1 lam2}.
\end{proof}
\subsection{Degree method}
In this subsection, we describe the degree method to construct the admissible solution $\{\Sigma_t\}$. Let's define the domain of the map first.
\begin{definition}
 Define $\mathcal O\subset B^{l-1}(O,\beta^{-\tilde\alpha})\times[t_0,0)$ as follows: $({\mathbf a},t^\circ)\in \mathcal O$ iff

 (1) the corresponding smooth MCF $\{\Sigma_t^{\mathbf a}\}$ exists for $t_0\leq t\leq t^\circ$ and can be extended beyond $t^\circ$;

 (2) $\{\Sigma_t^{\mathbf a}\}$ is admissible for $t_0\leq t\leq t^\circ$.
\end{definition}
For $t_0\leq t^\circ<0$, let $\mathcal O_{t^\circ}=\{{\mathbf a}\in B^{l-1}(O,\beta^{-\tilde\alpha})|({\mathbf a},t^\circ)\in \mathcal O)\}$, then $\mathcal O_{t^\circ}$ is an open set of $B^{l-1}(O,\beta^{-\tilde\alpha})$, and is decreasing in $t^\circ$, and $\mathcal O_{t_0}=B^{l-1}(O,\beta^{-\tilde\alpha})$.

Recall that when $\{\Sigma_t\}_{t_0\leq t\leq t^\circ}$ is admissible, we have profile function $v(y,\theta,s)$ for the type I rescaled flow $\{\Pi_s\}_{s_0\leq s\leq s^\circ}$ defined in \eqref{eq vu} for $(y,\theta)\in [\beta e^{-\sigma_ls}, \rho e^{\frac{s}{2}}]\times \Sigma$, $s_0\leq s\leq s^\circ$. In the following, we will cut off $v$ to define the degree map. Let \begin{equation}
    \eta(x)=\begin{cases}
        0,\quad x\leq 0,\\
        1,\quad x\geq 1;
    \end{cases}
\end{equation}
be a smooth cut-off function, and 
\begin{equation}\label{eq def tilde v}
    \tilde v(y,\theta,s;\mathbf a)=\eta(e^{\sigma_ls}y-\beta)\eta(\rho e^{\frac{s}{2}}-y)v(y,\theta,s;\mathbf a)
\end{equation}
with $v(y,\theta,s_0,\mathbf a)$ defined in \eqref{eq ini int reg}.

Define
\begin{equation}
    \Phi({\mathbf a},t)=e^{\lambda_ls}(\langle c_1\tilde v(y,\theta,s;{\mathbf a}),\varphi_{1}\rangle,\cdots,\langle c_{l-1}\tilde v(y,\theta,s;{\mathbf a}),\varphi_{l-1}\rangle)_{s=-\ln|t|}
\end{equation}
For $t_0\leq t<0$, we also define $\Phi_t({\mathbf a})=\Phi(t,{\mathbf a})$, ${\mathbf a}\in \mathcal O_t$. 
\begin{lemma}\label{lem diff id}
    If $s_0\gg 1$ (depending on $n,\Sigma,l,\rho,\beta$), then
     \begin{align*}
        &|\langle\eta(e^{\sigma_ls}y-\beta)\eta(\rho e^{\frac{s}{2}}-y)\varphi_{ij},\varphi_{kp}\rangle_{C,W}-\delta_{ik}\delta_{jp}|
        \leq C(n,\Sigma,l)e^{-(n+2\alpha_j)\sigma_ls},\\
        &\|(1-\eta(e^{\sigma_ls}y-\beta)\eta(\rho e^{\frac{s}{2}}-y))\varphi_{ij}\|_{W}
        \leq C(n,\Sigma,l,\beta)e^{-\frac{1}{2}(n+2\alpha_j)\sigma_ls},
    \end{align*}
for $\lambda_{ij},\lambda_{kj}\leq \lambda_l$, $p\geq 1$.
\end{lemma}
\begin{proof}
    Note that $\langle \varphi_{ij},\varphi_{kp}\rangle_{W}=
    \delta_{ik}\delta_{jp}$. Thus, from the definition of $\eta$, we get
    \begin{align*}
       & |\langle\eta(e^{\sigma_ls}y-\beta)\eta(\rho e^{\frac{s}{2}}-y)\varphi_{ij},\varphi_{kp}\rangle_{W}-\delta_{ik}\delta_{jp}|
        =|\langle \big(1-\eta(e^{\sigma_ls}y-\beta)\eta(\rho e^{\frac{s}{2}}-y)\big)\varphi_{ij},\varphi_{kp}\rangle_{W}|\\
        =&|\int_0^\infty\big(1-\eta(e^{\sigma_ls}y-\beta)\eta(\rho e^{\frac{s}{2}}-y)\big)\tilde\varphi_{ij}\tilde\varphi_{kp}y^{n-1}e^{-\frac{|y|^2}{4}}\int_\Sigma\omega_j\omega_pd\theta dy |\leq (I_1+I_2),
        \end{align*}
        where
        \begin{align*}
       I_1:=\int_0^{(\beta+1)e^{-\sigma_ls}}|\tilde\varphi_{ij}||\tilde\varphi_{kp}|y^{n-1}e^{-\frac{|y|^2}{4}}dy,\quad I_2:=\int_{\rho e^{\frac{s}{2}}-1}^\infty |\tilde\varphi_{ij}||\tilde\varphi_{kp}|y^{n-1}e^{-\frac{|y|^2}{4}}dy.
       \end{align*}
       By \eqref{eq tilde phiij},
       \begin{align*}
           I_1\leq C(n,\Sigma,l,\beta) \int_0^{(\beta+1)e^{-\sigma_ls}}y^{2\alpha_j}y^{n-1}dy\leq  C (n,\Sigma,l,\beta) e^{-(2\alpha_j+n)\sigma_ls}\\
           I_2\leq C(n,\Sigma,l,\beta)\int_{\rho e^{\frac{s}{2}}-1}^\infty y^{2\lambda_{ij}+2\lambda_{kj}+2}y^{n-1}e^{-\frac{|y|^2}{4}}dy
        \leq C (n,\Sigma,l,\beta) e^{-(2\alpha_j+n)\sigma_ls}.
       \end{align*}
    These two inequalities imply the first inequality. Similarly,
    \begin{align*}
        &\|\big (1-\eta(e^{\sigma_ls}y-\beta)\eta(\rho e^{\frac{s}{2}}-y)\big)\varphi_{ij}\|^2_{W}
        \leq \langle \big(1-\eta(e^{\sigma_ls}y-\beta)\eta(\rho e^{\frac{s}{2}}-y)\big)\varphi_{ij}, \varphi_{ij}\rangle_{W}
        \leq C(n,\Sigma,l,\beta)e^{-(2\alpha_j+n)\sigma_ls}.
    \end{align*}
\end{proof}
By Lemma \ref{lem diff id}, $\Phi_{t_0}$ converges uniformly to the identity map on $B^{l-1}(O,\beta^{-\tilde\alpha})$ as $t_0\nearrow 0$. Thus if $|t_0|\ll 1$, we have $\Phi_{t_0}^{-1}(0)\subset\subset B^{l-1}(O,\beta^{-\tilde\alpha})$ and the topological degree
\begin{align*}
    deg(\Phi_{t_0},\mathcal O_{t_0},0)=deg(\Phi_{t_0},B^{l-1}(0,\beta^{-\tilde\alpha}),0)=deg(id,B^{l-1}(O,\beta^{-\tilde\alpha}),0)=1.
\end{align*}
We consider the set
\begin{align*}
    \mathcal I=\{t\in [t_0,0)|deg(\Phi_t,\mathcal O_t,0)=1\}.
\end{align*}
When $(\mathbf a,t)\in O$ and $\Phi_{t_1}(\mathbf a)=0$, and \eqref{eq para} holds, we have the a prior estimates which are important for the extension of the solution. 
\begin{proposition}\label{prop C0 est}
Let $n\geq 7$ be an large integer,  $\mC\subset \mathbb R^{n+1}$ be a regular  minimizing, strictly stable hypercone with isolated singularity. Let $\mathcal L_{\mC}$ be the Jacobi operator, $S_+$ be the Hardt-Simon foliation. If $\mathcal L_{\mC}$ has eigenvalue $\lambda_l$ satisfies the condition \eqref{eq lambda l} for some integer $l$, and the constant $\alpha,\tilde\alpha$ in \eqref{eq psi asy} satisfies
\begin{equation}\label{eq alpha ine}
  \frac{-1-\alpha}{1-\alpha}<\min\{\frac{2(1-\alpha)}{n+2\alpha+4},\frac{n-4+2\alpha}{n+4+2\alpha},\frac{2(1-\alpha)\delta_l}{(n+2\alpha+4)\lambda_l},\frac{\tilde\alpha}{1+\tilde\alpha}\}  
\end{equation}
there exists $\xi=\xi(n)>0$, $\vartheta=\vartheta(n)\in (0,1)$ so that
 \begin{equation}\label{eq xi}
        0<\xi<\min\{1,\frac{n-4+2\alpha}{2(1-\alpha)},\frac{\delta_l}{\lambda_l}\}
    \end{equation}
\begin{equation}\label{eq theta def}
    \frac{-1-\alpha}{1-\alpha}<\vartheta< \frac{1}{2}(1-\alpha)\min\{\frac{4\xi}{n+2\alpha},\frac{(1-\vartheta){\tilde \alpha} }{1-\alpha}\}.
\end{equation}
Assume that $\mathbf a\in \bar {\mathcal O}_{t_1}$ for which 
\begin{equation}\label{eq assum 0}
    \Phi_{t_1}(\mathbf a)=0,
\end{equation}
where $t_1\in [t_0,0)$. Suppose that $\mathbf a\in\bar  {\mathcal O}_{t^\circ}$ for some $t^\circ\in [t_1,e^{-1}t_1]$. Then if $\Lambda\gg1$ (depending on $n,\Sigma$), $0<\rho\ll 1\ll \beta$ (depending on $n,\Lambda$) and $|t_0|\ll 1$ (depending on $n,\Sigma,\Lambda,\rho,\beta$), we have the following estimates

1. In the outer region, the function $u(z,\theta,t)$ of the hypersurface $\Sigma_t^{\mathbf a}$ defined in \eqref{eq Sigma t gra eq} satisfies 
\begin{equation}\label{eq u C2 sum}
    \begin{cases}
        |u(x,\theta,t)|\leq \frac{1}{3}\min\{x,1\},\\
        |\nabla_{(x,\theta)} u(x,\theta,t)|\leq \frac{1}{3},\\
        |\nabla_{(x,\theta)}^2(x,\theta,t)|\leq C(n,\Sigma,\rho);
    \end{cases}
\end{equation}
for $(x,\theta)\in [\frac{1}{3}\rho,\infty)\times \Sigma$, $t_0\leq t\leq t ^\circ$ and
\begin{equation}\label{eq Lambda equ}
    x^{|\gamma|}|\nabla ^\gamma u(x,\theta,t)|\leq \frac{\Lambda }{2}(|t|^{i_1}x^\alpha+x^{2\lambda_l+1}), \quad |\gamma|\in\{0,1,2\}
\end{equation}
for $(x,\theta)\in [\beta|t|^{\frac{1}{2}+\sigma_l}, \rho]\times \Sigma,t_0\leq t\leq t^\circ$.

2. In the intermediate region, if we do the type I rescaling, the function $v(y,\theta,s)$ of the rescaled hypersurface $\Pi_{s}^{\mathbf a}$ defined in \eqref{eq v eq} satisfies
\begin{equation}\label{eq v C0 sum1}
    |v(y,\theta,s)-\frac{\kappa}{c_l}e^{-\lambda_ls}\varphi_l(y,\theta)|\leq C(n,\Sigma,l,\Lambda,\beta,\rho,R)e^{-(1+\tilde k)\lambda_ls}y^{\alpha}\min\{1,y^2\}
\end{equation}
for $(y,\theta)\in [\frac{1}{2}e^{-\vartheta\sigma_ls}, 2R]\times \Sigma$, $s_0\leq s\leq s^\circ$; and
\begin{equation}\label{eq v C0 sum2}
|v(y,\theta,s)-e^{-\sigma_ls}\psi_{\kappa}(e^{\sigma_ls}y,\theta)|\leq C(n,\Sigma,l,\Lambda,\beta,\rho,R)\beta^{-\frac{{\tilde \alpha} }{4}}e^{-2\varrho\sigma_l(s-s_0)}e^{-\lambda_ls}y^\alpha    
\end{equation}
for $(y,\theta)\in [\beta e^{-\sigma_ls},e^{-\vartheta\sigma_ls}]\times \Sigma$, $s_0\leq s\leq s^\circ$. Here, $s^\circ=-\ln|t^\circ|$, $R\gg 1$ is large, $\tilde k>0$ is defined in \eqref{eq tilde k def}, and $\varrho>0$ is defined in \eqref{eq varrho def}.

3. In the tip region, if we perform the type II rescaling, the function of the rescaled hypersurface $\Gamma_\tau$ on $S_{\kappa,+}$ defined in \eqref{eq hat w def} satisfies
\begin{equation}\label{eq hatw C2 sum}
   \begin{cases}
        | \hat w|(\tilde z,\theta,\tau)\leq  C(n,\Sigma)\psi_{\beta^{-{\frac{{\tilde \alpha} }{4}}}}\\
        |\nabla \hat w(\tilde z,\theta,\tau)|\leq C(n,\Sigma,l,\beta),\\
        |\nabla^2 \hat w(\tilde z,\theta,\tau)|\leq C(n,\Sigma,l,\beta);
    \end{cases}
\end{equation}
for $(\tilde z,\theta)\in S_{\kappa,+,3\beta}$, and $\tau_0\leq \tau\leq \tau^\circ$, where $\tau^\circ=(2\sigma_l)^{-1}|t^\circ|^{-2\sigma_l}$. Here $S_{\kappa,+,R}:=B(O,R)\cap S_{\kappa,+,R}$ for $R>0$.
\end{proposition}
\begin{remark}\label{rem nonemp}
If $n\gg1 $, and $\mC$ is the Simons' cone, we have (see \cite{GS})
\begin{align*}
    \alpha\approx -1-\frac{2}{n+1},\quad \frac{-1-\alpha}{1-\alpha}\approx\frac{1}{n+2},\quad 
    \tilde\alpha=2-2\alpha\approx 4+\frac{4}{n+1}.
\end{align*}
We can take 
\begin{equation}
    \vartheta\approx \frac{1}{n+2}, \quad\delta_l\approx\frac{2\lambda_l}{n}
\end{equation}
So, at least for Simons' cone $\mC$ and large $n$, condition \eqref{eq alpha ine} is satisfied.
\end{remark}
Moreover, we have the following asymptotic and smooth estimates.
\begin{proposition}\label{prop higher est}
    Under the hypothesis of Proposition \ref{prop C0 est}, there is 
    \begin{equation}
        \kappa\in (1-C(n,\Sigma,\Lambda,l,\rho,\beta)|t_0|^{\xi\lambda_l},1+C(n,\Sigma,\Lambda,l,\rho,\beta)|t_0|^{\xi\lambda_l})
    \end{equation}
    so that for any given $0<\delta\ll 1$, $m,q\in \mathbb N_+$, the following estimates holds.

    1. In the outer region, the function $u$ of $\Sigma_t$ defined in \eqref{eq Sigma t gra eq} satisfies 
    \begin{equation}\label{eq u sum1}
        |\nabla_{(x,\theta)}^{m}\nabla_t^q u(x,\theta,t)|\leq C(n,\Sigma,l,\rho,\delta,m,q)
    \end{equation}
    for $(x,\theta)\in [\frac{1}{2}\rho,\infty)\times\Sigma$, $t_0+\delta^2\leq t\leq t^\circ$, and 
    \begin{equation}\label{eq u sum2}
        x^{m+2q}\left |\nabla_{(x,\theta)}^{m}\nabla_t^q \left(u(x,\theta,t)-\frac{\kappa}{c_l}|t|^{\lambda_l+\frac{1}{2}}\varphi_l(\bar x,\theta)\right)\right|\leq C(n,\Sigma,l,\Lambda,\delta,m,q)\rho^{4\lambda_l}x^{2\lambda_l+1}
    \end{equation}
    for $(x,\theta)\in[R\sqrt{|t|},\frac{3}{4}\rho]\times\Sigma$, $t_0+\delta^2x^2\leq t\leq t^\circ$.

    2. In the intermediate region, we rescale the hypersurface by the Type I rescaling, then the function $v$ of the rescaled hypersurface $\Pi_s$ defined in \eqref{eq v eq}  satisfies 
    \begin{equation}\label{eq v sum1}
        y^{m+2q}\left|\nabla_{(y,\theta)}^m\nabla_t^q\left(v(y,\theta,s)-\frac{\kappa}{c_l}e^{-\lambda_ls}\varphi_l(y,\theta)\right)\right|\leq C(n,\Sigma,l,\Lambda,\delta,m,q)e^{-(1+\xi)\lambda_ls}y^{\alpha}\min\{1,y^2\}
    \end{equation}
    for $(y,\theta)\in [\frac{3}{4}e^{-\vartheta\sigma_l s},\frac{3}{2}R]\times \Sigma$, $s_0+\delta^2 y^2\leq s\leq s^\circ$, and
    \begin{equation}\label{eq v sum2}
        y^{m+2q}|\nabla_{(y,\theta)}^m\nabla_s^q\left(v(y,\theta,s)-e^{-\sigma_ls}\psi_{\kappa}(e^{\sigma_ls}y,\theta)\right)|\leq C(n,\Sigma,l,\Lambda,\delta,m,q)\beta^{\alpha-{\tilde \alpha} }e^{-2\varrho\sigma_l(s-s_0)}e^{-\lambda_ls}y^\alpha
    \end{equation}
    for $(y,\theta)\in [\frac{3}{2}\beta e^{-\sigma_ls},\frac{4}{5}e^{-\vartheta\sigma_l s}]$, $s_0+\delta^2y^2\leq s\leq s^\circ$, where $\xi$ satisfies \eqref{eq xi}, and
    \begin{equation}\label{eq varrho def}
        \varrho=1-\frac{1}{2}(1-\alpha)(1-\vartheta)\in (0,\vartheta)
    \end{equation}
    are positive constants.

    3. In the tip region, if we rescale the hypersurface hy Type II rescaling, then the function $\hat w(\tilde z,\theta,\tau)$ for the rescaled hypersurface $\Gamma_\tau$ over $S_{\kappa,+}$ defined in \eqref{eq hat w def} satisfies,
    \begin{equation}\label{eq hat w sum}
        \delta^{m+2q}|\nabla_{(\tilde z,\theta))}^m\nabla^q_\tau\hat w(\tilde z,\theta,\tau)|\leq C(n,\Sigma,m,q)\beta^{-\frac{{\tilde \alpha} }{4}}(\frac{\tau}{\tau_0})^{-\varrho}
    \end{equation}
    for $(\tilde z,\theta)\in S_{\kappa,+,2\beta}$, $\tau_0+\delta^2\leq \tau\leq \tau^\circ$.
\end{proposition}
Proposition \ref{prop C0 est}, \ref{prop higher est} will be proved in Section \ref{sec C0 est}, \ref{sec higher est}, respectively. After we prove Proposition \ref{prop C0 est}, \ref{prop higher est}, we can apply the results in \cite{EH} to conclude that $\mathbf a\in e^{-1}t_1$. Moreover, we can prove
\begin{corollary}
    If $|t_0|\ll1$ (depending on $n,\Sigma$), then we have $\mathcal I=[t_0,0)$.
\end{corollary}
\begin{proof}
   The proof is the same as that of Corollary 4.7 in \cite{GS}, except that we apply Proposition 
    \ref{prop C0 est}, \ref{prop higher est} here. 
\end{proof}
\begin{theorem}\label{thm main aim}
   Under the assumption of Proposition \eqref{prop C0 est}, if $|t_0|\ll 1$ (depending on $n,\Sigma$), there is an admissible MCF $\{\Sigma_t\}_{t_0\leq t<0}$ for which the function $u$ defined in \eqref{eq Sigma t gra eq} satisfies \eqref{eq u C2 sum}. Moreover, in the tip region, if we do the type II rescaling, the rescaled function $\hat w$ over $S_{\kappa,+}$ (defined in \eqref{eq hat w def})  satisfies \eqref{eq hatw C2 sum}, with 
    \begin{align*}
        \kappa\in (1-C(n,\Sigma|t_0|^{\xi\lambda_l},1+C(n,\Sigma)|t_0|^{\xi\lambda_l}).
    \end{align*}
In addition, for any given $0<\delta\ll 1,m,q\in\mathbb N_+$, there hold

1. In the outer region, the function $u$ of $\Sigma_t$ defined in \eqref{eq Sigma t gra eq} satisfies \eqref{eq u sum1} and \eqref{eq u sum2}.

2. In the intermediate region, if we do the type I rescaling, the function $v$ of $\Pi_s$ defined in \eqref{eq v eq} satisfies \eqref{eq v sum1} and \eqref{eq v sum2}.

3. In the tip region, if we do the type II rescaling, the function $\hat w$ of $\Gamma_\tau$  defined in \eqref{eq hat w def} satisfies \eqref{eq hat w sum}. 
\end{theorem}
\begin{proof}
   The proof is the same as the proof of Theorem 4.8 in \cite{GS}, so we omit it.
\end{proof}
\begin{proof}[Proof of Theorem \ref{thm main aim intro}]
Let $\{\Sigma_t\}_{t_0\leq t<0}$ be the solution in Theorem \ref{thm main aim}. From \eqref{eq v eq}, \eqref{eq u sum2}, and \eqref{eq v sum1}, the type I rescaled hypersurface $\Pi_s$ converges to $\mathcal C$ locally smoothly, i.e. for any $0<r<R<\infty$
\begin{align*}
    \Pi_s\to \mathcal C\quad \text{in }C^\infty(B(O,R)\setminus B(O,r)).
\end{align*}
as $s\to\infty$. Likewise, from \eqref{eq hat w def}, \eqref{eq v sum2}, \eqref{eq hat w sum}, \eqref{eq psik asym}, the type II rescaled hypersurfaces $\Gamma_\tau$ converges to $S_{\kappa,+}$ locally smoothly, i.e.
\begin{align*}
    \Gamma_\tau\to S_{\kappa,+} \text{ in }C_{loc}^\infty (S_{\kappa,+}).
\end{align*}
as $\tau\to\infty$.
\end{proof}
\section{$C^0$ estimates}\label{sec C0 est}
In this section, we will prove Proposition \ref{prop C0 est} under the assumption \eqref{eq assum 0}. More precisely, we will show that if $0<\rho\ll 1\ll\beta$ (depending on $n,\Sigma,l,\Lambda$), and $|t_0|\ll 1$ (depending on $n,\Sigma,l,\Lambda,\rho,\beta$) there holds 
\begin{equation}
    |\mathbf a|\leq C(n,\Sigma,\lambda,l,\rho,\beta)|t_0|^{\xi\lambda_l}
\end{equation}
where $\xi>0$ is the constant defined in \eqref{eq xi}. Moreover, there exsits 
\begin{equation}
    \kappa\in (1-C(n,\Sigma,l,\Lambda,\rho,\beta)|t_0|^{\xi\lambda_l},1+C(n,\Sigma,l,\Lambda,\rho,\beta)|t_0|^{\xi\lambda_l})
\end{equation}
so that Proposition \ref{prop C0 est} holds.

The idea is to first prove the estimate for the type I rescaled flow $\Pi_s^{\mathbf a}$ in the intermediate region by the assumption \eqref{eq assum 0}, and the constructions of $\Pi_{s_0}^{\mathbf a}$ in \eqref{eq ini int reg}. Then we use the maximum principle and barrier arguments together the estimate at the boundary of intermediate region and initial condition to prove the $C^0$ estimates in the compact outer region and tip region. Finally, we use Ecker-Huisken's estimates \cite{EH} to extend the estimates in the compact outer region and initial condition to noncompact outer region. These four estimates are given in Proposition \ref{pro inter C0}, \ref{pro outer C0}, \ref{prop tip C0}, \ref{prop C2 estim outer noncom first}  respectively. We consider the intermediate region first.
\subsection{Intermediate region} 
We are going to prove the following estimate in the intermediate region. 
\begin{proposition}\label{pro inter C0}
Assume that \eqref{eq assum 0} holds, if $0<\rho\ll1$, $\beta\gg 1$ (depending on $n,\Sigma,l,\Lambda$), and $s_0\gg 1$ (depending on $n,\Sigma,l,\rho,\beta$) then there exists $\kappa\in\mathbb R$, $\vartheta\in (0,1)$, any $R\geq 1$, and all $s_0\leq s\leq s^\circ$ satisfying 
    \begin{equation}\label{eq inter C0}
        |v(y,\theta,s)-\frac{\kappa}{c_l}e^{-\lambda_ls}\varphi_l(y,\theta)|\leq C(n,\Sigma,l,\Lambda,\beta,\rho,R)e^{-(1+\tilde k)\lambda_ls}y^{\alpha}\min\{1,y^2\}
    \end{equation}
for $\frac{1}{2}e^{-\vartheta\sigma_ls}\leq y\leq 2R$, 
\begin{equation}\label{eq tilde k def}
    \tilde k=\xi-\vartheta \frac{\frac{n}{2}+\alpha+2}{1-\alpha}>0
\end{equation}
if $\vartheta$ satisfies \eqref{eq theta def1}.
\begin{remark}
     Note \eqref{eq inter C0} is exactly \eqref{eq v C0 sum1}.
\end{remark}
\end{proposition}
In the following of this subsection, $C=C(n,\Sigma,l,\Lambda,\beta,\rho)$ is a constant depending on $n,\Sigma,l,\Lambda,\beta,\rho$ if there is no other clarifications. Recall the definition of $\tilde v$ in \eqref{eq def tilde v}. By \eqref{eq v eq}, 
\begin{equation}\label{eq tilde v eq}
    \tilde v_s-L_{\mC}\tilde v=[\partial_s-\partial_{yy}-\frac{n-1}{y}\partial y-\frac{\Delta_\Sigma+|A_\Sigma|^2}{y^2}-(\frac{1}{2}-\frac{1}{2}y\cdot\partial_y)]\tilde v=f=:f_1+f_2+f_3,
\end{equation}
where 
\begin{equation}\label{eq f1 f2 f3}
\begin{aligned}
    &f_1(y,\theta,s)=\eta(e^{\sigma_ls}y-\beta)\eta(\rho e^{\frac{s}{2}}-y)E(v)\\
    &f_2(y,\theta,s)=\big\{\eta'(e^{\sigma_ls}y-\beta)e^{\sigma_ls}[(\sigma_l+\frac{1}{2})y-\frac{2v_y}{v}-\frac{n-1}{y}]-\eta''(e^{\sigma_ls}y-\beta)e^{2\sigma_ls}\big\}\eta(\rho e^{\frac{s}{2}}-y)v\\
    &f_3(y,\theta,s)=\big\{\eta'(\rho e^{\frac{s}{2}}-y)[e^{\frac{s}{2}}\frac{\rho}{2}+\frac{2v_y}{v}+\frac{n-1}{y}-\frac{1}{2}y]-\eta''(\rho e^{\frac{s}{2}}-y)\big\}\eta(e^{\sigma_ls}y-\beta)v.
\end{aligned}
\end{equation}
are smooth compactly supported functions. Note that $(\beta+1) e^{-\sigma_ls}<\rho e^{\frac{s}{2}}-1$ for $s$ sufficiently large, thus $\eta'(e^{\sigma_ls}y-\beta)\eta'(\rho e^{\frac{s}{2}}-y)\equiv0$ for $s$ sufficiently large. We have the following estimate for $f_i$ by \eqref{eq adm y},
\begin{align}
&|f_1(y,\theta,s)|\leq |E(v)|\chi_{(\beta e^{-\sigma_ls},\rho e^{\frac{s}{2}})}   \\ 
&|f_2(y,\theta,s)|\leq  C[(1+\frac{1}{y^2})|v|+\frac{|\nabla v|}{y}]\chi_{(\beta e^{-\sigma_ls},(\beta+1) e^{-\sigma_ls})}\leq C e^{-\lambda_ls}y^{\alpha-2}\chi _{(\beta e^{-\sigma_ls},(\beta+1) e^{-\sigma_ls})}\label{eq boun f2}\\
&|f_3(y,\theta,s)|\leq C[(|y|+1)|v|+|\nabla v|]\chi_{(\rho e^{\frac{s}{2}}-1,\rho e^{\frac{s}{2}})}\leq Ce^{-\lambda_ls}y^{2\lambda_l+2}\chi_{(\rho e^{\frac{s}{2}}-1,\rho e^{\frac{s}{2}})}\label{eq boun f3}
\end{align}
We estimate $|f_1|$ as follows. Recall $E(v)$ is defined in \eqref{eq error outer}. Since $v$ satisfies \eqref{eq adm y}, and $\sigma_l=\frac{\lambda_l}{1-\alpha}$, $v$ satisfies \eqref{eq basic assum}, with 
\begin{equation}
\mu=\Lambda e^{-\lambda_ls}(y^{\alpha-1}+y^{2\lambda_l})\text{ for }\beta e^{-\sigma_ls}\leq y\leq\rho e^{\frac{s}{2}}, s_0\leq s\leq s^\circ.
\end{equation} 
Thus $|E(v)|\leq C\frac{\mu^2}{y}$ by \eqref{eq E stru 2}, and 
\begin{equation}\label{eq boun f1}
    |f_1(y,\theta,s)|\leq C e^{-2\lambda_ls}(y^{2\alpha-3}+y^{4\lambda_l-1})\chi_{(\beta e^{-\sigma_ls},\rho e^{\frac{s}{2}})}.
\end{equation}
Next, we compute the $L_W^2$ norm of $f_i$ ($i=1,2,3$). By \eqref{eq boun f1}, we have
\begin{align*}
    \|f_1\|_{W}(s)=&\big(\int_0^\infty \int_\Sigma |f_1(\cdot,\cdot,s)|^2y^{n-1}e^{-\frac{y^2}{4}}d\theta dy\big)^{\frac{1}{2}}
    \leq C|\Sigma|^{\frac{1}{2}}e^{-2\lambda_ls} (\int_{\beta e^{-\sigma_ls}}^{\rho e^{\frac{s}{2}}}(y^{4\alpha-6}+y^{8\lambda_l-2})y^{n-1}e^{-\frac{y^2}{4}}dy)^{\frac{1}{2}}\\
    \leq &C e^{-2\lambda_ls}((\beta e^{-\frac{\lambda_ls}{1-\alpha} })^{\frac{(n+4\alpha-6)}{2}}+C)\leq C e^{-\lambda_l(1+\xi)s}
\end{align*}
since $\xi<\min\{1+\frac{n+4\alpha-6}{2(1-\alpha)},1\}=\frac{n+2\alpha-4}{2(1-\alpha)}$, and $ s_0(n,\Sigma,l,\Lambda,\beta,\rho)\gg1$ large. Similarly, by \eqref{eq boun f2} and \eqref{eq boun f3}, we obtain
\begin{align*}
    \|f_2\|_{W}(s)
    \leq& C|\Sigma|^{\frac{1}{2}}e^{-\lambda_ls} (\int_{\beta e^{-\sigma_ls}}^{(\beta+1) e^{-\sigma_ls}}y^{2\alpha-4}y^{n-1}e^{-\frac{y^2}{4}}dy)^{\frac{1}{2}}
    \leq C e^{-\lambda_ls}(\beta e^{-\frac{\lambda_ls}{1-\alpha} })^{\frac{(n+2\alpha-4)}{2}}
    \leq Ce^{-\lambda_l(1+\xi)s}
\end{align*}
since $\xi<\frac{n+2\alpha-4}{2(1-\alpha)}$, and $s_0(n,\Sigma,l,\Lambda,\beta)\gg1$ large; and
\begin{align*}
    \|f_3\|_{W}(s)\leq C|\Sigma|^{\frac{1}{2}}e^{-\lambda_ls} (\int_{\rho e^{\frac{s}{2}}-1}^{\infty}y^{4\lambda_l+4}y^{n-1}e^{-\frac{y^2}{4}}dy)^{\frac{1}{2}}
    \leq C e^{-\lambda_ls}e^{-\frac{(\rho e^{\frac{s}{2}}-1)^2}{8}}\leq C e^{-\lambda_l(1+\xi)s}
\end{align*}
since $\xi<1$, and $s_0(n,\Sigma,l,\Lambda,\beta,\rho)\gg 1$ large.
Thus if we take $\xi>0$ as in \eqref{eq xi},
then we get
\begin{equation}\label{eq boun f norm}
    \|f\|_{W}\leq \sum_{i=1}^3\|f_i\|_{W}\leq C e^{-\lambda_l(1+\xi)s}.
\end{equation}
for $ s_0(n,\Sigma,l,\Lambda,\beta,\rho)\gg1 $ large.

Next, we estimate the lower frequency components.
\begin{lemma}
Under the assumption of Proposition \ref{pro inter C0}, we have
\begin{equation}
\begin{aligned}
     |a_{ij}|
     \leq C e^{-\xi\lambda_ls_0},
\end{aligned}
\end{equation}
for $\lambda_{ij}<\lambda_l$, $s_0\leq s\leq s^\circ$; and
\begin{equation}
\begin{aligned}
     &|\kappa-1|
    \leq
    Ce^{-\xi\lambda_ls_0}
\end{aligned}
\end{equation}
for $s_0\leq s\leq s^\circ$, where
\begin{equation}\label{eq k def}
    \kappa:=c_{i_11}e^{\lambda_ls_1}\langle\tilde v(\cdot,s_1),\varphi_{1i_1}\rangle.
\end{equation}
\end{lemma}
\begin{proof}
For simplicity, we omit $W$ in $\langle\cdot,\cdot,\rangle_W$ and $\|\cdot\|_W$. First, we note that we have $s^\circ\leq s_1+1$ since $t^\circ\leq e^{-1}t_1$. We have for $\lambda_{ij}< \lambda_l$ (i.e. $i<i_1$ for $j=1$, or $i\leq i_j$ for $1<j\leq m$),
\begin{equation}
    \begin{cases}
    \partial_s\langle \tilde v,\varphi_{ij}\rangle+\lambda_{ij}\langle\tilde v,\varphi_{ij}\rangle=\langle  f,\varphi_{ij}\rangle,\\
        \langle \tilde v,\varphi_{ij}\rangle(s_1)=0.
    \end{cases}   
\end{equation}
This together with \eqref{eq boun f norm} implies that for $s_1\leq s\leq s^\circ$,
\begin{align*}
    &|\langle\tilde v,\varphi_{ij}\rangle (s)|\leq |\int_{s_1}^{s}e^{\lambda_{ij}(s'-s)}\langle  f,\tilde\varphi_{ij}\rangle ds'|\leq \int_{s_1}^s e^{\lambda_{ij}(s'-s)}\|f\|\|\varphi_{ij}\|ds'
    \leq Ce^{|\lambda_{ij}|}\int_{s_1}^{s} e^{-(1+\xi)\lambda_ls'}ds'\leq Ce^{-(1+\xi)\lambda_ls}
\end{align*}
since $|s^\circ -s_1|\leq 1$; and for $s_0\leq s\leq s_1$,
\begin{align*}
    &|\langle\tilde v,\varphi_{ij}\rangle (s)|\leq |\int_{s}^{s_1}e^{\lambda_{ij}(s'-s)}\langle f,\varphi_{ij}\rangle(s') ds'|\leq Ce^{-\lambda_{ij}s}\int_{s}^{s_1}e^{\lambda_{ij}s'} e^{-(1+\xi)\lambda_ls'}ds'
    \leq Ce^{-(1+\xi)\lambda_ls}
\end{align*}
since  $(1+\xi)\lambda_l-\lambda_{ij}>\lambda_l-\lambda_{ij}>0$. Thus for $\lambda_{ij}<\lambda_l$, there holds 
\begin{equation}\label{eq boun tilde small proj}
\begin{aligned}
|\langle\tilde v,\varphi_{ij}\rangle (s)|\leq  C(n,\Sigma,l)e^{-(1+\xi)\lambda_ls}.
\end{aligned}
\end{equation}
On the other hand, by Lemma \ref{lem diff id}, we have (note $\tilde v(\cdot,s)= \tilde v(\cdot,s;\mathbf a)=\eta(e^{\sigma_ls}y-\beta)\eta(\rho e^{\frac{s}{2}}-y)v(\cdot,s;\mathbf a)$, with $\ba=(a_1,\cdots,a_{l-1)}\in B^{l-1}(0,\beta^{-\tilde \alpha})$ defined in \eqref{eq a def})
\begin{align*}
    &|e^{\lambda_ls_0}\langle\tilde v(\cdot,s_0),c_{ij}\varphi_{ij}\rangle-a_{ij}|=|\langle \eta(e^{\sigma_ls}y-\beta)\eta(\rho e^{\frac{s}{2}}-y)(\sum_{k\leq i_j}\frac{a_{kj}}{c_{kj}}\varphi_{kj}),c_{ij}\varphi_{ij}\rangle-a_{ij}|
    \leq C e^{\frac{-(n+2\alpha_j)}{1-\alpha}\lambda_ls_0}\leq Ce^{-2\xi\lambda_ls_0}
\end{align*}
(and this also holds for $a_{i_11}=a_l=1$). Thus, for $\lambda_{ij}<\lambda_l$
\begin{equation}\label{eq aij less norm}
\begin{aligned}
     |a_{ij}|\leq& |e^{\lambda_ls_0}\langle\tilde v(\cdot,s_0),c_{ij}\varphi_{ij}\rangle-a_{ij}|+ |e^{\lambda_ls_0}\langle\tilde v,\varphi_{ij}\rangle (s_0)|\leq C(n,\Sigma,l,\Lambda,\beta,\rho) e^{-\xi\lambda_ls_0}
\end{aligned}
\end{equation}
for $s_0\leq s\leq s^\circ$.

Now, for $i=i_1$, $\lambda_{i_11}=\lambda_l$, by Lemma \ref{lem diff id}, we have
\begin{equation}
    \begin{cases}
    \partial_s\langle \tilde v,\varphi_{i_11}\rangle+\lambda_{i_11}\langle\tilde v,\varphi_{i_11}\rangle=\langle  f,\varphi_{i_11}\rangle,\\
        |e^{\lambda_{i_11}s_0}\langle \tilde v,\varphi_{i_11}\rangle(s_0) -\frac{1}{c_{i_11}}|\leq Ce^{-(2\alpha+n)\frac{\lambda_l}{1-\alpha}s_0}.
    \end{cases}   
\end{equation}
Then for $s_1\leq s\leq s^\circ$, by \eqref{eq k def}, \eqref{eq boun f norm}, we have
\begin{align*}
    &|e^{\lambda_{l}s}\langle\tilde v,\varphi_{i_11}\rangle (s)-\frac{\kappa}{c_{i_11}}|\leq |\int_{s_1}^{s}e^{\lambda_{i_1i}s'}\langle  f,\varphi_{i_11}\rangle(s') ds'|\leq C\int_{s_1}^{s_1+1} e^{\lambda_{i_11}s'} e^{-\lambda_l(1+\xi)s'}ds'\leq  Ce^{-\xi\lambda_{i_11}s}
\end{align*}
since $s_1\leq s\leq s^\circ\leq s_1+1$; and for $s_0\leq s\leq s_1$, 
\begin{align*}
    &|e^{\lambda_{l}s}\langle\tilde v,\varphi_{i_11}\rangle (s)-\frac{\kappa}{c_{i_11}}|\leq |\int_{s}^{s_1}e^{\lambda_{i_1i}s'}\langle  f,\varphi_{i_11}\rangle(s') ds'|\leq C\int_{s_1}^{s_1+1} e^{\lambda_{i_11}s'} e^{-\lambda_l(1+\xi)s'}ds'\leq  Ce^{-\xi\lambda_{i_11}s}.
\end{align*}
Thus we get,
\begin{equation}\label{eq boun tilde l proj}
\begin{aligned}
    &|e^{\lambda_{l}s}\langle\tilde v,\varphi_{i_11}\rangle (s)-\frac{\kappa}{c_{i_11}}|\leq
    Ce^{-\xi\lambda_{i_11}s},
\end{aligned}
\end{equation}
and
\begin{equation}\label{eq bound k diff}
\begin{aligned}
     &|\kappa-1|\leq c_{i_11}(|e^{\lambda_{i_11}s_0}\langle \tilde v,\varphi_{i_11}\rangle(s_0) -\frac{1}{c_{i_11}}|+|e^{\lambda_{l}s_0}\langle\tilde v,\varphi_{i_11}\rangle (s_0)-\frac{\kappa}{c_{i_11}}|)\leq Ce^{-\xi\lambda_ls_0}
\end{aligned}
\end{equation}
for $s_0\leq s\leq s^\circ$ since $0<\xi\leq \frac{n+2\alpha}{2(1-\alpha)}$.
\end{proof}
\begin{lemma}\label{lem H*}
    Let $\mathbf H_*$ be the closed subspace of $\mathbf H$ sapnned by eigenfucntions $\{\varphi_{ij}\}_{i\geq i_j+1}$ ($i_j=0$ if $j>m$) of $L_{\mC}$. Given 
    \begin{align*}
        f(\cdot,s)\in L^2([s_0,s^\circ];L^2_W)
    \end{align*}
    and $h\in \mathbf H_*$, let $v(\cdot,s)\in C([s_0,s^\circ];\mathbf H_*)$ be the weak solution of 
    \begin{equation}
\begin{cases}
    (\partial_s-L_{\mC})v=f,\\
    v(\cdot,0)=h.
    \end{cases}
  \end{equation}
  Then for any $\delta\in(0,1)$, we have
  \begin{align*}
      \|v\|_W^2(s)\leq e^{-2(1-\delta)\lambda_{l+1}(s-s_0)}\|v\|_W^2(s_0)+\frac{1}{2\delta\lambda_{l+1}}\int_{s_0}^se^{-2(1-\delta)\lambda_{l+1}(s-s')}\|f\|_W^2(s')ds',
  \end{align*}
  and
  \begin{align*}
      \langle -L_{\mC}v,v\rangle_W(s)\leq  e^{-2(1-\delta)\lambda_{l+1}(s-s_0)}\langle - L_{\mC}h,h\rangle_W(s_0)+\frac{1}{2\delta}\int_{s_0}^se^{-2(1-\delta)\lambda_{l+1}(s-s')}\|f\|_W^2(s')ds'.
  \end{align*}
\end{lemma}
\begin{proof}
    The same as the proof of Lemma 6.2 in \cite{GS}.
\end{proof}

Now we estimate the high frequency parts in the Fourier expansion of $\tilde v$. Let
\begin{align*}
    \tilde v_*=\tilde v-\sum_{j=1}^m\sum_{i=0}^{i_j}\langle \tilde v,\varphi_{ij}\rangle_W\varphi_{ij},
\end{align*}
then $\tilde v_*\in C([s_0,s_1];\mathbf H_*)$, where $\mathbf H_*$ is defined in Lemma \ref{lem H*}. By \eqref{eq tilde v eq}, we have
\begin{align*}
    (\partial_s-L_{\mC})\tilde v_*=f-\sum_{j=1}^m\sum_{i=0}^{i_j}\langle f,\varphi_{ij}\rangle_W\varphi_{ij}=f_*
\end{align*}
Note that, $\|f_*\|_W\leq \|f\|_W$, $\lambda_{l+1}\geq\lambda_{i_11}+\delta_l$ by assumptions on $l$. By Lemma \ref{lem H*}, we have
\begin{align*}
      \|\tilde v_*\|_W^2(s)\leq e^{-2(1-\delta)\lambda_l(s-s_0)}\|\tilde v_*\|_W^2(s_0)+\frac{1}{2\delta\lambda_{l+1}}\int_{s_0}^se^{-2(1-\delta)\lambda_{l+1}(s-s')}\|f\|_W^2(s')ds'
  \end{align*}
  and
  \begin{align*}
      \langle -L_{\mC}\tilde v_*,\tilde v_*\rangle_W(s)\leq  e^{-2(1-\delta)\lambda_l(s-s_0)}\langle - L_{\mC}\tilde v_*,\tilde v_*\rangle_W(s_0)+\frac{1}{2\delta}\int_{s_0}^se^{-2(1-\delta)\lambda_{l+1}(s-s')}\|f\|_W^2(s')ds'
  \end{align*}
  for $\delta<\delta_0$, $s_0\leq s\leq s^\circ$. We claim that
  \begin{equation}\label{eq claim}
      \|\tilde v_*\|_W(s_0)+\|- L_{\mC}\tilde v_*\|_W(s_0)\leq Ce^{-(1+\xi)\lambda_ls_0}.
  \end{equation}
  Assume the claim is true. Note that since
  \begin{equation}
      \xi<\frac{\delta_l}{\lambda_l},
  \end{equation}
we have $(1-\delta)(\lambda_{l}+\delta_l)>(1+\xi)\lambda_l$ for $\delta\in(0,1)$ small. Thus, we get
\begin{align*}
     \|\tilde v_*\|_W^2(s)+\langle -L_{\mC}\tilde v_*,\tilde v_*\rangle_W(s)\leq Ce^{-2(1+\xi)\lambda_ls},
\end{align*}
which by Lemma \ref{lem cocer}, yields
\begin{align*}
    \|\tilde v_*\|_W^2(s)+\|\nabla_{\mC}\tilde v_*\|_W^2(s)\leq Ce^{-2(1+\xi)\lambda_ls}.
\end{align*}
By Lemma \ref{lem Morrey}, we then get
\begin{equation}\label{eq tilde v*}
    |\tilde v_*(y,\theta,s)|\leq C(y^{-\frac{n}{2}}+e^{\frac{(y+1)^2}{4}})(\|\nabla \tilde v_*\|_W+\|\tilde v_*\|_W)\\
    \leq Ce^{-(1+\xi)\lambda_ls}(y^{-\frac{n}{2}}+e^{\frac{(y+1)^2}{4}})
\end{equation}
for $s_0\leq s\leq s^\circ$. To prove \eqref{eq claim}, we use Lemma \ref{lem diff id}, \eqref{eq aij less norm}, \eqref{eq bound k diff}
\begin{align*}
    &\|\tilde v_*\|_W(s_0)=\|\tilde v(\cdot,s_0)-\sum_{j=1}^m\sum_{i=0}^{i_j}\langle \tilde v,\varphi_{ij}\rangle_W(s_0)\varphi_{ij}\|_W\\
    \leq &\|\tilde v(\cdot,s_0)-e^{-\lambda_ls_0}\sum_{j=1}^m\sum_{i=0}^{i_j}\frac{a_{ij}}{c_{ij}}\varphi_{ij}\|_W+\|e^{-\lambda_ls_0}\sum_{j=1}^m\sum_{i=0}^{i_j}\frac{a_{ij}}{c_{ij}}\varphi_{ij}-\sum_{j=1}^m\sum_{i=0}^{i_j}\langle \tilde v,\varphi_{ij}\rangle_W(s_0)\varphi_{ij}\|_W\\
    \leq &e^{-\lambda_ls_0}\|\left(1-\eta(e^{\sigma_ls_0}-\beta)\eta(\rho e^{\frac{s_0}{2}}-y)\right)\sum_{j=1}^m\sum_{i=0}^{i_j}\frac{a_{ij}}{c_{ij}}\varphi_{ij}\|_W+\sum_{j=1}^m\sum_{i=0}^{i_j}\frac{1}{c_{ij}}|\langle \tilde v,c_{ij}\varphi_{ij}\rangle_W(s_0)-a_{ij}e^{-\lambda_ls_0}|\\
    \leq& e^{-(1+\xi)\lambda_ls_0}
\end{align*}
where $a_{i_11}=1$, and 
\begin{align*}
   & \|L_{\mC}\tilde v_*\|_W(s_0)=\|L_{\mC}\left(\eta(e^{\sigma_ls_0}-\beta)\eta(\rho e^{\frac{s_0}{2}}-y)v(\cdot,s_0)\right)+\sum_{j=1}^m\sum_{i=0}^{i_j}\langle \tilde v,\varphi_{ij}\rangle_W(s_0)\lambda_{ij}\varphi_{ij}\|_W\\
    =&\|L_{\mC}\left(\eta(e^{\sigma_ls_0}-\beta)\eta(\rho e^{\frac{s_0}{2}}-y)\sum_{j=1}^m\sum_{i=0}^{i_j}\frac{a_{ij}}{c_{ij}}\lambda_{ij}\varphi_{ij}\right)+\sum_{j=1}^m\sum_{i=0}^{i_j}\langle \tilde v,\varphi_{ij}\rangle_W(s_0)\lambda_{ij}\varphi_{ij}\|_W\\
    \leq &\|\left(1-\eta(e^{\sigma_ls_0}-\beta)\eta(\rho e^{\frac{s_0}{2}}-y)e^{-\lambda_ls_0}\right)e^{-\lambda_ls_0}\sum_{j=1}^m\sum_{i=0}^{i_j}\frac{a_{ij}}{c_{ij}}\lambda_{ij}\varphi_{ij}\|_W+\|h\|_W\\
    &+\|\sum_{j=1}^m\sum_{i=0}^{i_j}\langle \tilde v,\varphi_{ij}\rangle_W(s_0)\lambda_{ij}\varphi_{ij}-e^{-\lambda_ls_0}\sum_{j=1}^m\sum_{i=0}^{i_j}\frac{a_{ij}}{c_{ij}}\lambda_{ij}\varphi_{ij}\|_W\\
\leq &Ce^{-(1+\xi)\lambda_ls_0}+\|\tilde h\|_W
\end{align*}
where 
\begin{align*}
    \tilde h=&\big\{\eta'(e^{\sigma_ls_0}y-\beta)e^{\sigma_ls_0}[\frac{1}{2}y-\frac{2v_y}{v}-\frac{n-1}{y}]-\eta''(e^{\sigma_ls_0}y-\beta)e^{2\sigma_ls_0}\big\}\eta(\rho e^{\frac{s_0}{2}}-y)v\\
   &+\big\{\eta'(\rho e^{\frac{s_0}{2}}-y)[\frac{2v_y}{v}+\frac{n-1}{y}-\frac{1}{2}y]-\eta''(\rho e^{\frac{s_0}{2}}-y)\big\}\eta(e^{\sigma_ls_0}y-\beta)v.
\end{align*}
By a similar computations as for $f_2,f_3$, we obtain
\begin{align*}
    \|\tilde h\|_W\leq Ce^{-(1+\xi)\lambda_ls_0}.
\end{align*}
Hence,
\begin{align*}
    \|L_{\mC}\tilde v_*\|_W(s_0)\leq Ce^{-(1+\xi)\lambda_ls_0}.
\end{align*}
The claim is true.

Finally, combining \eqref{eq boun tilde small proj}, \eqref{eq boun tilde l proj},\eqref{eq tilde v*}, we conclude
\begin{align*}
    &|\tilde v(y,\theta,s)-\frac{\kappa}{c_{i_11}}e^{-\lambda_{i_11}s}\varphi_{i_11}(y,\theta)|=|\sum_{j=1}^m\sum_{i=0}^{i_j}\langle \tilde v,\varphi_{ij}\rangle_W\varphi_{ij}+\tilde v_*(y,\theta,s)-\frac{\kappa}{c_{i_11}}e^{-\lambda_{i_11}s}\varphi_{i_11}(y,\theta)|\\
    \leq &|(\sum_{i=0}^{i_1-1}\langle \tilde v,\varphi_{i1}\rangle_W\varphi_{i1}+\sum_{j=2}^m\sum_{i=0}^{i_j}\langle \tilde v,\varphi_{ij}\rangle_W\varphi_{ij}|+|\langle \tilde v,\varphi_{i_11}\rangle_W\varphi_{i_11}-\frac{\kappa}{c_{i_11}}e^{-\lambda_{i_11}s}\varphi_{i_11}(y,\theta)|+|\tilde v_*(y,\theta,s)|\\
    \leq& Ce^{-(1+\xi)\lambda_ls}(y^{-\frac{n}{2}}+e^{\frac{(y+1)^2}{4}})
\end{align*}
for $s_0\leq s\leq s^\circ$. As a result, for $\frac{1}{2}e^{-\vartheta \sigma_ls}\leq y\leq 1$, we have
\begin{align*}
    |\tilde v(y,\theta,s)-\frac{\kappa}{c_{i_11}}e^{-\lambda_{i_11}s}\varphi_{i_11}(y,\theta)|\leq C\frac{e^{-\xi\lambda_ls}}{y^{\frac{n}{2}+\alpha+2}}e^{-\lambda_ls}y^{\alpha+2}\leq Ce^{-(\xi\lambda_l-\vartheta\sigma_l(\frac{n}{2}+\alpha+2))s}e^{-\lambda_ls}y^{\alpha+2}.
\end{align*}
Note $\xi\lambda_l-\vartheta\sigma_l(\frac{n}{2}+\alpha+2)=(\xi-\vartheta \frac{\frac{n}{2}+\alpha+2}{1-\alpha})\lambda_l>0$ if
\begin{equation}\label{eq theta def1}
    \vartheta<\frac{2(1-\alpha)\xi}{n+2\alpha+4}.
\end{equation} 

For $1\leq y\leq 2R$, 
we have
\begin{align*}
     |\tilde v(y,\theta,s)-\frac{\kappa}{c_{i_11}}e^{-\lambda_{i_11}s}\varphi_{i_11}(y,\theta)|\leq C (e^{-\xi\lambda_ls}e^{\frac{(y+1)^2}{4}})e^{-\lambda_ls}y^{\alpha}\leq C(n,\Sigma,l,\Lambda,\beta,\rho,R) e^{-(1+\xi)\lambda_ls}y^{\alpha}.
\end{align*}
\subsection{The compact outer region}
 We use the estimate in Proposition \ref{pro inter C0} as boundary to prove the $C^0$ estimates in the compact outer region in this subsection. 
\begin{proposition}\label{pro outer C0}
    If $0<\rho\ll1$ (depending on $n,\Sigma,l,\Lambda$), $R\gg 1$ (depending on $n,\Sigma,l,\Lambda,\beta$) and $|t_0|\ll 1$ (depending on $n,\Sigma,l,\Lambda,\rho,\beta,R$), then 
    \begin{equation}\label{eq outer C0}
        |u(x,\theta,t)-\frac{\kappa}{c_l}|t|^{-\lambda_l+\frac{1}{2}}\varphi_l(\bar x,\theta)|\leq C(n,\Sigma,l)R^{-2}x^{2\lambda_l+1}\omega_1(\theta)
    \end{equation}
for $(x,\theta,t)\in  \Omega:=\{(x,\theta,t):2R\sqrt{|t|}\leq x\leq \rho,\theta\in \Sigma, t_0\leq t\leq t^\circ \}$.
\end{proposition}
\begin{proof}
   In the following of the proof, $C=C(n,\Sigma,l,\Lambda,\beta,\rho)$ denote a positive constant depending on $n,\Sigma,l,\Lambda,\beta,\rho$ which may change from line to line there is no other illustration. We prove it by constructing sub and supersolutions. The proof follows \cite{Liu}. First note, by \eqref{eq flow outer}, we have
    \begin{equation}\label{eq outer}
        \partial_tu=\mathcal L_{\mC}u+E(u)=[\partial_{xx}+\frac{n-1}{x}\partial_x +\frac{\Delta_\Sigma+|A_{\Sigma}|^2}{x^2}]u+E(u),
    \end{equation}
    where $E(u)$ is given in \eqref{eq error outer} with $X=\mC$ there. Then \eqref{eq E stru 2} and \eqref{eq adm x} yield 
      \begin{equation}
          |E(u)|\leq Cx^{-1}(|t|^{i_1}x^{\alpha-1}+x^{2\lambda_l})^2\leq Cx^{-1}(x^{2\lambda_l})^2=Cx^{4\lambda_l-1},\quad 
      \end{equation}
     in $\Omega$ if $\rho\ll 1\ll R$. On the hand, since $\omega_1>0$ on $\Sigma$ which is compact, $\min_{\Sigma}\omega_1(\theta)\geq \varepsilon_1(\Sigma)>0$. This together with \eqref{eq adm x} implies that
      \begin{equation}
      \begin{aligned}
          |\partial_tu|\leq& C\big(|u_{xx}|+\frac{|u_x|}{x}+\frac{|\nabla^2_\Sigma u|+|u|}{x^2}\big)+E(u)
          \leq Cx^{-1}\big(|t|^{i_1}x^{\alpha-1}+x^{2\lambda_l}\big)+C(n,\Lambda)x^{4\lambda_l-1}
          \leq Cx^{2\lambda_l-1}\omega_1,
      \end{aligned}
      \end{equation}
     in $\Omega$, if $\rho\ll 1\ll R$. Moreover, let $\bar x:=\frac{x}{\sqrt{|t|}}$, by \eqref{eq phiij eq} and \eqref{eq bound k diff}, we obtain
      \begin{align*}
          &|\partial_t\big(\kappa|t|^{\lambda_l+\frac{1}{2}}\varphi_l(\bar x,\theta)\big)|
          \leq C x^{2\lambda_l-1}\omega_1
      \end{align*}
      in $\Omega$, if $\rho\ll 1\ll R$. Thus,
      \begin{equation}\label{eq outer boun speed}
      \begin{aligned}
          &|\partial_t\big(u(x,\theta,t)-\frac{\kappa}{c_l}|t|^{\lambda_l+\frac{1}{2}}\varphi_l(\bar x)\big)|
          \leq C x^{2\lambda_l-1}\omega_1
          \end{aligned}
      \end{equation}
      in $\Omega$, if $\rho\ll 1\ll R$. In particular,
      \begin{equation}\label{eq outer boun speed rho}
      \begin{aligned}
          |\partial_t\big(u(\rho,\theta,t)-\frac{\kappa}{c_l}|t|^{\lambda_l+\frac{1}{2}}\varphi_l(\frac{\rho}{\sqrt{|t|}})\big)|\leq C x^{2\lambda_l-1}
          \leq C\rho^{-2} x^{2\lambda_l+1}\omega_1
        \end{aligned}
      \end{equation}
      in $\Omega$. 
      
      On the other hand, let $\bar x_0:=\frac{x}{\sqrt{|t_0|}}$. Then, by \eqref{eq phiij eq}, \eqref{eq u t0 eq}, \eqref{eq bound k diff}, and \eqref{eq aij less norm}, we have
      \begin{equation}\label{eq outer boun ini}
      \begin{aligned}
          &|u(x,\theta,t_0)-\kappa|t_0|^{\lambda_l+\frac{1}{2}}\varphi_l({\bar x_0})|\leq |t_0|^{\lambda_l+\frac{1}{2}}\big(\frac{|1-\kappa|}{c_l}\varphi_{l}({\bar x_0},\theta)+\sum_{k=0}^{l-1}\frac{|a_k|}{c_k}\varphi_{k}({\bar x_0},\theta)\big)\\
          \leq& C|t_0|^{\xi\lambda_l}|t_0|^{\lambda_l+\frac{1}{2}}\big({\bar x_0}^{2\lambda_l+1}+{\bar x_0}^{\alpha}\big)
           \leq C|t_0|^{\xi\lambda_l}x^{2\lambda_l+1}\omega_1   
      \end{aligned} 
      \end{equation}
for $2R\sqrt{|t_0|}\leq x\leq \rho\ll1,\theta\in\Sigma$. Combining \eqref{eq outer boun speed rho} and \eqref{eq outer boun ini} gives 
\begin{equation}\label{eq outer boun rig bun}
    \begin{aligned}
    &|u(\rho,\theta,t)-\frac{\kappa}{c_l}|t|^{\lambda_l+\frac{1}{2}}\varphi_l(\frac{\rho}{\sqrt{|t|}})|
    \leq C|t_0|^{\xi\lambda_l}x^{2\lambda_l+1}+ C\rho^{-2} x^{2\lambda_l+1}|t_0|
    \leq C|t_0|^{\xi\lambda_l}x^{2\lambda_l+1}\omega_1,
    \end{aligned}
\end{equation}
for $\theta\in\Sigma, t_0\leq t\leq t^\circ$. At last, \eqref{eq inter C0} implies
      \begin{equation}\label{eq outer boun boun}
      \begin{aligned}
         & |u(x,\theta,t)-\frac{\kappa}{c_l}|t|^{\lambda_l+\frac{1}{2}}\varphi_l(\bar x,\theta))|\leq \tilde C|t|^{(1+\tilde k)\lambda_l+\frac{1}{2}}\bar x^\alpha\\
          \leq& \tilde C|t|^{(1+\tilde k)\lambda_l+\frac{1}{2}} \bar x^{2\lambda_l+1} (2R)^{-(2\lambda_l+1-\alpha)}
          \leq \tilde C(2R)^{-2i_1}|t|^{\tilde k\lambda_l}x^{2\lambda_l+1}\omega_1
    \end{aligned}
      \end{equation}
       for $x=2R\sqrt{|t|}$, $\theta\in\Sigma$, $t_0\leq t\leq t^\circ$, where $\tilde C=\tilde C(n,\Sigma,l,\Lambda,\beta,R)$ depends on $n,\Sigma,l,\Lambda,\beta,R$.
       
Combining \eqref{eq outer boun ini}, \eqref{eq outer boun rig bun}, \eqref{eq outer boun boun}, we get 
      \begin{equation}\label{eq u para bound boun}
      \begin{aligned}
          &|u(x,\theta,t)-\frac{\kappa}{c_l}|t|^{\lambda_l+\frac{1}{2}}\varphi_l(\bar x,\theta))|
          \leq \bar C|t_0|^{\tilde k\lambda_l}x^{2\lambda_l+1}\omega_1
          \end{aligned}
      \end{equation}
 on the parabolic boundary $P\Omega$ of $\Omega$, since $\tilde k\leq \xi$, with $\bar C=\bar C(n,\Sigma,l,\Lambda,\beta,R)$ depending on $n,\Sigma,l,\Lambda,\rho,\beta,R$.
     
      To construct super and subsolutions, we need to compare $u$ and $\kappa K_{i_1i_11}x^{2\lambda_l+1}\omega_1$ on $P\Omega$. Using \eqref{eq phiij eq}, we get
      \begin{equation}\label{eq diff phiij high}
      \begin{aligned}
          &|\frac{\kappa}{c_l}|t|^{\lambda_l+\frac{1}{2}}\varphi_l(\bar x,\theta)-(-1)^{i_1}\kappa K_{i_1i_11}x^{2\lambda_l+1}\omega_1(\theta)|
          \leq C(n,l,\Sigma) R^{-2} x^{2\lambda_l+1}\omega_1
          \end{aligned}
      \end{equation}
      for $x\geq 2R\sqrt{|t|},\theta\in\Sigma,t<0$. This and \eqref{eq u para bound boun} yield
      \begin{equation}\label{eq outer u high diff}
          \begin{aligned}
              &|u(x,\theta,t)-(-1)^{i_1}\kappa K_{i_1i_11}x^{2\lambda_l+1}\omega_1|
              \leq \bar C|t_0|^{\tilde k\lambda_l}x^{2\lambda_l+1}+ C(n,l,\Sigma) R^{-2} x^{2\lambda_l+1}\omega_1
              \leq C'(n,\Sigma,l)R^{-2}x^{2\lambda_l+1}\omega_1
          \end{aligned}
      \end{equation}
      on $P\Omega$ if $|t_0|\ll1$ small (depending on $n,\Lambda,\rho,\beta,R$), with $\bar C=\bar C(n,\Sigma,l,\Lambda,\beta,R)$ depending on $n,\Sigma,l,\Lambda,\rho,\beta,R$. Let 
       \begin{equation}
           u^{\pm}(x,t)=C_0^{\pm}(x^{2\lambda_l+1}-C^{\pm}|t|x^{2\lambda_l-1})\omega_1,
       \end{equation}
       where $C_0^{\pm}$, $C^{\pm}$ are constants to be determined. Direct computations show that 
      \begin{align*}
      u^{\pm}_t=&C_0^{\pm}C^{\pm}x^{2\lambda_l-1}\omega_1
      \end{align*}
      and 
      \begin{align*}
           \mathcal L_{\mC}u^{\pm}=& C_0^{\pm}\{M_lx^{2\lambda_l-1}-C^{\pm}|t|M'_lx^{2\lambda_l-3}\}\omega_1.
      \end{align*}
     where $M_l=(2\lambda_l+1)(2\lambda_l+n-1)-\mu_1$, $M'_l=(2\lambda_l-1)(2\lambda_l-2+n-1)-\mu_1$. If $i_1$ is even, we choose $C^+$ such that
          \begin{equation}
              C^+\geq 2M_l>2M'_l,
          \end{equation}
    then
      \begin{align*}
          (\partial_t-\mathcal L_{\mC})u^{+}=&C_0^{+}\{(C^{+}-M_l)x^{2\lambda_l-1}+C^+|t|M'_lx^{2\lambda_l-3}\}\omega_1\geq \frac{1}{2}C_0^+C^+x^{2\lambda_l-1}\omega_1.
          \end{align*}
          Moreover,
          \begin{align*}
          |E(u^{+})|\leq& Cx^{-1}(\frac{u^{+}}{x}+|\nabla u^+|+x|\nabla^2u^+|)^2\leq Cx^{-1}(C_0^+)^2C(\Sigma,l)(x^{2\lambda_l}+C_+R^{-2}x^{2\lambda_l})^2\\
        \leq& 2C(C_0^+)^2C(\Sigma,l)x^{4\lambda_l-1}\leq \frac{1}{4}C_0^+C^+ x^{2\lambda_l-1}\omega_1\qquad  
      \end{align*}
      for $2R\sqrt{|t|}\leq x\leq \rho$, $-1\leq t_0\leq t\leq 0$, if 
      \begin{align*}
         C_+\leq  R^2, C_0^+\leq \frac{\min_{\Sigma}\omega_1}{8CC(\Sigma,l)\rho^{2\lambda_l}}.
      \end{align*}
      In particular, if we take $R>0$ large, and
      \begin{equation}{\label{eq C+ C0+ cond}}
          C^+=2M_l\leq R^2, \,C_0^+\leq \frac{\min_{\Sigma}\omega_1}{8CC(\Sigma,l)\rho^{2\lambda_l}},
      \end{equation}
then we have
\begin{equation}\label{eq com outer sup equ}
     (\partial_t-\mathcal L_{\mC})u^{+}-E(u^+)\geq \frac{1}{2}C_0^+C^+x^{2\lambda_l-1}\omega_1(\theta)-\frac{1}{4}C_0^+C^+ x^{2\lambda_l-1}\omega_1\geq 0.
\end{equation} 
      Similarly,
      \begin{align*}
          (\partial_t-\mathcal L_{\mC})u^{-}=&C_0^{-}\{(C^{-}-M_l)x^{2\lambda_l-1}+C^-|t|M'_lx^{2\lambda_l-3}\}\omega_1,
          \end{align*}
          and
          \begin{align*}
           |E(u^{-})|\leq& Cx^{-1}(\frac{u^{-}}{x}+|\nabla u^-|+x|\nabla^2u^-|)^2\leq Cx^{-1}(C_0^-)^2C(\Sigma,l)(x^{2\lambda_l}+C^-(R)^{-2}x^{2\lambda_l})^2\\
           \leq &2C(C_0^-)^2C(\Sigma,l)x^{4\lambda_l-1}\leq \frac{1}{4}C_0^-M_l x^{2\lambda_l-1}\omega_1
      \end{align*}
      if we take      
      \begin{equation}\label{eq C- C0- cond}
          C^-=0,\, 0<C_0^-\leq \frac{M_l\min_{\Sigma}\omega_1}{8CC(\Sigma,l)\rho^{2\lambda_l}},
      \end{equation}
      Thus, under the assumption \eqref{eq C- C0- cond}, we have
      \begin{equation}\label{eq com outer sub equ}
          (\partial_t-\mathcal L_{\mC})u^{-}-E(u^-)\leq -C_0^-M_lx^{2\lambda_l-1}\omega_1+\frac{1}{4}C_0^-x^{2\lambda_l-1}\omega_1 \leq 0.
      \end{equation}
    Now, we can take 
      \begin{equation}
\begin{aligned}
   & 0<C_0^+=\{\kappa K_{i_1i_11}+C'R^{-2}\}(1-\frac{M_l}{2R^2})^{-1}\leq \frac{\min_{\Sigma}\omega_1}{8CC(\Sigma,l)\rho^{2\lambda_l}},\\
   & 0<C_0^-=(\kappa K_{i_1i_11}-C'R^{-2})\leq \frac{M_l\min_{\Sigma}\omega_1}{8CC(\Sigma,l)\rho^{2\lambda_l}},
\end{aligned}
      \end{equation}
if $\rho>0$ (depending on $n,\Sigma,l$) is small enough, so that $C_0^+,C_0^-$ satisfies \eqref{eq C+ C0+ cond}, \eqref{eq C- C0- cond}, respectively. Therefore \eqref{eq com outer sup equ}, \eqref{eq com outer sub equ} hold. Furthermore, we have by \eqref{eq bound k diff},
      \begin{align*}
          \frac{1}{2}K_{i_1i_11}\leq C_0^-\leq C_0^+\leq 2K_{i_1i_11}
      \end{align*}     
      for $R\gg 1$ (depending on $n,M_1(l),\Lambda,\beta$).
      
On the other hand, \eqref{eq outer u high diff} implies
      \begin{equation}\label{eq com outer sub super bounary}
      \begin{aligned}
          u^+=&C_0^+x^{2\lambda_l+1}(1-2M_l\bar x^{-2})\omega_1\geq C_0^+x^{2\lambda_l+1}(1-\frac{M_l}{2R^2})\omega_1=\{\kappa K_{i_1i_11}+C'R^{-2}\}x^{2\lambda_l+1}\omega_1\geq u\\
          u^-=&(\kappa K_{i_1i_11}-C'R^{-2})x^{2\lambda_l+1}\omega_1\leq u.
        \end{aligned}
      \end{equation}
      on $P\Omega$ if $-\frac{1}{2}\leq t_0$.
      Note \eqref{eq com outer sup equ}, \eqref{eq com outer sub equ}, \eqref{eq com outer sub super bounary} together show that $u^{+},u^-$ are super and subsolution to \eqref{eq outer} respectively. By comparison theorem, we have for all $(x,t)\in\bar\Omega$,
      \begin{align*}
       (\kappa K_{i_1i_11}-C'R^{-2})x^{2\lambda_l+1}\omega_1  
       =u^-\leq u\leq  u^+\leq C_0^+x^{2\lambda_l+1}\omega_1
      \end{align*}
      Note $(1-\frac{M_l}{2R^2})^{-1}\leq 1+\frac{M_l}{R^2}$ when $R\gg1$ (depending on $n,l$), we get 
      \begin{align*}
           (\kappa K_{i_1i_11}-C'R^{-2})x^{2\lambda_l+1}\omega_1\leq u^-\leq u\leq u^+\leq (\kappa K_{i_1i_11}+C'R^{-2})(1+\frac{M_l}{R^2})x^{2\lambda_l+1}\omega_1
      \end{align*}
      This combining with  \eqref{eq diff phiij high} gives
      \begin{equation}
          |u(x,\theta,t)-\frac{\kappa}{c_l}|t|^{\lambda_l+\frac{1}{2}}\varphi_l(\bar x,\theta)|\leq 100(C'(n,l,\Sigma)+M_l)R^{-2}x^{2\lambda_l+1}\omega_1\leq C(n,l,\Sigma)x^{2\lambda_l+1}\omega_1.
      \end{equation}
      When $i_1$ is odd, we let $C^+=0,C^-=2M_l$, then we can get a subsolution $u^+$ and a supersolution $u^-$ and \eqref{eq diff phiij high} and \eqref{eq outer u high diff} remain valid. The final choice of $C_0^{\pm}$ is 
      \begin{align*}
          C_0^+=&(-1)(\kappa K_{i_1i_11}-C'R^{-2}),\\
          C_0^-=&(-1)\{\kappa K_{i_1i_11}+C'R^{-2}\}(1-\frac{M_l}{2R^2})^{-1}\leq \frac{\min_{\Sigma}\omega_1}{8CC(\Sigma,l)\rho^{2\lambda_l}}.
      \end{align*}
 \end{proof}

 \subsection{Inner region}\label{sec inner reg}
 In this subsection, we prove the $C^0$ estimate in the tip region by barrier arguments. Th following proposition is the statement of the result.
\begin{proposition}\label{prop tip C0}
If $\beta\gg1 $ (depending on $n,\Sigma$) and $\tau_0\gg1 $ (depending on $n,\Sigma,l,\Lambda,\rho,\beta$), there holds
\begin{equation} \label{eq tip inner C0 est 0}
    |\hat w(z,\theta,\tau)\leq C(n,\Sigma) \beta^{-\frac{{\tilde \alpha} }{4}}(\frac{\tau}{\tau_0})^{-\varrho}z^\alpha,
\end{equation}
on $S_{k,+,(2\sigma_l\tau)^{\frac{1}{2}(1-\vartheta)}}\setminus S_{k,+,5\beta}$, for $\tau_0\leq\tau\leq \tau^\circ$; and
\begin{equation}\label{eq tip inner C0 est hat}
| \hat w(\tilde z,\theta,\tau)|\leq C(n,\Sigma) \beta^{-\frac{{\tilde \alpha} }{4}}(\frac{\tau}{\tau_0})^{-\varrho}
\end{equation}
on $S_{k,+,5\beta}$,  $\tau_0\leq\tau\leq \tau^\circ$. Here $\hat w$ is the function of $\Gamma_\tau$ as a normal graph over $S_{\kappa,+}$  which is defined in \eqref{eq hat w def}, $S_{\kappa,+,R}:=B(O,R)\cap S_{\kappa,+}$, for $R>0$.
\end{proposition}
 \begin{remark}
\eqref{eq de w vu} and \eqref{eq tip inner C0 est 0} implies \eqref{eq v C0 sum2}; and \eqref{eq tip inner C0 est hat} implies the first equation in \eqref{eq hatw C2 sum}.
\end{remark}
As mentioned before, Proposition \ref{prop tip C0} is proved by a barrier argument. To construct a barrier, we need to know the bound for initial and boundary condition in the tip region. These bounds can be obtained from the estimates in the intermediate region and the initial assumptions. First we consider the boundary conditions. Note that by \eqref{eq de w vu}, \eqref{eq inter C0} implies
\begin{align*}
    |w(z,\theta,\tau)-\frac{\kappa}{c_l}(2\sigma_l\tau)^{\frac{\alpha}{2}}\varphi_l(\bar z,\theta)|\leq C(n,\Sigma,l,\Lambda,\beta,\rho,R)(2\sigma_l\tau)^{-\frac{1}{2}\tilde k(1-\alpha)}\bar z^2
z^\alpha
\end{align*}
for $\frac{1}{2}(2\sigma_l\tau)^{\frac{1-\vartheta}{2}}\leq z\leq (2R)(2\sigma_l\tau)^{\frac{1}{2}}$, $\tau_0\leq \tau\leq \tau^\circ$, where $\bar z:=\frac{z}{(2\sigma_l\tau)^{\frac{1}{2}}}$. On the other hand,
\begin{align*}
    \frac{\kappa}{c_l}(2\sigma_l\tau)^{\frac{\alpha}{2}}\varphi_l(\bar z,\theta)=&\kappa z^{\alpha}\omega_1\big(1+\sum_{m=1}^{i_1}(-1)^mK_{mi1}\bar z^{2m}\big)
\end{align*}
by \eqref{eq phiij eq}. Hence, we get
\begin{equation}
    |w(z,\theta,\tau)-\kappa z^\alpha\omega_1|\leq C(n,\Sigma,l,\Lambda,\beta,\rho,R)\big((2\sigma_l\tau)^{-\frac{1}{2}\tilde k(1-\alpha)}+1\big)\bar z^2 z^{\alpha}\leq C(n,\Sigma,l)\bar z^2z^{\alpha}
\end{equation}
 for $\frac{1}{2}(2\sigma_l\tau)^{\frac{1-\vartheta}{2}}\leq z\leq (2R)(2\sigma_l\tau)^{\frac{1}{2}}$, $\tau_0\leq \tau\leq \tau^\circ$, provided that $\tau_0\gg 1$ (depending on $\Lambda,\rho, \beta,R$). 

On the other hand, by \eqref{eq psik asym} and \eqref{eq bound k diff}, we have 
\begin{align*}
    |\psi_{\kappa}(z,\theta)-\kappa z^{\alpha}\omega_1|\leq C(n,\Sigma,\kappa)z^{\alpha-{\tilde \alpha} }\leq C(n,\Sigma)z^{\alpha-{\tilde \alpha} }.
\end{align*}
for $z\geq R_s$. Therefore, we get
\begin{equation}
    |w(z,\theta,\tau)-\psi_{\kappa}(z,\theta)|\leq |w(z,\theta,\tau)-\kappa z^{\alpha}\omega_1|+|\kappa z^{\alpha}\omega_1-\psi_{\kappa}(z,\theta)|\leq C(n,\Sigma,l)(\bar z^2+z^{-{\tilde \alpha} })z^{\alpha}.
\end{equation}
for $\frac{1}{2}(2\sigma_l\tau)^{\frac{1-\vartheta}{2}}\leq z\leq (2R)(2\sigma_l\tau)^{\frac{1}{2}}$, $\tau_0\leq \tau\leq \tau^\circ$. In particular, there holds
\begin{equation}\label{eq mcf intmed inn con}
    dist(\hat F(z,\theta,\tau),S_{\kappa,+})\leq C(n,\Sigma,l)(2\sigma_l\tau)^{-\vartheta}z^{\alpha}
\end{equation}
for $z=\frac{1}{2}(2\sigma_l\tau)^{\frac{1-\vartheta}{2}}$, $\tau_0\leq \tau\leq \tau^\circ$ if we choose $\vartheta>0$ small such that 
\begin{equation}\label{eq theta def3}
    \vartheta<\frac{1}{2}(1-\vartheta){\tilde \alpha} .
\end{equation}
For the initial condition, by \eqref{eq w ini half int}, \eqref{eq aij less norm}, \eqref{eq bound k diff}, and \eqref{eq psik asym}, we have
\begin{align*}
    &|w(z,\theta,\tau_0)-\psi_{\kappa}(z,\theta)|\leq |w(z,\theta,\tau_0)-\kappa z^{\alpha}\omega_1|+|\kappa z^{\alpha}\omega_1-\psi_{\kappa}(z,\theta)|\\
    \leq &\big(|\kappa-1|+|a_{i_1-1,1}|+\cdots|a_{0,1}|+C(n,\Sigma,l)(|\mathbf a|\bar z_0^{2\delta_\alpha}+\bar z_0^2+z^{-{\tilde \alpha} })\big)z^{\alpha}\omega_1\\
    \leq& \big(C(n,\Sigma,l,\Lambda,\rho,\beta)|t_0|^{\xi\lambda_l}+C(n,\Sigma,l)(\bar z_0^{2\delta_\alpha'}+\beta^{-{\tilde \alpha} })\big)z^{\alpha}\omega_1\\
    \leq &\big(C(n,\Sigma,l,\Lambda,\rho,\beta)(2\sigma_l\tau_0)^{-\frac{\xi(1-\alpha)}{2}}+C(n,\Sigma,l,\beta)((2\sigma_l\tau_0)^{-\vartheta\delta_\alpha'}+\beta^{-{\tilde \alpha} })\big)z^{\alpha}\omega_1\\
    \leq& C(n,\Sigma,l)\beta^{-{\tilde \alpha} }z^{\alpha}\omega_1
\end{align*}
for $\beta\leq z\leq 2(2\sigma_l\tau_0)^{\frac{1-\vartheta}{2}}$, provided $\tau_0\gg1$ (depending on $n,\Sigma,l,\Lambda,\rho,\beta$), where $\delta_\alpha'=\min\{\delta_\alpha,1\}>0$ ($\delta_\alpha=\frac{\alpha_2-\alpha}{2}>0$), $\bar z_0=\frac{z}{(2\sigma_l\tau_0)^{\frac{1}{2}}}$. That is
\begin{equation}\label{eq w ini z large}
    dist(\hat F(z,\theta,\tau_0),S_{\kappa,+})\leq C(n,\Sigma,l)\beta^{-{\tilde \alpha} }z^{\alpha}\omega_1
\end{equation}
for $\beta\leq z\leq 2(2\sigma_l\tau_0)^{\frac{1-\vartheta}{2}}$, provided $\tau_0\gg1$ (depending on $n,\Sigma,l,\Lambda,\rho,\beta$).
\subsubsection{Choice of $\vartheta,l$}
At this moment, let's discuss the value of $\vartheta$. We need 
\begin{equation}\label{eq theta def4}
    0<\frac{-1-\alpha}{1-\alpha}<\vartheta
\end{equation}
to define $\varrho$ in \eqref{eq varrho def}.
On the other hand, by \eqref{eq theta def1}, and 
\eqref{eq theta def3}, we have
\begin{align*}
    0<\vartheta<&\min \{\frac{2(1-\alpha)\xi}{n+2\alpha+4},\frac{1}{2}(1-\vartheta){\tilde \alpha} \}
    = \frac{1}{2}(1-\alpha)\min\{\frac{4\xi}{n+2\alpha+4},\frac{(1-\vartheta){\tilde \alpha} }{1-\alpha}\}
\end{align*}
with $\xi$ defined in \eqref{eq xi}. For the set of $\vartheta$ to be non-empty, what we need is exactly \eqref{eq alpha ine}. 
\subsubsection{Sub solutions}
Now we construct sub and supersolutions to control $\Gamma_\tau=\hat F$ in the tip region. Note by $\hat F(z,\theta,\tau)$ satisfies \eqref{eq rescaled hat flow}.
To get a lower barrier, we set
\begin{equation}\label{eq Fhat - def}
    \hat F^-=\lambda_-^{\frac{1}{1-\alpha}}(\tau)S_{\kappa,+}=S_{\lambda_-(\tau)\kappa,+},
\end{equation}
with 
\begin{equation}\label{eq lambda - def}
    \lambda_-(\tau)=1-\beta^{-\frac{{\tilde \alpha} }{4}}(\frac{\tau}{\tau_0})^{-\varrho}\in (\frac{99}{100},1), \tau\geq \tau_0.
\end{equation}
if $\beta\gg1$ (depending on $n,\Sigma,l$). We claim that $\hat F^-$ is a subsolution.

For the initial value, note
 \begin{equation}
   \lambda_-(\tau_0)\kappa=(1-\beta^{-\frac{{\tilde \alpha} }{4}})(1+o(1))\leq 1-\beta^{-\frac{{\tilde \alpha} }{2}}
\end{equation}
by \eqref{eq bound k diff} provided $\beta\gg 1$ (depending on $n,\Sigma,l$), and $|t_0|\ll 1$ (depending on $\beta,\tilde \alpha)$. Thus, $\hat F^-$ is below $S_{\kappa_1,+}$ by the monotonicity of $S_{\kappa,+}$ with respect to $\kappa$. This implies that $\hat F^-(\tau_0)$ is below $\hat F(\tau_0)$ in $B_{\frac{3}{2}\beta}$ by \eqref{eq ini lam1 lam2}.  

On the other hand, we note that $\hat F^-$ can be written as a normal graph over $\mC$ with profile function $\psi_{\lambda_-(\tau)\kappa}$ for $(z,\theta)\in [(\lambda_-(\tau)\kappa)^{\frac{1}{1-\alpha}}R_s,\infty)\times \Sigma$ by \eqref{eq S+ para}. Since
\begin{equation}
\begin{aligned}
    &\psi_{\lambda_-(\tau_0)\kappa}(z,\theta)=\kappa\lambda_-z^{\alpha}\omega_1(\theta)+O(z^{\alpha-{\tilde \alpha} })=((1-o(1))(1-\beta^{-\frac{{\tilde \alpha} }{4}}))z^{\alpha}\omega_1+O(z^{\alpha-{\tilde \alpha} })\leq w(z,\theta,\tau_0).
\end{aligned}
\end{equation}
on $S_{k,+,(2\sigma_l\tau_0)^{\frac{1-\vartheta}{2}}}\setminus S_{k,+,\frac{3}{2}\beta}$, by \eqref{eq w ini z large}, \eqref{eq psi asy}, provided $\beta\gg1$ (depending on $n,\Sigma,l$), $\hat F^{-}(\tau_0)$ is below $\hat F(\tau_0)$ in $B_{(2\sigma_l\tau_0)^{\frac{1-\vartheta}{2}}}\setminus B_{\frac{3}{2}\beta}$,

Second, for $\tau_0\leq \tau\leq \tau^\circ$, $(z,\theta)\in \partial B_{(2\sigma_l\tau)^{\frac{1-\vartheta}{2}}}\cap\mC$, we have
\begin{equation}
    \begin{aligned}
        &\psi_{\lambda_-(\tau)\kappa}(z,\theta)=\kappa\lambda_-z^{\alpha}\omega_1(\theta)+O(z^{\alpha-{\tilde \alpha} })\\
        =&((1-o(1))(1-\beta^{-\frac{{\tilde \alpha} }{4}}(\frac{\tau}{\tau_0})^{-\varrho}))z^{\alpha}\omega_1+O(z^{\alpha-{\tilde \alpha} })\\
    \leq& (1-C(n)(2\sigma_l\tau)^{-\vartheta})z^{\alpha}+O(z^{\alpha-{\tilde \alpha} })\leq w(z,\theta,\tau)
    \end{aligned}
\end{equation}
by \eqref{eq mcf intmed inn con}, \eqref{eq psi asy}, since $\frac{1}{2}<\kappa\lambda_-(\tau)<1$, provided $\tau_0\gg 1$ (depending on $n,\Sigma,l,\beta$) and $0<\varrho<\vartheta$. 

At last, let $\lambda=(\kappa\lambda_-(\tau))^{\frac{1}{1-\alpha}}$, then we have $\lambda_-\geq \frac{1}{2}$ if $\beta\geq1$, and
\begin{align*}
    &\lambda'-c_l\tau^{-1}\lambda={\kappa}^{\frac{1}{1-\alpha}}\lambda_-^{\frac{1}{1-\alpha}}[\frac{1}{1-\alpha}\lambda_-^{-1}(\varrho\beta^{-\frac{{\tilde \alpha} }{4}}(\frac{\tau}{\tau_0})^{-\rho})-c_l]\tau^{-1}
    \leq {\kappa}^{\frac{1}{1-\alpha}}\lambda_-^{\frac{1}{1-\alpha}}(\frac{2}{1-\alpha}\varrho\beta^{-\frac{{\tilde \alpha} }{4}}-c_l)\tau^{-1}<0,
\end{align*}
for $\beta\gg1$ (depending on $n,\Sigma,l$). Thus, 
\begin{equation}
\begin{aligned}
   &\langle \hat F^-_\tau-c_l\tau^{-1}\hat F^-+\hat H^+\hat\nu^-,\hat\nu^-\rangle=(\lambda'-c_l\tau^{-1}\lambda) \langle S_+(\frac{z}{\lambda}),\nu_{S_+}(\frac{z}{\lambda})\rangle<0
\end{aligned}
\end{equation}
since $\langle S_+(\frac{z}{\lambda}),\nu_{S_+}(\frac{z}{\lambda})\rangle>0$ (see \cite{HS}), and $S_{+}$ has zero mean curvature. That is, $ \hat F^-$ is a subsolution to \eqref{eq rescaled hat flow}. By avoidance principle for parabolic flows, $\hat F$ is above $\hat F^-$ in $B_{(2\sigma_l\tau)^{\frac{1-\vartheta}{2}}}$ and $\tau_0\leq \tau\leq \tau^\circ$.
\subsubsection{Upper barrier}
Now we construct an upper barrier. Let $\phi_1$ be the positive first eigenfunction of
$$\mathcal L_{S_{\kappa,+}}:=\Delta_{S_{\kappa,+}}+|A|^2_{S_{\kappa,+}}$$ 
on $S_{\kappa,+,2\beta}$ with Dirichlet boundary condition with eigenvalue $\lambda_{1,\beta}$, and $\|\phi_1\|_{L^2(S_{\kappa,+,2\beta})}=1$. That is 
\begin{equation}\label{eq phi1}
    \begin{cases}
        \mathcal L_{S_{\kappa,+}}\phi_1=-\lambda_{1,\beta}\phi_1\quad\text{ in }S_{\kappa,+,2\beta},\\
        \phi_1=0\quad\text{ on }\partial S_{\kappa,+,2\beta}.
    \end{cases}
\end{equation}
  Since $S_{\kappa,+}$ is stable (thus, is locally stable), $\lambda_{1,\beta}>0$. We extend $\phi_1$ outside $S_{\kappa,+,2\beta}$ by zero so that it is defined on the whole $S_{\kappa,+}$. 
  
Recall that $\psi_{k}$ is the profile function $S_{\kappa,+}$ over $\mC$ outside $B_{R_s}$ (see \eqref{eq S+ para}). Let $\bar \psi:=\langle S_{\kappa,+}, \nu_{S_{\kappa,+}}\rangle$ be the positive Jacobi field on $S_{\kappa,+}$ (corresponding to scaling), $\tilde \psi$ be the profile function of $S_{\kappa_2,+}$ when regarded as a normal graph over $S_{\kappa,+}$. We can write $\psi_{k},\bar \psi, \tilde \psi$ as functions on $S_{\kappa,+}$ in terms of coordinates $(\tilde z,\theta)$. Alternatively, we can write them as functions over $\mC$ outside $B_{{\kappa}^{\frac{1}{1-\alpha}}R_s}$ in coordinates $(z,\theta)$ via the parametrization \eqref{eq Sk+ para}. Here, $(\tilde z,\theta)$ is the global coordinates on $S_{\kappa,+}$ (see Appendix \ref{sec HS foliation}), and $(z,\theta)$ for the coordinates on $\mC$. By \eqref{eq psik asym}, we have
\begin{equation}\label{eq bar psi asym}
     \bar\psi(z,\theta)\sim \psi_{\kappa}(z,\theta)\sim \kappa z^\alpha\omega_1\sim \kappa\psi(z,\theta)\text{ as }z\to \infty.
\end{equation}
Since $\kappa_2=1+\beta^{-\frac{{\tilde \alpha} }{2}}$, we have 
\begin{equation}\label{eq tilde psi asy}
    \tilde\psi(z,\theta)=(\beta^{-\frac{{\tilde \alpha} }{2}}+1-\kappa)z^\alpha\omega_1+O(z^{\alpha-{\tilde \alpha} })\sim (\beta^{-\frac{{\tilde \alpha} }{2}}+1-\kappa)\psi(z,\theta)
    \sim \frac{(\beta^{-\frac{{\tilde \alpha} }{2}}+1-\kappa)}{\kappa}\psi_{\kappa}(z,\theta)\text{ as }z\to \infty.
\end{equation}
Now we construct the upper barrier $\hat F^+$ as a graph over $S_{\lambda_+(\tau)\kappa,+}=\lambda_+^{\frac{1}{1-\alpha}}S_{\kappa,+}$, 
\begin{equation}\label{eq Fhat + def}
    \hat F^+(\tilde z,\theta,\tau):=S_{\lambda_+(\tau)\kappa,+}(\tilde z,\theta)+(d_0(\tau)\phi_1-d_1(\tau)\tilde z)\nu_{S_{\lambda_+(\tau)\kappa,+}(\tilde z,\theta)}(\tilde z,\theta),
\end{equation}
where 
\begin{equation}\label{eq lambda + def}
\lambda_+(\tau)=1+\beta^{-\frac{{\tilde \alpha} }{4}}(\frac{\tau}{\tau_0})^{-\varrho},\quad  d_0(\tau)=C(\beta)(2\sigma_l\tau)^{-1+\varrho}(\frac{\tau}{\tau_0})^{-\varrho},\quad  d_1(\tau)=\delta\beta^{-\frac{{\tilde \alpha} }{4}}(2\sigma_l\tau)^{-1+\varrho}(\frac{\tau}{\tau_0})^{-\varrho}.
\end{equation}
for some constant $\delta(\beta)\ll1$, $C(\beta)\gg1$ depending on $\beta$ to be determined. Note that we can choose $\lambda_+-1,d_1\in (0,\frac{1}{1000})$ for $\beta\gg 1,0<\delta\ll1$ and $0<\rho<1$. Since $\nu_{S_{\lambda_+(\tau)\kappa,+}}(\tilde z,\theta)=\nu_{S_{\kappa,+}}(\tilde z,\theta)$, we can write $\hat F^+$ as
\begin{equation}
    \hat F^+(\tilde z,\theta,\tau)
    =S_{\lambda_+(\tau)\kappa,+}(\tilde z,\theta)+w^+(\tilde z,\theta,\tau)\nu_{S_{\kappa,+}}(\tilde z,\theta)
\end{equation}
where $w^+(\tilde z,\theta,\tau)=d_0(\tau)\phi_1-d_1(\tau)\tilde z$.

To prove that $\hat F^+$ satisfies the upper initial and boundary condition, we write $ \hat F^+$ as a normal graph over $S_{\kappa,+}$ with profile function $\hat w^+$. This can be done if $0<\lambda_+-1$, $d_0,d_1$ are sufficiently small, which is true if $\beta>0$ is large and $\tau_0>0$ is large depending on $\beta$. Moreover, outside $B_{\frac{1002}{1000}R_s}$, we can use coordinates $z$ for $\hat w$. That is  
\begin{equation}
     \hat F^+(\tilde z,\theta,\tau)=S_{\kappa,+}(\tilde z,\theta)+\hat w^+(\tilde z,\theta,\tau)\nu_{S_{\kappa,+}}(\tilde z,\theta), \quad (\tilde z,\theta)\in S_{\kappa,+};
\end{equation}
and
\begin{equation}
    \hat F^+(z,\theta,\tau)
    =S_{\kappa,+}(z,\theta)+\hat w^+(z,\theta,\tau)\nu_{S_{\kappa,+}}(z,\theta), \quad(z,\theta)\in [\frac{1002}{1000}R_s,\infty)\times \Sigma.
\end{equation}
Since $d_0(\tau)\geq 0$, $\phi_1\geq 0$, we have 
$$dist(S_{\kappa,+}(\tilde z,\theta),\lambda_+^{\frac{1}{1-\alpha}}S_{\kappa,+})\geq \frac{2}{3}(\lambda_+^{\frac{1}{1-\alpha}}-1) \langle S_{\kappa,+}(\tilde z,\theta),\nu_{S_{\kappa,+}}(\tilde z,\theta)\rangle\geq \frac{1}{C_0(\Sigma)}(\lambda_+^{\frac{1}{1-\alpha}}-1) \psi (\tilde z,\theta)$$ 
for some $C_0(\Sigma)>0$ by $\eqref{eq bar psi asym}$, \eqref{eq bound k diff}, if $0<\lambda_+-1$, $d_0,d_1$ is sufficiently small. Define
\begin{equation}\label{eq hat f+ tilde f+}
\begin{aligned}
   \tilde w^+(\tilde z,\theta,\tau):=\frac{1}{C_0}(\lambda_+^{\frac{1}{1-\alpha}}-1) \psi(\tilde z,\theta)-d_1\tilde z.
\end{aligned}
\end{equation}
Then the profile function $\hat w^+$ satisfies
\begin{align*}
     \hat w^+(\tilde z,\theta,\tau)=dist(S_{\kappa,+}(\tilde z,\theta),\hat F^+)
   \geq  dist(S_{\kappa,+}(\tilde z,\theta),\lambda_+^{\frac{1}{1-\alpha}}S_{\kappa,+})-d_1 \tilde z
   \geq\tilde w^+(\tilde z,\theta,\tau).
\end{align*}
We claim that
\begin{equation}\label{eq hat f+ tilde f+ f}
\tilde w^+(\tilde z,\theta,\tau)\geq \hat w(\tilde z,\theta,\tau)
\end{equation}
on $ S_{\kappa,+,(2\sigma_l\tau_0)^{\frac{1-\vartheta}{2}}}\times \{\tau=\tau_0\}$ and $ \partial S_{\kappa,+,(2\sigma_l\tau)^{\frac{1-\vartheta}{2}}}$, $\tau_0\leq \tau\leq \tau^\circ$; where $\hat w(\tilde z,\theta,\tau)$ is the profile function of $\hat F(\tau)$ as a graph over $S_{\kappa,+}$ defined in \eqref{eq hat w def}. The claim implies that $F^+$ satisfies the initial and boundary condition for an uppersolution. Let's prove \eqref{eq hat f+ tilde f+ f} now. First note
\begin{equation}
 \lambda_+^{\frac{1}{1-\alpha}}(\tau_0)-1\geq \frac{1}{2(1-\alpha)}\beta^{-\frac{{\tilde \alpha} }{4}}  (\frac{\tau}{\tau_0})^{-\rho}
\end{equation}
for $\tau_0\leq \tau\leq \tau^\circ$ if $\beta>0$ is small, and 
\begin{equation}
    \tilde z\leq C(\Sigma) z.
\end{equation}
for $(\tilde z,\theta)\in S_{\kappa,+}$. Thus, for $\tilde z\leq R_s$, we have from the definition of $d_1$ that
\begin{align*}
     &\tilde w^+(\tilde z,\theta,\tau_0)
    \geq \frac{1}{4C_0(1-\alpha)}\beta^{-\frac{{\tilde \alpha} }{4}}(\psi+\varepsilon(R_s))-\delta\beta^{-\frac{{\tilde \alpha} }{4}}(2\sigma_l\tau_0)^{-1+\varrho}R_s
    \geq \frac{1}{4C_0(1-\alpha)}\beta^{-\frac{{\tilde \alpha} }{4}}\psi\geq 100\tilde\psi.
\end{align*}
for $(\tilde z,\theta)\in S_{\kappa,+,R_s}$ if $\beta\gg1$ (depending on $\Sigma,{\tilde \alpha} $), $\tau_0\gg 1$ (depending on $\Sigma,R_s,\alpha,l$) by \eqref{eq bound k diff}, where $\varepsilon(R_s):=\inf_{S_{\kappa,+,R_s}}\psi(\tilde z,\theta)>0$. As a result,
\begin{equation}\label{eq tilde f+ ini1}
\begin{aligned}
    \tilde w^+(\tilde z,\theta,\tau_0)\geq &\hat w(\tilde z,\theta,\tau_0)
\end{aligned}
\end{equation}
for $(\tilde z,\theta)\in S_{k,+,R_s}$ by \eqref{eq ini lam1 lam2} and \eqref{eq tilde psi asy}. Therefore, $F^+(\tau_0)$ is above $\hat F(\tau_0)$ in $B{R_s}$. 

For $(z,\theta)\in (B_{(2\sigma_l\tau_0)^{\frac{1-\vartheta}{2}}}\setminus B_{R_s})\cap \mC$, we first note that by the definition of $d_1$,
\begin{align*}
     &\tilde w^+(z,\theta,\tau_0)\geq -C(\Sigma)d_1(\tau_0)z\min_\Sigma\omega_1+\frac{1}{C_0}(\lambda_+^{\frac{1}{1-\alpha}}-1)(\tau_0)\psi(z,\theta)
         \\
         \geq& -C(\Sigma)\delta\beta^{-\frac{{\tilde \alpha} }{4}}(2\sigma_l\tau_0)^{-1+\varrho}z\omega_1+\frac{1}{2C_0(1-\alpha)}\beta^{-\frac{{\tilde \alpha} }{4}}z^\alpha\omega_1\\
         \geq &  \beta^{-\frac{{\tilde \alpha} }{4}}z^\alpha\omega_1\big(\frac{1}{4C_0(1-\alpha)}-C(\Sigma)\delta (2\sigma_l\tau_0)^{-1+\varrho}z^{1-\alpha}\big)+\frac{1}{4C_0(1-\alpha)}\beta^{-\frac{{\tilde \alpha} }{4}}z^\alpha\omega_1.
\end{align*}     
Then we use \eqref{eq varrho def} to get, 
\begin{equation}\label{eq tau z multi boun}
    (2\sigma_l\tau)^{-1+\varrho}z^{1-\alpha}\leq (2\sigma_l\tau)^{-1+\varrho}(2\sigma_l\tau)^{\frac{(1-\alpha)(1-\vartheta)}{2}}=1
\end{equation} 
for $R_s\leq z\leq (2\sigma_l\tau)^{\frac{1-\vartheta}{2}}$, $\tau_0\leq \tau\leq \tau^\circ$. In particular, take $\tau=\tau_0$, we get
\begin{equation}\label{eq tilde f+ ini2}
    \begin{aligned}
        \tilde w^+(z,\theta,\tau_0)\geq \frac{1}{4C_0(1-\alpha)}\beta^{-\frac{{\tilde \alpha} }{4}}z^\alpha\omega_1\geq \hat w(z,\theta,\tau_0)
    \end{aligned}
\end{equation}
by \eqref{eq tilde psi asy} and \eqref{eq ini lam1 lam2} for $(z,\theta)\in (B_{(2\sigma_l\tau_0)^{\frac{1-\vartheta}{2}}}\setminus B_{R_s})\cap \mC$, if we choose $\delta<\frac{1}{4C_0(1-\alpha)C(\Sigma)}$, and $\beta\gg 1$ is large (depending on $\alpha, R_s$).

Similarly, for $\tau_0\leq \tau\leq \tau^\circ$, $(z,\theta)\in \partial B_{(2\sigma_l\tau)^{\frac{1-\vartheta}{2}}}\cap\mC$, we can use \eqref{eq tau z multi boun} to get 
\begin{align*}
     &\tilde w^+(z,\theta,\tau)
         \geq  \beta^{-\frac{{\tilde \alpha} }{4}}z^\alpha(\frac{\tau}{\tau_0})^{-\varrho}\omega_1\big(\frac{1}{4C_0(1-\alpha)}-C(\Sigma)\delta (2\sigma_l\tau)^{-1+\varrho}z^{1-\alpha}\big)+\frac{1}{4C_0(1-\alpha)}\beta^{-\frac{{\tilde \alpha} }{4}}(\frac{\tau}{\tau_0})^{-\varrho}z^\alpha \omega_1\\
         \geq 
         & \beta^{-\frac{{\tilde \alpha} }{4}}z^\alpha(\frac{\tau}{\tau_0})^{-\varrho}\omega_1\big(\frac{1}{4C_0(1-\alpha)}-C(\Sigma)\delta\big)+\frac{1}{4C_0(1-\alpha)}\beta^{-\frac{{\tilde \alpha} }{4}}z^\alpha(\frac{\tau}{\tau_0})^{-\varrho}\omega_1\\
         \geq& \frac{1}{4C_0(1-\alpha)}\beta^{-\frac{{\tilde \alpha} }{4}}z^\alpha(\frac{\tau}{\tau_0})^{-\varrho}\omega_1 \geq C(n,R,l)(2\sigma_l\tau)^{-\vartheta}z^\alpha\omega_1
\end{align*}
since $0<\varrho<\vartheta$ by \eqref{eq varrho def} if we choose $\delta<\frac{1}{4C_0(1-\alpha)C(\Sigma)}$ and $\tau_0\gg1$ (depending on $n,\Sigma,l,\beta,\varrho,\vartheta$). Here $C(n,\Sigma,l)$ is the constant in \eqref{eq mcf intmed inn con}. Thus, 
\begin{equation}\label{eq tilde f+ boun}
    \begin{aligned}
        \tilde w^+(\tilde z,\theta,\tau)\geq \hat w(\tilde z,\theta,\tau).
    \end{aligned}
\end{equation}
for $\tau_0\leq \tau\leq \tau^\circ$, $(\tilde z,\theta)\in \partial S_{\kappa,+,(2\sigma_l\tau)^{\frac{1-\vartheta}{2}}}$ by \eqref{eq mcf intmed inn con}. 

Next, we explore the evolution equation of $\hat F^+$. To simplify notations, we define $\tilde\lambda_+:=\lambda_+^{\frac{1}{1-\alpha}}$, and we write the metric, second fundamental form, Jacobi operator of $S_{\lambda_+\kappa,+}$ as $\tilde g$, $\tilde A$, $\tilde L=\Delta_{\tilde g}+|\tilde A|^2$, respectively. Then we have
\begin{equation}\label{eq II super eq df}
\begin{aligned}
    &J:=\langle \hat F^+_\tau-c_l\tau^{-1}\hat F^++\hat H^+\hat\nu^+,\hat\nu^+\rangle\\
    =& \langle\tilde\lambda_+'S_{\kappa,+}+(d_0'\phi_1-d_1'\tilde z)\nu_{S_{\kappa,+}}-c_l\tau^{-1}F^+,\hat\nu^+\rangle+\hat H^+\\
    =&\tilde\lambda_+'\langle S_{\kappa,+},\hat\nu^+\rangle+(d_0'\phi_1-d_1'\tilde z)\langle \nu_{S_{\kappa,+}},\hat\nu^+\rangle-c_l\tau^{-1}\langle \tilde\lambda_+S_{\kappa,+}+(d_0\phi_1-d_1\tilde z)\nu_{S_{\kappa,+}},\hat\nu^+\rangle +\hat H^+.
\end{aligned}
\end{equation}
Since $S_{\kappa,+}$ is a smooth minimal hypersurface, $S_{\lambda_+\kappa,+}=\tilde\lambda_+S_{\kappa,+}$ is also a minimal hypersurface. In particular, it has zero mean curvature. We have  
\begin{align*}
    &\hat H^+ \langle \nu_{S_{\lambda_+\kappa,+}},\hat\nu^+\rangle^{-1}
    =-\tilde L(w^+)-E(w^+)
    =-\tilde L(d_0\phi_1-d_1\tilde z)-E(d_0\phi_1-d_1\tilde z)
\end{align*}  
where $E$ is defined in \eqref{eq error outer}. Using the equation \eqref{eq phi1} for $\phi_1$, we get
\begin{equation}
    \hat H^+(\tilde z,\theta)\langle \nu_{S_{\kappa,+}},\hat\nu^+\rangle^{-1}=d_0\tilde\lambda_+^{-2}\lambda_{1,\beta}\phi_1+d_1\tilde\lambda_+^{-2}L_{S_{\kappa,+}}\tilde z-E(d_0\phi_1-d_1\tilde z).
\end{equation}
and thus,
\begin{align*}
    J\geq &\tilde\lambda_+'\langle S_{\kappa,+},\hat\nu^+\rangle+(d_0'\phi_1-d_1'\tilde z)\langle \nu_{S_{\kappa,+}},\hat\nu^+\rangle-c_l\tau^{-1}\langle \tilde\lambda_+S_{\kappa,+}+(d_0\phi_1-d_1\tilde z)\nu_{S_{\kappa,+}},\hat\nu^+\rangle\\
    &+[d_0\tilde\lambda_+^{-2}\lambda_{1,\beta}\phi_1+d_1\tilde\lambda_+^{-2}L_{S_{\kappa,+}}\tilde z-E(d_0\phi_1-d_1\tilde z)]\langle \nu_{S_{\kappa,+}},\hat\nu^+\rangle.
\end{align*}
On the other hand, note $0<\frac{1}{C_1(\Sigma)}\psi\leq \langle S_{\kappa,+},\hat\nu^+\rangle \leq C_1(\Sigma) \psi$, $0\leq \nu_{S_{\kappa,+}},\hat\nu^+\leq 1$ $\phi_1\geq 0$ on $S_{\kappa,+}$ for some $C_1(\Sigma)$ depending on $\Sigma$, if $\lambda_+-1, d_0,d_1$ are sufficiently small. We get
\begin{equation}\label{eq F+t nu}
\begin{aligned}
    J\geq&  (d_0'-c_l\tau^{-1}d_0+d_0\tilde\lambda_+^{-2}\lambda_{1,\beta})\phi_1 \langle \nu_{S_{\kappa,+}},\hat\nu^+\rangle+C_1(\Sigma)(\tilde\lambda_+'-c_l\tau^{-1}\tilde\lambda_+)\psi\\
    &+d_1\tilde\lambda_+^{-2}L_{S_{\kappa,+}}(\tilde z)\langle \nu_{S_{\kappa,+}},\hat\nu^+\rangle-|E(d_0\phi_1-d_1\tilde z)|=:J_1+J_2+J_3
 \end{aligned}
\end{equation}
since $\tilde\lambda_+,d_0,d_1>0$, $\tilde\lambda_+',d_0,,d_1'<0$,  and $\langle \nu_{S_{\kappa,+}},\hat\nu^+\rangle>0$ if  $\lambda_+-1, d_0,d_1$ are sufficiently small. Here
\begin{align*}
    &J_1:=(d_0'-c_l\tau^{-1}d_0+d_0\tilde\lambda_+^{-2}\lambda_{1,\beta})\phi_1\langle \nu_{S_{\kappa,+}},\hat\nu^+\rangle\quad 
    J_2:=C_1(\tilde\lambda_+'-c_l\tau^{-1}\tilde\lambda_+)\psi,\\
    &J_3:=d_1\tilde\lambda_+^{-2}L_{S_{\kappa,+}}(\tilde z)\langle \nu_{S_{\kappa,+}},\hat\nu^+\rangle-|E(d_0\phi_1-d_1\tilde z)|.
\end{align*}
Now we estimate $J_i$ ($i=1,2,3$) separately. We first consider the case when $(z,\theta)\in S_{k,+,(2\sigma_l\tau)^{\frac{1-\vartheta}{2}}}\setminus S_{\kappa,+,\beta}$. By the choice of $\lambda_+,d_0,d_1$ and \eqref{eq tilde psi asy}, we have
\begin{equation}\label{eq J1}
    J_1\geq 0,
\end{equation}
and
\begin{equation}\label{eq J2}
    J_2\geq C_1(\tilde\lambda_+'-c_l\tau^{-1}\tilde\lambda_+)\tilde z^{\alpha}\omega_1,
\end{equation}
since
\begin{equation}
    \quad d_0'-c_l\tau^{-1}d_0+d_0\tilde\lambda_+^{-2}\lambda_{1,\beta}>0
\end{equation}
for $\beta\leq \tilde z\leq (2\sigma_l\tau)^{\frac{1-\vartheta}{2}}$, $\tau_0$ large. For $J_3$, note that $S_{\kappa,+}(z,\theta)$ is asymptotic to $\mC$ as $\tilde z\to\infty$. We have 
\begin{equation}\label{eq linear}
    L_{S_{\kappa,+}}\tilde z\geq \frac{1}{2}[\partial_{\tilde z \tilde z}+\frac{\partial_{\tilde z}}{\tilde z}+\frac{\Delta_{\tilde\Sigma} }{\tilde z^2}+|\tilde A|^2_{\tilde g}]\tilde z= \frac{1}{2\tilde z}+\frac{1}{2}|\tilde A|_{\tilde g}^2\tilde z\geq \frac{1}{2\tilde z} ,
\end{equation}
and
\begin{equation}\label{eq error}
\begin{aligned}
     |E(d_0\phi_1-d_1\tilde z)|\leq& C(n,\Sigma)[d_0^2(\frac{\phi_1}{\tilde z}+|\tilde \nabla \phi_1|+\tilde z|\tilde\nabla^2\phi_1|)^2+d_1^2(\frac{\tilde z}{\tilde z}+|\tilde \nabla \tilde z|+\tilde z|\tilde\nabla^2\tilde z|)^2]\\
     \leq& C(n,\Sigma)(d_0^2\|\phi_1\|_{C^2(B_1(\tilde z,\theta))\cap S_{\kappa,+} }^2+d_1^2)
\end{aligned}
\end{equation}
by \eqref{eq E stru 2} and \eqref{eq z C2 norm} for $\tilde z\geq \beta$, $\beta\gg1 $ large (depending on $n,\Sigma$), if $\lambda_+-1, d_0,d_1$ are sufficiently small. Thus, if we choose $\lambda_+-1, d_0,d_1$ are sufficiently small such that $\langle \nu_{S_{\kappa,+}},\hat\nu^+\rangle\geq \frac{1}{2}$, then we have
\begin{equation}\label{eq J3}
\begin{aligned}
    J_3\geq& d_1\tilde\lambda_+^{-2} (\frac{1}{8\tilde z}+\frac{1}{32\beta})-C(n,\Sigma)(d_0^2C(\Sigma,\beta)+d_1^2)\geq d_1\tilde\lambda_+^{-2} \frac{1}{8\tilde z}-C(n,\Sigma)d_1^2\geq d_1\tilde \lambda_+^{-2} \frac{1}{16\tilde z}
\end{aligned}
\end{equation}
since $\phi_1=0$ outside $S_{\kappa,+,2\beta}$, $|\phi_1|_{C^2(S_{\kappa,+,2\beta+1})}\leq C(\Sigma,\beta)$,  $\frac{1}{\tilde z}\geq \frac{1}{4\beta}$ on $S_{\kappa,+,2\beta+1}$, and
\begin{equation}
    \beta C(\Sigma,\beta)d_0^2\leq d_1\leq \frac{1}{16\tilde\lambda_+^2\tilde z}
\end{equation}
for $\beta\leq \tilde z\leq (2\sigma_l\tau)^{\frac{1-\vartheta}{2}}$, and $\tau_0\gg1$. This can be achieved if $\varrho$ defined in \eqref{eq varrho def} satisfies $ -1+\varrho<\frac{-1+\vartheta}{2}$, i.e. $ \varrho<\frac{1+\vartheta}{2}.$
Since we can take $\vartheta\in(0,\frac{1}{2})$ in \eqref{eq theta def}, then \eqref{eq varrho def} implies $\varrho<\vartheta<\frac{1}{2}< \frac{1+\vartheta}{2}$ is satisfies automatically. Plugging \eqref{eq J1}, \eqref{eq J2}, and \eqref{eq J3} into \eqref{eq F+t nu}, we get
\begin{equation}
   J\geq d_1\lambda_+^{-2} \frac{1}{16\tilde z}+C_1(\Sigma)(\tilde\lambda_+'-c_l\tau^{-1}\tilde\lambda_+)\tilde z^{\alpha}\omega_1\geq 0
\end{equation}
for $(\tilde z,\theta)\in S_{k,+,(2\sigma_l\tau)^{\frac{1-\vartheta}{2}}}\setminus S_{\kappa,+,\beta}$, and $\tau_0\gg 1$ since 
\begin{align*}
    &d_1\geq C(\beta)\max \{\tau^{-1},|\tilde\lambda_+'|\}\max_{\Sigma}\omega,\quad \alpha<-1
\end{align*}
for $\beta\leq \tilde z\leq (2\sigma_l\tau)^{\frac{1-\vartheta}{2}}$, and $\tau_0\gg 1$ large. 

For the case $(\tilde z,\theta)\in S_{\kappa,+,\beta}$, we note that $\phi_1$ is strictly positive on $S_{\kappa,+,2\beta}$ and $\lambda_{1,\beta}>0$ since $S_{k,+}$ is locally stable. This implies that 
\begin{equation}
    \lambda_{1,\beta}\phi_1\geq \varepsilon_1(\beta)>0\text{ on }S_{\kappa,+,\beta}.
\end{equation} 
Moreover, by the definition of $\lambda_+,d_0,d_1$, we have
\begin{equation}
    d_0'-c_l\tau^{-1}d_0+\frac{1}{2}d_0\tilde\lambda_+^{-2}\lambda_{1,\beta}>0.
\end{equation}
Thus, if we choose $\lambda_+-1, d_0,d_1$ are sufficiently small such that $\langle \nu_{S_{\kappa,+}},\hat\nu^+\rangle\geq \frac{1}{2}$, then we have
\begin{equation}\label{eq J1 1}
    J_1\geq \frac{1}{4}d_0\tilde\lambda_+^{-2}\lambda_{1,\beta}\phi_1\geq \frac{1}{4}d_0\tilde\lambda_+^{-2}\varepsilon_1(\beta).
\end{equation}
on $S_{\kappa,+,\beta}$. On the other hand, since 
\begin{equation}
    |\psi|_{C^0(S_{\kappa,+,\beta})}+|\phi_1|_{C^2(S_{\kappa,+,\beta})}\leq C(\beta)
\end{equation}
on $S_{\kappa,+,\beta}$, we have
\begin{equation}\label{eq J2 1}
    J_2\geq C(\beta) (\tilde\lambda_+'-c_l\tau^{-1}\tilde\lambda_+),
\end{equation}
\begin{equation}\label{eq J3 1}
    J_3\geq -C(\beta)(d_1+d_0^2)
\end{equation}
on $S_{\kappa,+,\beta}$, for $\tau_0$ large. Thus
\begin{equation}
    J\geq \frac{1}{4}d_0\tilde\lambda_+^{-2}\varepsilon_1(\beta)-C(\beta)(c_l\tau^{-1}\tilde\lambda_+-\tilde\lambda_+'+d_1+d_0^2)\geq 0
\end{equation}
on $S_{\kappa,+,\beta}$, since
\begin{equation}
    d_0\geq \frac{4C(\beta)}{\varepsilon_1(\beta)}(c_l\tau^{-1}\tilde\lambda_+-\tilde\lambda_+'+d_1+d_0^2),\quad 1\leq \lambda_+\leq 2,
\end{equation}
for $C(\beta)$ large, and $\tau_0$ large (depending on $\beta$).

Now we prove \eqref{eq tip inner C0 est hat}. For this purpose, we only need to estimate the lower bound of $\hat w^{-}$ and upper bound of $\hat w^+$ which are the profile function of $\hat F^-$ and $\hat F^+$ over $S_{\kappa,+}$, respectively. For $\hat w^{-}$, we note from the definition of $\hat F^-$ in \eqref{eq Fhat - def}, \eqref{eq lambda - def}, \eqref{eq bound k diff}, and \eqref{eq psik asym}, that
\begin{equation}
    |\hat w^{-}( z,\theta,\tau)|\leq C(\Sigma) (1-\lambda_-^{\frac{1}{1-\alpha}})\psi_{\kappa}(z,\theta)\leq C(\Sigma) \beta^{-\frac{{\tilde \alpha} }{4}}(\frac{\tau}{\tau_0})^{-\varrho}z^\alpha
\end{equation}
for $z\geq \beta\geq (\lambda_-\kappa)^{\frac{1}{1-\alpha}}R_s, \theta\in\Sigma,\tau_0\leq\tau\leq \tau^\circ$, for $\beta\gg 1$ large (depending on $R_s$). Inside $S_{\kappa,+,\beta}$, we have
\begin{equation}
     |\hat w^{-}( \tilde z,\theta,\tau)|\leq C(\Sigma) (1-\lambda_-^{\frac{1}{1-\alpha}})\psi_{\kappa}(\tilde z,\theta)\leq C_1(\Sigma) \beta^{-\frac{{\tilde \alpha} }{4}}(\frac{\tau}{\tau_0})^{-\varrho}
\end{equation} 
Here $ C_1(\Sigma)\geq \max_{S_{\kappa,+,\beta}}\psi_{\kappa}$ can be chosen so that it only depends on $\Sigma$ and independent of $\kappa,\beta$ by \eqref{eq psik asym} and \eqref{eq bound k diff}.

Similarly, from the definition of $\hat F^+$ in \eqref{eq Fhat + def}, \eqref{eq lambda + def}, \eqref{eq bound k diff}, and \eqref{eq psik asym}, we know that 
\begin{equation}
  |\hat w^{+}( z,\theta,\tau)|\leq C(\Sigma) (\lambda_+^{\frac{1}{1-\alpha}}-1)\psi_{\kappa}(z,\theta)+d_0\phi_1\leq C(\Sigma) (\frac{\tau}{\tau_0})^{-\varrho}(\beta^{-\frac{{\tilde \alpha} }{4}}\psi_{\kappa}+C(\beta)(2\sigma_l\tau)^{-1+\varrho}\phi_1)\leq C_1(\Sigma)\beta^{-\frac{{\tilde \alpha} }{4}}(\frac{\tau}{\tau_0})^{-\varrho}z^\alpha
\end{equation}
for $z\geq \beta\geq (\lambda_+\kappa)^{\frac{1}{1-\alpha}}R_s, \theta\in\Sigma,\tau_0\leq\tau\leq \tau^\circ$ (note $d_1>0$), for $\beta\gg 1$ large (depending on $R_s$), and $\tau_0$ large (depending on $\beta$) such that 
\begin{align*}
    (2\sigma_l\tau_0 )^{-1+\varrho}C(\beta)\max_{S_{\kappa,+,2\beta}}(\phi_1\psi^{-1})\leq \beta^{-\frac{{\tilde \alpha} }{4}}.
\end{align*}
Since $\phi_1=0$ on $S_{\kappa,+,2\beta}^c$, such a $\tau_0$ exists. Inside $S_{\kappa,+,\beta}$, we have
\begin{equation}
    |\hat w^{+}( \tilde z,\theta,\tau)|\leq C(\Sigma) (\lambda_+^{\frac{1}{1-\alpha}}-1)\psi_{\kappa}(\tilde z,\theta)+d_0\phi_1\leq C(\Sigma) (\frac{\tau}{\tau_0})^{-\varrho}(\beta^{-\frac{{\tilde \alpha} }{4}}\psi_{\kappa}+C(\beta)(2\sigma_l\tau)^{-1+\varrho}\phi_1)\leq C_1(\Sigma)\beta^{-\frac{{\tilde \alpha} }{4}}(\frac{\tau}{\tau_0})^{-\varrho},
\end{equation}
for $\tau_0$ large (depending on $\beta$) such that 
$$(2\sigma_l\tau_0 )^{-1+\varrho}C(\beta)\max_{S_{\kappa,+,\beta}}\phi_1\leq \beta^{-\frac{{\tilde \alpha} }{4}}.$$
From the above four inequalities, we get \eqref{eq tip inner C0 est 0}, \eqref{eq tip inner C0 est hat}.

\subsection{Noncompact outer region}
We use the gradient and curvature estimates in \cite{EH} to prove the first and second derivative estimates in the region $\{(x,\theta)\in \mC|(x,\theta)\in [\frac{1}{5}\rho,\infty)\times \Sigma\}$.
\begin{proposition}\label{prop C2 estim outer noncom first}
If $0<\rho\ll 1$ (depending on $n,\Sigma,\Lambda$) and $|t_0|\ll \rho^2$ (depending on $n,\Sigma$), there holds 
\begin{equation}\label{eq C0 est outer noncom}
    |u(x,\theta,t)-u(x,\theta,t_0)|\leq C(n,\Sigma)\sqrt{t-t_0},
    \end{equation}
and
\begin{align*}
    |\nabla_{(x,\theta)} u(x,\theta,t)|\lesssim 1,\\
    |\nabla^2_{(x,\theta)}u(x,\theta,t)|\leq \frac{C(n,\Sigma)}{\sqrt{t-t_0}},
\end{align*}
in the region $\{(x,\theta)\in \mC|(x,\theta)\in [\frac{1}{5}\rho,\infty)\times \Sigma\}$.
\end{proposition}
\begin{remark}
    \eqref{eq C0 est outer noncom} implies the first equation in \eqref{eq u C2 sum}.
\end{remark}
\begin{proof}
   Recall by the admissible condition, $\Sigma_t:=\Sigma_t^{\mathbf a}$ can we written as a normal graph over $\mC$ outside $B(O,\frac{1}{6}\rho)$ with profile function $u(x,\theta,t)$ satisfying 
    \begin{equation}\label{eq ux quo ini}
        \max_{(x,\theta)\in [\frac{1}{6}\rho,\infty)\times \Sigma}\{x^{-1}|u(x,\theta,t_0)|,|\nabla_{(x,\theta)} u(x,\theta,t_0)|\}\leq \varepsilon_0, 
    \end{equation}
   for some $\varepsilon_0(\Sigma)$ small by \eqref{eq ini out noncom reg}. We thus have
    \begin{equation}\label{eq ini nor prod}
    \begin{aligned}
        \frac{1}{2}g_{ij}\leq \bar{g}_{ij}(t_0)\leq 2g_{ij},\\
        \nu(x,\theta)\cdot\bar\nu(x,\theta,t_0)\geq\frac{99}{100}. 
    \end{aligned}
    \end{equation}
    for any $(x,\theta)\in \mC\cap B(O,\frac{1}{6}\rho)^c$ by \eqref{eq nor C nor gra eq}. Now, fix a point $p^*=(x^*,\theta^*)\in \mC\cap B(O,\frac{1}{5}\rho)^c$, let $\nu^*:=\nu_{\mC}(p^*)$, $q^*=F(p^*,t_0)$, $\bar\nu^*=\bar\nu(q^*)$. Then for any $(x,\theta)\in B(p^*,100\varepsilon_1 x^*)\cap \mC$ ($\varepsilon_1\geq \varepsilon_0$), we have
    \begin{equation}\label{eq nor dif cone}
    \begin{aligned}
        \nu^*\cdot \nu(x,\theta)\geq \frac{99}{100},
    \end{aligned}
    \end{equation}
    and 
    $$x\geq x^*-100\varepsilon_1 x^*\geq \frac{99}{100}x^*
    >\frac{1}{6}\rho$$ 
    if $0<\varepsilon_1<\frac{1}{10^4}$ small. Thus, $(x,\theta)\in \mC\cap B(O,\frac{1}{6}\rho)^c$. Moreover, by \eqref{eq ini nor prod}, \eqref{eq nor dif cone}, we have
    \begin{equation}\label{eq nor prod image domain}
        \bar\nu(x,\theta,t_0)\cdot \nu^*\geq \frac{1}{2}.
    \end{equation}
    for any $(x,\theta)\in B(p^*,100\varepsilon_1 x^*)\cap \mC$ if $0<\varepsilon_1<\frac{1}{10^4}$ small. 
    
    We claim that for any $F(x,\theta,t_0)\in\Sigma_{t_0}\cap B(q^*,\varepsilon_1\rho)$, we have $(x,\theta)\in B(p^*,100\varepsilon_1 x^*)\cap \mC$. We prove the claim by a contradiction argument. First note, for any $(x,\theta)\in B(p^*,100\varepsilon_1 x^*)^c\cap \mC$, by \eqref{eq ux quo ini}, we have
    \begin{align*}
     |F(x,\theta,t_0)-q^*|      
        \geq |(x,\theta)-(x^*,\theta^*)|-|u(x,\theta,t_0)|-|u(x^*,\theta^*,t_0)|\geq  |(x,\theta)-(x^*,\theta^*)|-\varepsilon_0(x+x^*).
    \end{align*}    
   Suppose, $F(x,\theta,t_0)\in\Sigma_{t_0}\cap B(q^*,\varepsilon_1\rho)$, and $(x,\theta)\in B(p^*,100\varepsilon_1 x^*)^c\cap \mC$. If $2x^*\leq x$, we have $|(x,\theta)-(x^*,\theta^*)|\geq |x-x^*|\geq \frac{1}{2}x$, and
     \begin{align*}
        \varepsilon_1\rho&\geq |F(x,\theta,t_0)-q^*|\geq \frac{1}{2}x-\varepsilon_0x-\varepsilon_0x^*\geq \frac{1}{4} x  \geq \frac{1}{2}x^*\geq \frac{1}{10}\rho,
    \end{align*}   
 if we take $\varepsilon_0\leq \frac{1}{8}$, which is a contradiction since $\varepsilon_1<\frac{1}{10^4}$. If $2x^*\geq x$, then  
   \begin{align*}
       \varepsilon_1\rho&\geq |F(x,\theta,t_0)-q^*|\geq 100\varepsilon_1 x^*-2\varepsilon_0x^*-\varepsilon_0x^* = 97\varepsilon_1 x^*\geq \frac{97}{5}\varepsilon_1\rho
    \end{align*}   
   since $\varepsilon_1\geq \varepsilon_0$, which is again a contradiction. Thus, the claim is true. Then \eqref{eq nor prod image domain} implies that for any $q\in \Sigma_{t_0}\cap B(q^*,\varepsilon_1\rho)$, we have
    \begin{equation}\label{eq ini normal fixed} 
        \bar\nu(q)\cdot \nu^*\geq \frac{1}{2}.
    \end{equation}
    By gradient estimate in \cite{EH}, we have
    \begin{equation}
    (\bar\nu (F_t)\cdot \nu^*)^{-1}\leq \left(1-\frac{|F_t-q^*|^2+2n(t-t_0)}{(\varepsilon_1\rho)^2}\right)^{-1}\sup_{\Sigma_{t_0}\cap B(q^*,\varepsilon_1\rho)}(\bar\nu(F_{t_0})\cdot\nu^*)^{-1}
    \end{equation}
    for any $F_t\in \Sigma_t\cap B(q^*,\sqrt{(\varepsilon_1\rho)^2-2n(t-t_0)})$. Consequently, 
    \begin{equation}
        (\bar\nu (F_t)\cdot \nu^*)^{-1}\leq (1-\frac{1}{4})\cdot 2=\frac{3}{2}
    \end{equation}
    for any $F_t\in \Sigma_t\cap B(q^*,\sqrt{(\frac{\varepsilon_1\rho}{2})^2-2n(t-t_0)})$. It follows, by the curvature estimates in \cite{EH}, that 
    \begin{equation}
        |\bar A(F_t)|\leq C(n,\Sigma)(\frac{1}{\sqrt{t-t_0}}+\frac{1}{\varepsilon_1\rho})
    \end{equation}
    for any $F_t\in \Sigma_t\cap B(q^*,\sqrt{(\frac{\varepsilon_1\rho}{3})^2-2n(t-t_0)})$. By choosing $|t_0|\ll\varepsilon_1^2\rho^2$ (depending on $n,\Sigma$), we may assume that
    \begin{align*}
        \sqrt{(\frac{\varepsilon_1\rho}{4})^2-2n(t-t_0)}\geq \frac{\varepsilon_1\rho}{5}
    \end{align*}
    and 
    \begin{equation}\label{eq out sec fund norm}
         |\bar A(F_t)|\leq \frac{C(n,\Sigma)}{\sqrt{t-t_0}}.
    \end{equation}
    for any $F_t\in \Sigma_t\cap B(q^*,\frac{\varepsilon_1\rho}{5})$, $t_0\leq t\leq t^\circ$.
    
    Now, for each $(x,\theta)\in \mC$ with $x\geq \frac{1}{5}\rho$, let $t_{(x,\theta)}$ be the maximal time so that
    \begin{align*}
        F(x,\theta,t)\in \Sigma_t\cap B(F(x,\theta,t_0),\frac{\varepsilon_1\rho}{5})
    \end{align*}
    for all $t_0\leq t\leq t_{(x,\theta)}$. Then we have,
    \begin{align*}
        |\partial_tF(x,\theta,t)|=|\bar H(F(x,\theta,t))|\leq \frac{C(n,\Sigma)}{\sqrt{t-t_0}},
    \end{align*}
   by \eqref{eq out sec fund norm}. Integrating with respect to $t$, we get
    \begin{equation}
        |F(x,\theta,t)-F(x,\theta,t_0)|\leq C(n,\Sigma)\sqrt{t-t_0}
    \end{equation}
    for all $t_0\leq t\leq t_{(x,\theta)}$. Thus, if $|t_0|\ll 1$ (depending on $n,\Sigma$), we may assume that $t_{(x,\theta)}=t^\circ$ and
    \begin{equation}
        d_H(F_t\setminus B(O,\frac{\rho}{5}),F_{t_0}\setminus B(O,\frac{\rho}{5}))\leq C(n,\Sigma)\sqrt{t-t_0}
    \end{equation}
    for all $t_0\leq t\leq t^\circ$, where $d_H$ is the Hausdorff distance in $\mathbb R^{n+1}$. It follows from \eqref{eq ini out noncom reg} and the above inequaltiy that
    \begin{equation}\label{eq C0 outer noncom}
    \begin{aligned}
        &|u(x,t)-u(x,t_0)|\leq C(n,\Sigma)\sqrt{t-t_0},\\
        &-\varepsilon_0x-C(n,\Sigma)\sqrt{t-t_0}\leq u(x,t)\leq \varepsilon_0x+C(n,\Sigma)\sqrt{t-t_0}
    \end{aligned}
    \end{equation}
    for $(x,\theta)\in [\frac{\rho}{5},\infty)\times \Sigma$, $t_0\leq t\leq t^\circ$. Plugging this into \eqref{eq bar gij u}, we get 
    \begin{equation}\label{eq hat g low outer noncom}
        \bar g_{ij}\geq g_{ij}-(\varepsilon_0 x+C(n)\sqrt{t-t_0})\frac{C(\Sigma)}{x}\geq g_{ij}-(\varepsilon_0+C(n)\frac{\sqrt{t-t_0}}{\rho})C(\Sigma)\geq c_1(n,\Sigma)g_{ij}    
    \end{equation}
for some $c_1(n,\Sigma)>0$ if $|t_0|\leq \frac{\varepsilon_0\rho}{C(n,\Sigma)}$. So, $\bar g_{ij}$ is positive definite. In particular, $\bar g_{ij}$ is invertible. Furthermore, by taking $x=x^*,\theta=\theta^*$ in \eqref{eq ini normal fixed} and replacing $t_0$ by $t$, we obtain from \eqref{eq nor C nor gra eq} that
    \begin{equation}\label{eq V lower}
        (1-u_i(p^*,t)u_j(p^*,t)\bar g^{ij})^{\frac{1}{2}}=\nu^*\cdot \nu (F(p^*,t))\geq \delta>0.
    \end{equation}
    Thus,
    \begin{equation}\label{eq trac grad}
        \frac{|\nabla u(p^*,t)|^2}{n+u(p^*,t)^2|A(p^*)|^2+|\nabla u(p^*,t)|^2}=\frac{|\nabla u(p^*,t)|^2}{tr(g^{-1}\bar g)}\leq 
        \lambda_{min}(g^{-1}\bar g)|\nabla u(p^*,t)|^2\leq \bar g^{ij}u_i(p^*,t)u_j(p^*,t)\leq 1-\delta^2.
    \end{equation}
    By \eqref{eq C0 outer noncom},
    $$u^2|A|^2\leq C(n,\Sigma)\frac{u^2}{x^2}\leq \varepsilon_0^2+C(n,\Sigma)$$ 
    for $(x,\theta)\in [\frac{\rho}{5},\infty)\times \Sigma$, \eqref{eq trac grad} implies that
    \begin{equation}\label{eq gra outer noncom}
        |\nabla u(p^*,t)|^2\leq C(n,\Sigma).
    \end{equation}
    Since $(x^*,\theta^*)$ is arbitrary, \eqref{eq gra outer noncom} holds for $(x,\theta)\in [\frac{\rho}{5},\infty)\times \Sigma$, $t_0\leq t\leq t^\circ$. From this and \eqref{eq bar gij u}, we have
    \begin{equation}\label{eq hat g upp outer noncom}
        \bar g_{ij}\leq C_1(n,\Sigma)g_{ij}.
    \end{equation}
   for some $C_1(n,\Sigma)>0$. Combined with \eqref{eq hat g low outer noncom}, there are exists $c_1(n,\Sigma),C_1(n,\Sigma)>0$, such that
   \begin{equation}\label{eq gij bar gij equi}
       c_1g_{ij}\leq \bar g_{ij}\leq C_1g_{ij}
   \end{equation}
   for $(x,\theta)\in [\frac{\rho}{5},\infty)\times \Sigma$, $t_0\leq t\leq t^\circ$.
   
   For the second derivative, we note that by \eqref{eq out sec fund norm}, and \eqref{eq gij bar gij equi}, 
   \begin{equation}\label{eq sec deri sec fund outer noncom 1}
        \frac{C(n,\Sigma)}{t-t_0}\geq |\bar A(F_t)|_{\bar g}^2=\bar g^{ij}\bar g^{pq}\bar h_{ip}\bar h_{jq}\geq \frac{1}{C_1^2} g^{ij}g^{pq}\bar h_{ip}\bar h_{jq}
   \end{equation}
   for $t_0\leq t\leq t^\circ$, $(x,\theta)\in [\frac{\rho}{5},\infty)\times \Sigma$. On the other hand, by \eqref{eq bar hij u}, \eqref{eq V lower} and \eqref{eq gra outer noncom} (we can replace $(x^*,\theta^*)$ by $(x,\theta)$ there since $(x^*,\theta^*)$ is arbitrary), we get
   \begin{equation}\label{eq sec deri sec fund outer noncom 2}
     \frac{1}{C_1^2} g^{ij}g^{pq}\bar h_{ip}\bar h_{jq}
        \geq
       \frac{1}{C_1^2} g^{ij}g^{pq}(1-u_iu_j\bar g^{ij})u_{ip}u_{jq}-C(n,\Sigma)|\nabla^2 u|-C(n,\Sigma)
        \geq \frac{\delta^2}{C_1^2}|\nabla^2 u|^2-C(n,\Sigma)|\nabla^2u|-C(n,\Sigma).
   \end{equation}
   Combining \eqref{eq sec deri sec fund outer noncom 1} and \eqref{eq sec deri sec fund outer noncom 2}, we get
    \begin{equation}
        |\nabla^2 u|\leq \frac{C(n,\Sigma)}{\sqrt{t-t_0}}
    \end{equation}
    for $(x,\theta)\in [\frac{\rho}{5},\infty)\times \Sigma$, $t_0\leq t\leq t^\circ$.
\end{proof}

\section{Higher order estimates}\label{sec higher est} 
In this section we are going to prove Proposition \ref{prop higher est}. The estimates is based on the $C^0$ estimate in Proposition \ref{prop C0 est} and maximum principle. First, we prove the $C^2$ estimates, then the higher estimates can be obtained via Schauder theory. We consider the noncompact outer region first.
\subsection{$C^2$ estimates for the noncompact outer region}
\begin{lemma}\label{lem imp gra est outer noncom}
    If $0<\rho\ll 1$ (depending on $n,\Lambda$) and $|t_0|\ll 1$ (depending on $n,\rho$), there holds 
    \begin{equation}\label{eq impro grad est noncom outer} 
        \sup_{(x,\rho)\in [\frac{1}{4}\rho,\infty)\times \Sigma}|\nabla_{(x,\theta)} u(x,\theta,t)|\leq\sup_{(x,\rho)\in [\frac{1}{5}\rho,\infty)\times \Sigma}|\nabla_{(x,\theta)} u(x,\theta,t_0)|+C(n,\Sigma,\rho)\sqrt{t-t_0}
    \end{equation}
\end{lemma}
for $t_0\leq t\leq t^\circ$.
\begin{remark}
    \eqref{eq impro grad est noncom outer} implies the second equation in \eqref{eq u C2 sum}.
\end{remark}
\begin{proof}
In the following proof, $C=C(n,\Sigma,\rho)$ is a positive constant depending on $n,\Sigma,\rho$, if there is no other clarifications. $C$ may change from line to line as before. 

Since $u$ satisfies the equation 
\eqref{eq flow outer}, \eqref{eq gra diff ine cut app} holds for any non-negative function $\eta$ independent of $t$. Since $\mC$ is a cone, by Proposition \ref{prop C2 estim outer noncom first}, we have
\begin{equation}\label{eq C2 estim outer noncom}
    \begin{aligned}
     &\max_{(x,\theta)\in [\frac{1}{5}\rho,\infty)\times \Sigma}|u||A|+|\nabla u|\leq \max_{(x,\theta)\in [\frac{1}{5}\rho,\infty)\times \Sigma}|u|x^{-1}+|\nabla u|\leq C,\\
        &\max_{(x,\theta)\in [\frac{1}{5}\rho,\infty)\times \Sigma}|\nabla^2u|\leq \frac{C}{\sqrt{t-t_0}},
    \end{aligned}
    \end{equation}
 if $|t_0|\ll \frac{1}{C(n,\Sigma,\rho)}$. This implies   
 \begin{equation}\label{eq Gijk tilde Vijk norm}
     |G_{ijk}|\leq C,|\tilde V_{ijk}|\leq \frac{C(n,\Sigma,\rho)}{\sqrt{t-t_0}},
 \end{equation}
  on $\mC\setminus B(O,\frac{1}{5}\rho)$, where $G_{ijk}$, $\tilde V_{ijk}$ are the tensor in \eqref{eq gra diff ine cut app}. On the other hand, for $R\geq 2$, we can choose $\eta(x)$ to be a smooth non-negative function so that $\eta=1$ on $(\frac{1}{4}\rho,R-1)$ and supported on $(\frac{1}{5}\rho,R)$, $|\eta_x|^2\eta^{-1}+|\eta_{xx}|\leq C(\rho)$, then
  \begin{equation}\label{eq cut off function norm}
    |\nabla \eta|^2\eta^{-1}+|\nabla^2\eta|\leq C(\rho)  
  \end{equation}
  by \eqref{eq z C2 norm}. Plugging \eqref{eq Gijk tilde Vijk norm}, \eqref{eq cut off function norm} into \eqref{eq gra diff ine cut app}, we get
  \begin{equation}\label{eq gra ine 1}
    \begin{aligned}
        &\frac{\partial \Big(\eta |\nabla u|^2\Big)}{\partial t}
        \leq  \bar g^{ij}(\eta |\nabla u|^2)_{ij}-\bar g^{ij}(2\eta_i |\nabla u|^2_j+\eta_{ij} |\nabla u|^2+2\eta u_{pi}u_{pj})\\
        &-\bar g^{ip}\bar g^{qj}(u_{ij}+V_{ij})\big((\eta|\nabla u|^2)_pu_q+(\eta|\nabla u|^2)_qu_p\big)+\frac{C}{\sqrt{t-t_0}}
    \end{aligned}
    \end{equation}
By Cauchy inequality, for any $\varepsilon>0$ 
\begin{align*}
    -2\eta_i|\nabla u |^2_j=-4\eta_iu_{k}u_{kj}\leq 2C_2(n)( \varepsilon \eta|\nabla^2u|^2+\varepsilon^{-1}|\nabla \eta|^2\eta^{-1}|\nabla u|^2).
\end{align*}
Take $\varepsilon=\frac{1}{C_2(n)}$, then we get 
\begin{align*}
    -\bar g^{ij}(2\eta_i |\nabla u|^2_j+\eta_{ij} |\nabla u|^2+2\eta u_{pi}u_{pj})\leq (2C_2(n)^2|\nabla \eta|^2\eta^{-1}+|\bar g^{ij}|)|\nabla  u|^2\leq C |\nabla u|^2\leq \frac{C}{\sqrt{t-t_0}}.
\end{align*}
by \eqref{eq C2 estim outer noncom} and \eqref{eq cut off function norm}, if $|t_0|<1$. Plugging this into \eqref{eq gra ine 1}, and using maximum principle, we get
    \begin{equation}
        \partial_t(\max_{(x,\theta)\in C} (\eta |\nabla u|^2))\leq \frac{C}{\sqrt{t-t_0}},
    \end{equation}
   or
    \begin{equation}
        \max_{(x,\theta)\in C}(\eta |\nabla u|^2)(t)\leq  \max_{(x,\theta)\in C}(\eta |\nabla u|^2)(t_0)+C\sqrt{t-t_0}.
    \end{equation}
    Likewise, we have
    \begin{equation}
        \min_{(x,\theta)\in C}(\eta |\nabla u|^2)(t)\geq  \min_{(x,\theta)\in C}(\eta |\nabla u|^2)(t_0)-C\sqrt{t-t_0}.
    \end{equation}
    This yields \eqref{eq impro grad est noncom outer}.    
\end{proof}

\begin{lemma}\label{lem imp sec der est outer noncom}
    If $0<\rho\ll 1$ (depending on $n,\Lambda$) and $|t_0|\ll 1$ (depending on $n,\rho$), there holds 
    \begin{equation}\label{eq imp sec der est outer noncom}
        \sup_{(x,\rho)\in [\frac{1}{3}\rho,\infty)\times \Sigma} |\nabla^2_{(x,\theta)}u(x,.\theta,t)|\leq 2\sup_{(x,\rho)\in [\frac{1}{4}\rho,\infty)\times \Sigma}|\nabla^2_{(x,\theta)}u(x,\theta,t_0)|+C(n,\rho,\Sigma)
    \end{equation}
    for $t_0\leq t\leq t^\circ$.
\end{lemma}
\begin{remark}
    \eqref{eq imp sec der est outer noncom} implies the third equation in \eqref{eq u C2 sum}.
\end{remark}
\begin{proof} 
Note that $u$ satisfies \eqref{eq flow outer}, and
$$c_1g_{ij}\leq \bar g_{ij}\leq C_1g_{ij},$$ 
for some $c_1(n,\Sigma,\rho)$, $C_1(n,\Sigma,\rho)>0$ by \eqref{eq gij bar gij equi} on $\mC\setminus B(O,\frac{1}{5}\rho)$, for $t_0\leq t\leq t^\circ$. Moreover, by Lemma \ref{lem imp gra est outer noncom} and \eqref{eq adm x},
\begin{equation}\label{eq imp der outer non com Sigma}
\begin{aligned}
|u|x^{-1}+|\nabla u|\leq 2(  \varepsilon_0+C(n,\Sigma,\rho)\sqrt{t-t_0}).
\end{aligned}
\end{equation}
on $\mC\setminus B(O,\frac{1}{5}\rho)$, for $t_0\leq t\leq t^\circ$. Thus, \eqref{eq sec der diff ine cut app} holds by Lemma \ref{lem sec der diff ine cut app}. If we take $0<\varepsilon=c_1<1$ in \eqref{eq sec der diff ine cut app} and then $|t_0|\ll 1$, such that $$\varepsilon_0+C(n,\Sigma,\rho)\sqrt{t-t_0})\leq 4\varepsilon_0=:\mu<\frac{c_1}{2\sqrt{C(n,\Sigma,c_1,C_1,\rho,\varepsilon)}},$$ 
then \eqref{eq sec der diff ine cut app} implies
\begin{equation}\label{eq eta D2u upper J1234}
\begin{aligned}
    \partial_t(\eta|\nabla^2u|^2)\leq& \bar g^{ij}(\eta|\nabla^2 u|^2)_{ij}-\bar g^{ij}(2\eta_i |\nabla^2 u|^2_j+\eta_{ij}|\nabla^2 u|^2)-c_1^2\eta|\nabla^3u|^2\\
    &-\bar g^{ip}\bar g^{qj}[(\eta |\nabla^2u|^2)_pu_q+(\eta|\nabla^2u|^2)_qu_p-(\eta_p u_q+\eta_q u_p)|\nabla^2 u|^2](u_{ij}+V_{ij})\\
       &-c_1^2\eta|\nabla ^2u|^4+C(n,\Sigma,\rho,\varepsilon)\eta(|\nabla^2 u|^3+1) =:J_1+J_2+J_3,
       \end{aligned}
\end{equation}
on $\mC\setminus B_{\frac{1}{5}\rho}$, for $t_0\leq t\leq t^\circ$. Here
\begin{align*}
    J_1:=&\bar g^{ij}(\eta|\nabla^2 u|^2)_{ij}-\bar g^{ij}(2\eta_i |\nabla^2 u|^2_j+\eta_{ij}|\nabla^2 u|^2)-c_1^2\eta|\nabla^3u|^2\\
   J_2:=&-\bar g^{ip}\bar g^{qj}[(\eta |\nabla^2u|^2)_pu_q+(\eta|\nabla^2u|^2)_qu_p-(\eta_p u_q+\eta_q u_p)|\nabla^2 u|^2](u_{ij}+V_{ij})-c_1^2\eta|\nabla^2u|^4.   \\
   J_3:=&C(n,\Sigma,\rho,\varepsilon)\eta(|\nabla^2u|^3+1)
\end{align*}
By \eqref{eq gij bar gij equi}, \eqref{eq C2 estim outer noncom}, and Cauchy-Schwarz inequality, there exists $C_2(n)\geq 1>c_1$ such that 
\begin{align*}
    J_1
    \leq &\bar g^{ij}(\eta |\nabla^2u|^2)_{ij}+\frac{C_2(n)}{c_1}(2|\nabla\eta||\nabla^3u||\nabla^2u|+|\nabla^2\eta||\nabla^2u|^2)-c_1^2\eta|\nabla^3 u|^2\\
    \leq&\bar g^{ij}(\eta |\nabla^2u|^2)_{ij}+\frac{C_2}{c_1}(\frac{C_2}{c_1^3}|\nabla\eta|^2\eta^{-1} |\nabla^2u|^2+\frac{c_1^3}{C_2}\eta |\nabla^3u|^2)+\frac{C_2}{c_1}|\nabla^2\eta||\nabla^2u|^2-c_1^2\eta|\nabla^3 u|^2\\
    \leq &\bar g^{ij}(\eta |\nabla^2u|^2)_{ij}+\frac{C_2^2}{c_1^4}(|\nabla\eta|^2\eta^{-1}+|\nabla^2\eta|)|\nabla^2u|^2\\
    \leq &\bar g^{ij}(\eta |\nabla^2u|^2)_{ij}+\frac{C_2^2}{c_1^4}(|\nabla\eta|^2\eta^{-\frac{3}{2}}+|\nabla^2\eta|\eta^{-\frac{1}{2}})\eta^{\frac{1}{2}}|\nabla^2u|\frac{1}{\sqrt{t-t_0}}\\
    \leq  &\bar g^{ij}(\eta |\nabla^2u|^2)_{ij}+\frac{C_2^2}{c_1^4}(|\nabla\eta|^2\eta^{-\frac{3}{2}}+|\nabla^2\eta|\eta^{-\frac{1}{2}})(\eta |\nabla^2u|^2+1)\frac{1}{\sqrt{t-t_0}}.
\end{align*}
By \eqref{eq imp der outer non com Sigma}, and Young's inequality, for any $\bar\varepsilon>0$, there holds
\begin{align*}
    J_2 \leq &P_{pq}((\eta|\nabla^2u|^2)_pu_q+(\eta|\nabla^2u|^2)_qu_p)+C_2(n)|\nabla\eta|(|\nabla ^2u|^3+|\nabla^2u|)-c_1^2\eta|\nabla^2u|^4\\
    \leq &P_{pq}((\eta|\nabla^2u|^2)_pu_q+(\eta|\nabla^2u|^2)_qu_p)+C_2(n)\frac{|\nabla\eta|}{\eta^{\frac{3}{4}}}(\bar \varepsilon\eta^{\frac{3}{2}\cdot\frac{4}{3}}|\nabla^2u|^4+\frac{1}{\bar\varepsilon})+\frac{C_2(n)|\nabla \eta|}{\sqrt{t-t_0}}-c_1^2\eta|\nabla^2u|^4,
\end{align*}
with $P_{pq}:=-\bar g^{ip}\bar g^{qj}(u_{ij}+V_{ij})$, and
\begin{align*}
    J_3\leq \frac{1}{2}C(n,\Sigma,\rho,\varepsilon)(\bar \varepsilon\eta|\nabla^2 u|^4+\frac{1}{\bar\varepsilon}\eta)+C(n,\Sigma,\rho,\varepsilon)\eta.
\end{align*}
Suppose 
\begin{equation}\label{eq cut off boun}
    \frac{|\nabla \eta|^2}{\eta^{\frac{3}{2}}}+\frac{|\nabla^2\eta|}{\eta^{\frac{1}{2}}}+\frac{|\nabla\eta|}{\eta^{\frac{3}{4}}}+|\nabla\eta|\leq  C_3(n,\Sigma,\rho),
\end{equation}
and choose $\bar \varepsilon>0$ small so that 
\begin{equation}(C_2(n)C_3(n,\Sigma,\rho)+C(n,\Sigma,\rho,\varepsilon))\bar\varepsilon\leq \frac{1}{2}c_1^2,
\end{equation}
then we are done by maximum principle. In fact, if \eqref{eq cut off boun} holds, then by the choice of $\bar\varepsilon$ and the estimates of $J_1,I_2,J_3$ above, we have 
\begin{align*}
     J_1+J_2+J_3\leq& \bar g^{ij}(\eta|\nabla^2u|^2)_{ij}
   +P_{pq}((\eta|\nabla^2u|^2)_pu_q+(\eta|\nabla^2u|^2)_qu_p)+\frac{C_2^2C_3}{c_1^4}(\eta|\nabla^2 u|^2+1)\frac{1}{\sqrt{t-t_0}}\\
   &+C_2C_3\bar\varepsilon^{-1}+\frac{C_2C_3}{\sqrt{t-t_0}}+\frac{1}{2}C(n,\Sigma,\rho,\varepsilon)\bar\varepsilon^{-1}\eta-\frac{1}{2}c_1^2\eta^2|\nabla^2u|^4+C(n,\Sigma,\rho,\varepsilon)\eta\\
   \leq &\bar g^{ij}(\eta|\nabla^2u|^2)_{ij}
   +P_{pq}((\eta|\nabla^2u|^2)_pu_q+(\eta|\nabla^2u|^2)_qu_p)+\frac{C_2^2C_3}{c_1^4}\eta|\nabla^2 u|^2\frac{1}{\sqrt{t-t_0}}+\frac{C_4(n,\Sigma,\rho,\varepsilon)}{\sqrt{t-t_0}}.
\end{align*}
for some $C_4(n,\Sigma,\rho,\varepsilon)$ large, if $|t_0|\leq 1$. The maximum principle shows that
\begin{equation}
    \max_{(x,\theta)\in[\frac{1}{4}\rho,\infty)\times\Sigma}\eta e^{-C_5t}|\nabla^2u(x,\theta,t)|^2\leq C_4\sqrt{|t_0|}+\max_{(x,\theta)\in[\frac{1}{4}\rho,\infty)\times\Sigma}\eta e^{-C_5t_0}|\nabla^2 u(x,\theta,t_0)|
\end{equation}
for any $t_0\leq t\leq t^\circ$, where $C_5=\frac{C_2^2C_3}{c_1^4}$. By taking $|t_0|\leq 1$ small depending on $C_5$, we can make $e^{C_5(t-t_0)}\leq 2$, and we are done. 

Now we prove \eqref{eq cut off boun}. Note $ |\nabla \eta|^2=\eta'^2$, and by \eqref{eq z C2 norm},
\begin{align*}
    |\nabla^2\eta|^2=&\eta'^2\frac{n-1}{x^2}+(\eta'')^2\leq C(n,\Sigma,\rho)(\eta'^2+(\eta'')^2).
\end{align*}
Thus, to prove \eqref{eq cut off boun}, we only need to prove $\eta'\eta^{-\frac{3}{4}}$ and $\eta''\eta^{-\frac{1}{2}}$ are bounded.
In fact, we can choose a smooth cut-off function $\tilde\eta$ such that $\chi_{[\frac{1}{3}\rho,R_1]}\leq \tilde\eta\leq \chi_{(\frac{1}{4}\rho,R)},$ and $\tilde \eta(x_0)=\tilde\eta'(x_0)=0$ ($x_0=\frac{1}{4}\rho$ or $x_0=R$), and let $\eta=\tilde \eta^4$. We then have $\eta'=4\tilde \eta^3\tilde\eta'$, $\eta''=12\tilde \eta^2\tilde\eta'^2+4\tilde \eta^3\tilde \eta''$. Thus, $\eta'\eta^{-\frac{3}{4}}=4\tilde\eta'\to 0$, and $\eta''\eta^{-\frac{1}{2}}=12\tilde\eta'^2+4\tilde\eta\tilde\eta''\to 0$, as $x\to x_0$. We are done.
\end{proof}
After we get Lemma \ref{lem imp gra est outer noncom}, and Lemma \ref{lem imp sec der est outer noncom}, then we can prove \eqref{eq u sum1} by a scaling and Schauder estimate. So we only give a sketch of them.

\begin{proof}[Proof of \eqref{eq u sum1}]
This follows from the standard regularity theory of parabolic equations, and a change of variable $(R_0,2R_0)\times \Sigma\to (1,2)\times \Sigma, (x,\theta)\mapsto (R_0\tilde x, \theta)$ for any $R_0>0$, and use the equation of $u$. Then the coefficients $g_{ij}$ will be uniformly bounded, and the domain $(R_0,2R_0)\times \Sigma$ will also be bounded in the new coordinates $(\tilde x.\theta)$. Then we can use Lemma \ref{lem imp gra est outer noncom}, \ref{lem imp sec der est outer noncom}, and Schauder theory for parabolic equations to derive \eqref{eq u sum1} as the proof of Proposition 7.4 of \cite{GS}. 
\end{proof}
Similarly, we can prove \eqref{eq u sum2} by using Proposition \ref{pro outer C0} as the proof of Proposition 7.5 of \cite{GS}, and prove \eqref{eq v sum1}, \eqref{eq v sum2}, by using \eqref{eq v C0 sum1}, \eqref{eq v C0 sum2} as the proof of Proposition 7.6 of \cite{GS}. So we omit them.
\subsection{$C^2$ estimates in the inner region}
\begin{lemma}\label{lem sec der est in intermedi reg hat w}
   If $\beta\gg1$ (depending on $n,\Sigma,\Lambda$) and $\tau_0\gg 1$ (depending on $n,\Sigma,\Lambda,\rho,\beta$), there holds
   \begin{equation}\label{eq gra est in intermedi reg hat w}
   z|\nabla_{(z,\theta)}\hat w(z,\theta,\tau)|\leq C(n,\Sigma,l,\Lambda)z^\alpha,
   \end{equation}
   \begin{equation}\label{eq sec der est in intermedi reg hat w}
       z^2|\nabla^2_{(z,\theta)}\hat w(z,\theta,\tau)|\leq C(n,\Sigma,l,\Lambda)z^\alpha,
   \end{equation}
   for $(z,\theta)\in [2\beta,\frac{1}{2}(2\sigma_l\tau)^{\frac{1}{2}(1-\vartheta)}]\times \Sigma$, $\tau_0\leq \tau\leq\tau^\circ$.
\end{lemma}
\begin{proof}
    This can be derived from the admissible condition and the asymptotics of $\psi_k$ in \eqref{eq psik asym}. In fact, by \eqref{eq adm z}, we have the profile function $w$ of $\hat F$ over $\mC$ satisfies
    \begin{equation}
        z^{|\gamma|}|\nabla^\gamma w(z,\theta,\tau)|\leq \Lambda(z^\alpha+\frac{z^{2\lambda_l+1}}{(2\sigma_l\tau)^l})\leq C(n,\Sigma,\Lambda)z^\alpha,\quad |\gamma|\in\{0,1,2\}
    \end{equation}
    for $(z,\theta)\in [\beta, \frac{1}{2}(\sigma_l\tau)^{\frac{1}{2}(1-\vartheta)}]\cap\mC$, $\tau_0\leq \tau \leq\tau^\circ$. By definition, $\hat w(z,\theta,\tau)$ is the distance of $\hat F$ to $S_{\kappa,+}$. Since $\lambda_+(\tau)\in(\frac{1}{2},2)$, and $\kappa\approx 1$ by \eqref{eq bound k diff} for $|t_0|\ll1$, we have
    \begin{equation}
        z^{|\gamma|}|\nabla^\gamma\hat w(z,\theta,\tau)|\leq z^i|\nabla^i(w(z,\theta,\tau)-k\psi(z,\theta))|\leq C(n,\Sigma)z^\alpha, \quad |\gamma|\in \{0,1,2\}
    \end{equation}
   by \eqref{eq psik der asy}, for $(z,\theta)\in [2\beta, \frac{1}{2}(\sigma_l\tau)^{\frac{1}{2}(1-\vartheta)}]\cap\mC$, $\tau_0\leq \tau\leq \tau^\circ$.
\end{proof}
\subsection{$C^2$ estimates in the tip region}
Recall that $\hat w(z,\theta,\tau)$ is the profile function of $\hat F$ over $S_{\kappa,+}$. First, we use maximum principle and the equation for $|\nabla \hat w|$ to prove the gradient estimate.
\begin{proposition}\label{prop tip C1}
    If $\beta\gg 1$ (depending on $n,\Lambda$), there holds
    \begin{equation}\label{eq tip C1 est}
        |\nabla_{(\tilde z,\theta)}\hat w(\tilde z,\theta,\tau)|\leq C(n,\Sigma,l,\Lambda)
    \end{equation}
    for $(\tilde z,\theta)\in S_{\kappa,+,\beta^2}$, $\tau_0\leq \tau\leq \tau^\circ$. Here $\nabla_{(\tilde z,\theta)}$ is the covariant derivative on $S_{\kappa,+}$.
\end{proposition}
\begin{remark}
    \eqref{eq tip C1 est} is the second inequality in \eqref{eq hat w sum}. 
\end{remark}
\begin{proof}
    By \eqref{eq ini tip reg}, and Lemma \eqref{lem sec der est in intermedi reg hat w},
    \begin{equation}
        \max_{(\tilde z,\theta)\in [0,2\beta^2]\times \Sigma} \{| \hat w(\tilde z,\theta,\tau_0)|\beta^{\frac{{\tilde \alpha} }{4}},|\nabla_{(\tilde z,\theta)} \hat w (\tilde z,\theta,\tau_0)|\}\leq C_0(n,\Sigma,l,\Lambda), 
    \end{equation}
   for some $C_0(n,\Sigma,l,\Lambda)>0$.
   This implies that
    \begin{equation}
    \begin{aligned}
        \frac{1}{C_1}g_{ij}\leq \hat{g}_{ij}(t_0)\leq C_1g_{ij},\\
        \nu(\tilde z,\theta)\cdot\hat\nu(\tilde z,\theta,t_0)\geq  h_1(C_1)
    \end{aligned}
    \end{equation}
    for some $C_1=C_1(C_0)>0$, for any $(\tilde z,\theta)\in S_{\kappa,+,2\beta^2}$ by \eqref{eq bar gij u}, \eqref{eq nor C nor gra eq}, where $h_1(s):\mathbb R_+\to\mathbb R_+$ is a positive decreasing function, and $_1 h(s)\to 0$ as $s\to \infty$. 
 
 Fix a point $p^*:=(\tilde z^*,\theta^*)\in S_{\kappa,+,2\beta^2}$, let $\pi_{p^*}=T_{p^*}S_{\kappa,+}$ be the tangent plane to $S_{\kappa,+}$ at $p^*$. Let $q^*=F(p^*,\tau_0),\,\nu^*=\nu_{S_{\kappa,+}}(p^*),\,\hat\nu^*=\hat\nu(q^*)$. Then for any $( \tilde z,\theta)\in B(p^*,\varepsilon(C_1))\cap S_{\kappa,+,2\beta^2}$, we have
 \begin{equation}
     \nu^*\cdot\nu(\tilde z,\theta)\geq \frac{99}{100}
 \end{equation}
 for $\varepsilon(C_1)>0$ small. Thus, 
 \begin{equation}
     \hat\nu(\tilde z,\theta,\tau_0)\cdot\nu^*\geq \frac{1}{2}\hat h(C_1).
 \end{equation}
and a neighborhood of $\hat F(p^*,\tau_0)$ can be written as a graph over a small ball $B^n(p^*,\bar \varepsilon(C_1))\subset \pi_{p^*}$ for some $\bar\varepsilon\leq \varepsilon$. By abuse of notation, we set $p^*$ as an origin and use $z=(z_1,\cdots,z_n)$ as coordinates for $\pi_{p^*}$. Since $\hat F$ evolves by \eqref{eq rescaled hat graph}, if we part of $\hat F$ can be written as a graph over $B^n(p^*,\bar \varepsilon)\subset \pi_{p^*}$ with profile function $f$, then $f$ evolves by
\begin{equation}
    f_\tau=c_l\tau^{-1}(-f_iz_i+f)+(\delta_{ij}-\frac{f_if_j}{1+|\nabla f|^2})f_{ij}.
\end{equation}
Differentiating this equation with respect to $z_l$, we get
\begin{align*}
    f_{\tau l}=c_l\tau^{-1}(f_l-f_{il}z_i-f_i\delta_{il})+(-\frac{f_{il}f_j+f_if_{jl}}{1+|\nabla f|^2}+\frac{2f_if_jf_kf_{kl}}{(1+|Df|^2)^2})f_{ij}+(\delta_{ij}-\frac{f_if_j}{1+|\nabla f|^2})f_{ijl},   
\end{align*}
and thus
\begin{equation}
   |\nabla f|^2_\tau=2f_lf_{\tau l}=c_l\tau^{-1}(-z_i|\nabla f|^2_i)+Q_i|\nabla f|^2_i+(\delta_{ij}-\frac{f_if_j}{1+|\nabla f|^2})(|\nabla f|^2_{ij}-2|\nabla^2 f|^2),
\end{equation}
where $Q_i=Q_i(\nabla f,\nabla^2f)$ is some smooth function of $\nabla f,\nabla^2f$. 

Let $\tau'\in [\tau_0, \tau^\circ]$ be the maximal time for which we can write $\hat F$ as a graph over  $B^n(p^*,\bar\varepsilon)\subset \pi_{p^*}$ for all $p^*\in S_{\kappa,+,\beta^2}$ and $\tau_0\leq\tau\leq \tau'$. Then $\tau'>\tau_0$ if we choose $\bar \varepsilon$ small enough. For each $p^*$, define $M_{p^*}:=\max_{[\tau_0, \tau']\times B^n(p^*,\bar \varepsilon)}|\nabla f|$. By maximum principle, $M_{p^*}$ is attained at some boundary point $(z,\theta,\tau'')\in (\partial B^n(p^*,\varepsilon)\times[\tau_0,\tau'])\cup (B^n(p^*,\bar \varepsilon)\times\{\tau_0\})$. By a covering argument, we have 
\begin{equation}\label{eq gradient es tip tan}
   \max_{p^*\in S_{\kappa,+,\beta^2}} M_{p^*}\leq h_2(\min\{ h_1(C_1),c_2\})
\end{equation}
where $c_2=\min_{(S_{\kappa,+,2\beta^2}\setminus S_{\kappa,+,\beta^2})\times [\tau_0,\tau^\circ]}\{\nu\cdot\hat \nu \}=h_2(C_2)>0$, $C_2=\min_{S_{\kappa,+,2\beta^2}\setminus S_{\kappa,+,\beta^2}\times [\tau_0,\tau^\circ]}\{|\nabla\hat w|\}<\infty$ by \eqref{eq gra est in intermedi reg hat w}, and $h_2(s):\mathbb R_+\to\mathbb R_+$ is a positive decreasing function, and $h_2(s)\to \infty$ as $s\to 0^+$. Thus, we can take $\bar \varepsilon>0$ small enough such that $\tau'=\tau^\circ$ and \eqref{eq gradient es tip tan} holds for $M_{p^*}=\max_{[\tau_0, \tau^\circ]\times B(p^*,\bar \varepsilon)}|\nabla f|$. Moreover, we get from \eqref{eq gradient es tip tan} that
\begin{equation}
    \max_{[\tau_0, \tau^\circ]\times S_{\kappa,+,\beta^2}} |\nabla f|(p^*)\leq  \max_{p^*\in S_{\kappa,+,\beta^2}} M_{p^*}\leq h_2(\min\{ h_1(C_1),c_2\}).
\end{equation}
and
\begin{equation}\label{eq tip normal prod}
   \inf_{p^*\in S_{\kappa,+,\beta^2}} \nu^*\cdot \hat\nu(p^*)=\frac{1}{\sqrt{1+|\nabla f|(p^*)^2}}\geq \frac{1}{\sqrt{1+C_3^2}}
\end{equation}
for $\tau\in[\tau_0,\tau^\circ]$. On the other hand, by \eqref{eq nor C nor gra eq},
$\nu\cdot\hat\nu\to 0$ if $|\nabla\hat w|\to\infty$.
Thus,
\begin{equation}
    |\nabla\hat w|(\tilde z,\theta,\tau)\leq C(C_3)\text{ on }S_{\kappa,+,\beta^2}\times [\tau_0,\tau^\circ].
\end{equation}
\end{proof}
Now we use maximum principle and the equation for $|\hat A|^2$ to prove the estimates for curvature for $\tau\in [\tau_0,\tau_0+\delta]$.
\begin{lemma}\label{lem tip sec func est short time}
If $\beta\gg 1$ (depending on $n,\Sigma,\Lambda$), then there is $\delta>0$ (depending on $n,\Sigma$) so that the second fundamental form of $\hat F$ satisfies 
\begin{equation}\label{eq tip sec func est short time}
    \max_{\hat F\cap B(O,3\beta)}|A_{\hat F }|\leq C(n,\Sigma,l). 
\end{equation}
for $\tau_0\leq \tau\leq\min\{\tau_0+\delta,\tau^\circ\}$. In particular, there holds
\begin{equation*}
    |\nabla^2_{(\tilde z,\theta)}\hat w(\tilde z,\theta,\tau)|\leq C(n,\Sigma,l)
\end{equation*}
for $(\tilde z,\theta)\in \cap S_{\kappa,+,3\beta}$, $\tau_0\leq \tau\leq \min\{\tau_0+\delta,\tau^\circ\}$.
\end{lemma}
\begin{proof}
    The second fundamental form $|A|^2$ of $F$ evolves by
    \begin{equation}
        \partial_t|A|^2=\Delta_{F}|A|^2+2|A|^4-2|\nabla_FA|^2.
    \end{equation}
    Since $\hat F=\frac{1}{|t|^{\frac{1}{2}+\sigma_l}}F$, and $t=-(2\sigma_l\tau)^{-\frac{1}{2\sigma_l}}$, the second fundamental form $|\hat A|^2$ of $\hat F$ satisfies  
    \begin{equation}\label{eq rescaled hat A}
    \begin{aligned}
        \partial_\tau|\hat A|^2
        =&\Delta_{\hat F}|\hat A|^4+2|\hat A|^2-2|\nabla_{\hat F}\hat A|^2-\frac{1+2\sigma_l}{2\sigma_l\tau}|\hat A|^2.
    \end{aligned}
    \end{equation}
    Following the same argument as in Lemma 7.10 of \cite{GS}, we have
    \begin{equation}
        \max_{\hat F_\tau\cap B(O,3\beta)}|\hat A|^2\leq 2C
    \end{equation}
    for $\tau_0\leq \tau\leq\min\{\tau_0+\delta,\tau^\circ\}$, for some $\delta=\delta(n,\Sigma,l)$, where $C=|\hat A_{\tau}|^2_{\max}(\tau_0)+\max_{Z_\tau\in\hat F_\tau,|Z_\tau|=3\beta}|\hat A_\tau(Z_\tau)|^2\leq C(n,\Sigma,l)$. The second conclusion follows from \eqref{eq tip inner C0 est hat}, \eqref{eq tip C1 est}, 
    \eqref{eq tip sec func est short time}, \eqref{eq bar gij u}, and \eqref{eq bar hij u}, with $X=S_{\kappa,+}$, and $u$ replaced by $\hat w$.
\end{proof}

At last, we use the results from the previous steps and \cite{EH}'s gradient and curvature estimate to prove second derivative estimates for $\tau\in[\tau_0,\tau^\circ)$.
\begin{proposition}\label{prop tip C2}
    If $\beta\gg 1$ (depending on $n,\Sigma,\Lambda$), there holds 
    \begin{equation}\label{eq hat w C2 tip}
        |\nabla^2_{(\tilde z,\theta)} \hat w(\tilde z,\theta,\tau)|\leq C(n,\Sigma,l)
    \end{equation}
    for $(\tilde z,\theta)\in S_{\kappa,+,3\beta}$, $\tau_0\leq\tau\leq \tau^\circ$.
\end{proposition}
\begin{remark}
    \eqref{eq hat w C2 tip} is the third inequality in \eqref{eq hat w sum}. 
\end{remark}
\begin{proof}
    By Proposition \ref{prop tip C1}, there is $\delta(n,\Sigma,l)$ so that
    \begin{equation}
        |\nabla^2\hat w(\tilde z,\theta,\tau)|\leq C(n,\Sigma,l)
    \end{equation}
    for $(\tilde z,\theta)\in  S_{\kappa,+,3\beta}$, $\tau_0\leq \min\{\tau_0+\delta,\tau^\circ\}$. Hence, to prove the lemma, we only need to consider the case $\tau^\circ-\tau>\delta$. Fix $\tau_0+\delta\leq \tau_*\leq \tau^\circ$, and let
    \begin{equation}
        \bar{F}_\iota=(2\sigma_l\tau_*)^{c_l}F_{-(2\sigma_l\tau_*)^{\frac{-1}{2\sigma_l}}(1-\frac{\iota}{2\sigma_l\tau_*})}
    \end{equation}
    where $c_l=\frac{1}{2}+\frac{1}{4\sigma_l}$ is the same constant as before. Then $\bar F_\iota$ defines a MCF for $-(2\sigma_l\tau_*)\left((\frac{\tau_*}{\tau_0})^{\frac{1}{2\sigma_l}}-1\right)\leq \iota\leq 0$. Note that 
    $$\bar F_0=(2\sigma_l\tau_*)^{c_l}F_{-(2\sigma_l\tau_*)^{\frac{-1}{2\sigma_l}}}=\hat F_{\tau_*}$$
and
\begin{equation}\label{eq time}
    (2\sigma_l\tau_*)\left((\frac{\tau_*}{\tau_0})^{\frac{1}{2\sigma_l}}-1\right)\geq\frac{\delta}{2}
\end{equation}
provided $\tau_0\gg 1$ (depending on $n,\Sigma,l$). By \eqref{eq type ii rescaling graph}, $(2\sigma_l\tau_*)^{\frac{-1}{2\sigma_l}}(1-\frac{\iota}{2\sigma_l\tau_*})=|t|=(2\sigma_l\tau)^{-\frac{1}{2\sigma_l}}$, thus
    \begin{equation}
   \tau=\frac{\tau_*}{(1-\frac{\iota}{2\sigma_l\tau_*})^{2\sigma_l}}. 
    \end{equation}
Since $F$ is admissible, by rescaling, we can write $\bar F_\iota$ as a graph over $(1-\frac{\iota}{2\sigma_l\tau_*})^{2\sigma_lc_l}S_{\kappa,+}(\tilde z,\theta)$ with profile function $\hat h(z,\theta,\iota)$. That is
\begin{equation}
\begin{aligned}
    \bar F_\iota=&(1-\frac{\iota}{2\sigma_l\tau_*})^{2\sigma_lc_l}\hat F_\tau=(1-\frac{\iota}{2\sigma_l\tau_*})^{2\sigma_lc_l}(S_{\kappa,+}(\tilde z,\theta)+\hat w(z,\theta,\frac{\tau_*}{(1-\frac{\iota}{2\sigma_l\tau_*})^{2\sigma_l}})\nu(z,\theta)).
\end{aligned}
\end{equation}
Let $c(\iota)=(1-\frac{\iota}{2\sigma_l\tau_*})^{2\sigma_lc_l}$. Since $|S_{\kappa,+}|\geq c_2(\Sigma)>0$ by \eqref{eq bound k diff}, and $|\hat w(z,\theta,\tau)|\leq C_2(n,\Sigma)\beta^{-\frac{{\tilde \alpha} }{4}}$ by \eqref{eq tip inner C0 est hat} for $|\tilde z|\leq 5\beta$ for some $C_2\geq c_2>0$, we have
\begin{equation}\label{eq bar F ciota quo}
   \frac{|\bar F(\tilde z,\theta,\iota)|}{c(\iota)|S_{\kappa,+}(\tilde z,\theta)|}=\frac{|S_{\kappa,+}(\tilde z,\theta)+\hat w(\tilde z,\theta,\tau)\nu(z,\theta)|}{|S_{\kappa,+}(\tilde z,\theta)|}\in (1-\delta_1,1+\delta_1),
\end{equation}
and
\begin{equation}\label{eq bar F ciota diff quo}
\frac{|\bar F(\tilde z,\theta,\iota)-c(\iota)S_{\kappa,+}(\tilde z,\theta)|}{c(\iota)|S_{\kappa,+}(\tilde z,\theta)|}=\frac{|\hat w(\tilde z,\theta,\tau)|}{|S_{\kappa,+}(\tilde z,\theta)|}\leq\delta_1
\end{equation}
where $\delta_1:=\frac{C_2\beta^{-\frac{{\tilde \alpha} }{4}}}{c_2}\ll 1$, if $\beta\gg 1$ large depending on $c_2,C_2$. 

If $|\bar F(\tilde z,\theta,\iota)-\bar F(\tilde z_*,\theta_*,0)|\leq \varepsilon_1$, then by \eqref{eq bar F ciota quo},
\begin{equation}\label{eq bar F bar Fstar diff quo}
    \frac{|\bar F(\tilde z,\theta,\iota)-\bar F(\tilde z_*,\theta_*,0)|}{|\bar F(\tilde z_*,\theta_*,0)|}=\frac{\varepsilon_1}{|\bar F(\tilde z_*,\theta_*,0)|}\leq \frac{\varepsilon_1}{(1-\delta_1)|S_{\kappa,+}(\tilde z^*,\theta^*)|}\leq \frac{\varepsilon_1}{(1-\delta_1)c_2}.
\end{equation}
Now, by triangle inequality,
    \begin{equation}\label{eq Q split}
    Q:=\frac{|c(\iota)S_{\kappa,+}(\tilde z,\theta)-S_{\kappa,+}(\tilde z^*,\theta^*)|}{|S_{\kappa,+}(\tilde z^*,\theta^*)|}\leq Q_1+Q_2+Q_3
    \end{equation}
    where
    \begin{align*}
    Q_1:=\frac{|\bar F(\tilde z,\theta,\iota)-\bar F(\tilde z_*,\theta_*,0)|}{|S_{\kappa,+}(\tilde z^*,\theta^*)|},\quad Q_2:=\frac{|\bar F(\tilde z,\theta,\iota)-c(\iota)S_{\kappa,+}(\tilde z,\theta)|}{|S_{\kappa,+}(\tilde z^*,\theta^*)|},\quad Q_3:=\frac{|\bar F(\tilde z_*,\theta_*,0)-S_{\kappa,+}(\tilde z^*,\theta^*)|}{|S_{\kappa,+}(\tilde z^*,\theta^*)|}.
    \end{align*}
    By \eqref{eq bar F ciota quo} and \eqref{eq bar F bar Fstar diff quo},
    \begin{equation}
        Q_1\leq \frac{|\bar F(\tilde z,\theta,\iota)-\bar F(\tilde z_*,\theta_*,0)|}{|\bar F(\tilde z_*,\theta_*,0)|}(1+\delta_1)\leq \frac{\varepsilon_1}{(1-\delta_1)c_2}(1+\delta_1).
    \end{equation}
    By \eqref{eq bar F ciota diff quo}, and note $c(0)=1$, we have
    \begin{equation}
        Q_3\leq \delta_1.
    \end{equation} 
    At, last, by \eqref{eq bar F ciota diff quo} again and triangle inequality,
    \begin{equation}
        Q_2=\frac{|\bar F(\tilde z,\theta,\iota)-c(\iota)S_{\kappa,+}(\tilde z,\theta)|}{|c(\iota)S_{\kappa,+}(\tilde z,\theta)|}\frac{|c(\iota)S_{\kappa,+}(\tilde z,\theta)|}{|S_{\kappa,+}(\tilde z^*,\theta^*)|}\leq \delta_1(1+\frac{|c(\iota)S_{\kappa,+}(\tilde z,\theta)-S_{\kappa,+}(\tilde z^*,\theta^*)|}{|S_{\kappa,+}(\tilde z^*,\theta^*)|})=\delta_1(1+Q).
    \end{equation}
    Plugging the above three inequalities into \eqref{eq Q split}, we get
    \begin{equation}
        Q\leq \frac{\varepsilon_1}{(1-\delta_1)c_2}(1+\delta_1)+\delta_1(1+Q)+\delta_3.
    \end{equation}
Taking $\beta>0$ large (depending on $c_2,C_2$), $\varepsilon_1>0$ small, and rearranging the terms, we get
\begin{align*}
   Q\leq (1-\delta_1)^{-1}[\frac{\varepsilon_1}{(1-\delta_1)c_2}(1+\delta_1)+2\delta_1]\leq C(\Sigma)(\varepsilon_1+\beta^{-\frac{{\tilde \alpha} }{4}}).
\end{align*}
Thus
\begin{align*}
    \nu(S_{\kappa,+}(\tilde z,\theta))\cdot \nu(\tilde z^*,\theta^*)\geq 1-\varepsilon_2.
\end{align*}
for some $\varepsilon_2>0$ small if $\beta>0$ large and $\varepsilon_1>0$ small. On the other hand, there exists $\delta_2>0$ such that
\begin{align*}
    \bar\nu(\bar F(\tilde z,\theta,\iota))\cdot \nu (S_{\kappa,+}(\tilde z,\theta))=\bar \nu(S_{\kappa,+}(\tilde z,\theta)+\hat w(\tilde z,\theta,\frac{\tau_*}{(1-\frac{\iota}{2\sigma_l\tau_*})^{2\sigma_l}})\nu(\tilde z,\theta))\cdot \nu (S_{\kappa,+}(\tilde z,\theta))\geq \delta_2>0\\
    \bar \nu (\bar F(\tilde z^*,\theta^*,0))\cdot \nu(S_{\kappa,+}(\tilde z^*,\theta^*))=\bar \nu(S_{\kappa,+}(\tilde z^*,\theta^*)+\hat w(\tilde z^*,\theta^*,\tau^*)\nu(\tilde z^*,\theta^*))\cdot \nu(S_{\kappa,+}(\tilde z^*,\theta^*))\geq \delta_2>0
\end{align*}
for $|\tilde z|\leq 2\beta$ by \eqref{eq nor C nor gra eq} since $|\hat w(\tilde z,\theta,\tau)|\leq C\beta^{-\frac{{\tilde \alpha} }{2}}\leq \varepsilon_3, |\nabla \hat w(\tilde z,\theta,\tau)|\leq C$ by \eqref{eq tip inner C0 est hat}, \eqref{eq tip C1 est} for $\beta>0$ large. Thus
\begin{equation}
    \bar \nu (\bar F(\tilde z^*,\theta^*,0))\cdot \bar\nu(\bar F(\tilde z,\theta,\iota))
    \geq \frac{\delta_2^2}{2}>0
\end{equation}
if $|\tilde z^*|\leq 2\beta$ and $|\bar F(\tilde z,\theta,\iota)-\bar F(\tilde z_*,\theta_*,0)|\leq \varepsilon_1$ for $\beta>0$ large and $\varepsilon_1>0$ small. Thus by the curvature estimate in \cite{EH} and \eqref{eq time}, we obtain
\begin{equation}
    |\bar A(\tilde z^*,\nu^*,0)|=|\hat A (\tilde z^*,\theta^*,\tau^*)|\leq C(\delta_2)(\sqrt{\frac{2}{\delta}}+\frac{1}{\varepsilon_1}).
\end{equation}
Since $(\tilde z^*,\theta^*,\tau^*)\in S_{\kappa,+,3\beta}\times [\tau_0,\tau^\circ]$ is arbitrary, we are done.
\end{proof}
\begin{proof}[Proof of \eqref{eq hat w sum}]
    Use the $C^0$ estimate (Proposition \eqref{prop tip C0}), $C^1$ estimate (Proposition \ref{prop tip C1}) and $C^2$ estimate (Proposition \ref{prop tip C2}), and the standard theory of parabolic equations as the proof of Proposition 7.12 in \cite{GS}.
\end{proof}
\subsection{Determination of $\Lambda$}
We have to prove that we can find a $\Lambda>0$ depends only on $n,\Sigma,l$ such that \eqref{eq hat w sum} holds. This can be done in the same way as in \cite{GS}, which use the interior estimate of last subsection, maximum principle and initial values to extend the estimates to the initial time. Then we can choose $|t_0|\ll 1$ to achieve this.
\appendix

\section{Parametrization of normal graphs}\label{sec graph evo}
In this appendix, we collect some basic formulas for normal graphs. The readers can refer to section 2 of \cite{CHL} for details, we give them here for completeness. Let $X\subset\mathbb R^{n+1}$ be a smooth hypersurface embedded in $\mathbb R^{n+1}$, $\nu$ be a unit normal vector of $X$. Suppose $F$ is a nomrmal graph over a surface $X$ with profile function $u$, that is
$$F(x)=X(x)+u(x)\nu(x),$$ 
where $x$ is the local coordinates on $X$, and $u$ is a smooth function on $X$. In the following of this section, we use index $i,j,k,l$ to denote the covariant differentiation on $X$ with respect to $x$, and use $\bar{}$ to denote quantities for $F$. For example, we use $g_{ij},\nu,h_{ij}$ to denote the metric, unit normal vector, and the second fundamental form of $X$, and use $\bar g_{ij},\bar\nu,\bar h_{ij}$ to denote that of $F$ respectively. Then
\begin{align*}
    F_i=&X_i+uh_i^kX_k+u_i\nu,\quad
    F_{ij}
    =(-h_{ij}-uh_i^kh_{kj}+u_{ij})\nu+(u_ih_j^k+u_jh_i^k+u(h_i^k)_j)X_k.
    \end{align*}
Thus, the metric $\bar g$ on $F$ and $g$ on $X$ has the form
\begin{equation}\label{eq bar gij u}
\bar{g}_{ij}=F_i\cdot F_j=g_{ij}+2uh_{ij}+u^2h_i^kh_j^lg_{kl}+u_iu_j.
\end{equation}
Using the equation $\bar{\nu}\cdot F_i=0$, we get 
\begin{align*}
    \bar{\nu}=&\frac{\nu-u_i\bar{g}^{ij}F_j}{|\nu-u_i\bar{g}^{ij}F_j|}.
\end{align*}
Note we can define $V^2:=1-u_iu_j\bar g^{ij}=(\nu-u_i\bar{g}^{ij}F_j)\cdot(\nu-u_k\bar g^{kl}F_l)\geq 0$, then we get
\begin{equation}\label{eq nor C nor gra eq}
    \nu\cdot\bar{\nu}=\frac{1-u_iu_j\bar{g}^{ij}}{|\nu-u_i\bar{g}^{ij}F_j|}=(1-u_iu_j\bar g_{ij})^{\frac{1}{2}}=V.
\end{equation}
and 
\begin{equation}\label{eq bar hij u}
    \bar{h}_{ij}=-F_{ij}\cdot\bar{\nu}=\frac{1}{|\nu-u_i\bar{g}^{ij}F_j|}\Big[(1-u_iu_j\bar{g}^{ij})(h_{ij}+uh_i^kh_{kj}-u_{ij})+(u_ih_j^k+u_jh_i^k+u(h_i^k)_j)u_m\hat{g}^{mn}(g_{kn}+uh_{nk})\Big].
\end{equation}
Let $\bar{H}=\bar{g}^{ij}\bar{h}_{ij}$ be the mean curvature of $F$, and
\begin{equation}\label{eq mathcal Lcu}
    \mathcal L u=\Delta u+|A|^2u,
\end{equation}
to be the Jacobi operator of $X$, where $\Delta ,|A|^2=g^{ik}g^{jl}h_{ij}h_{kl}$ are the Laplacian operator and length square of the second fundamental form of $X$. We have
\begin{equation}\label{eq error outer}
\begin{aligned}
    E(u):=&-\frac{\bar H}{\nu\cdot\bar\nu}-\mathcal Lu=\bar{g}^{ij}\Big[-\frac{u_m\bar{g}^{mn}}{1-u_iu_j\bar{g}^{ij}}(u_ih_j^k+u_jh_i^k+u(h_i^k)_j)(g_{kn}+uh_{nk})\Big]\\
    &+(\bar{g}^{ij}-g^{ij})(u_{ij}-uh_i^kh_{kj})+(\bar{g}^{ij}-g^{ij}+2ug^{im}h_{mn}g^{nj})(-h_{ij})
    \end{aligned}
    \end{equation}
 Now, let's state a lemma about the structure of $E(u)$, which is crucial for our analysis of our solution in Section \ref{sec C0 est} and Section \ref{sec higher est}.
\begin{lemma} \label{lem E stru}
Suppose $c_1 g_{ij}\leq \bar g_{ij}\leq C_1g_{ij}$, $V\geq c_1>0$ for some uniform constants $C_1\geq c_1>0$, where $V=\nu\cdot\bar \nu$ is defined in \eqref{eq nor C nor gra eq}.

(1) If $X$ is a smooth hypersurface, $|h|+|\nabla h|\leq M$, $\|u\|_{C^2(X)}\leq \mu$ for some constant $M,\mu>0$, then there holds
    \begin{equation}\label{eq E stru 1}
    |E(u)|\leq C(n,M,c_1,C_1)\mu^2.
\end{equation} 
(2) If $X=\mC=\{(r,\theta)|r\in\mathbb R_+,\theta\in \Sigma\}$ is a regular cone in $\mathbb R^{n+1}$ ($\Sigma=\mC\cap \mathbb S^{n}$ is the link of $\mC$ which is a smooth hypersurface of $\mathbb S^n$), and 
\begin{equation}\label{eq basic assum}
    |u(r,\theta)|r^{-1}+|\nabla u(r,\theta)|+|\nabla^2u(r,\theta)|r\leq \mu
\end{equation}
for some uniform constant $\mu>0$, then there is a constant $\varepsilon(n,\Sigma)\ll1$, $C(n)\gg 1$ such that if $\mu\leq \varepsilon$, there holds 
\begin{equation}\label{eq E stru 2}
    |E(u)|\leq C(n,\Sigma,c_1,C_1)r^{-1}\mu^2
\end{equation}
\end{lemma}
\begin{proof}
    (1) \eqref{eq E stru 1} Follows directly from the expression of $E(u)$, \eqref{eq bar gij u}, and the formula
    \begin{equation}\label{eq inver matrix der}
        \frac{da^{ij}(s)}{ds}=-a^{ik}\frac{d a_{kl}(s)}{ds}a_{lj}
    \end{equation}
for any smooth one parameter invertible matrix $\{a_{ij}(s)\}$ of $s$, where $s$ is a parameter.

(2) Note that, for a cone $\mC$, we have $|\nabla^k h|\leq C(\Sigma,k)r^{-k-1}$. Thus we have
\begin{align*}
    &\left |\bar{g}^{ij}\Big[-\frac{u_m\bar{g}^{mn}}{1-u_iu_j\bar{g}^{ij}}(u_ih_j^k+u_jh_i^k+u(h_i^k)_j)(g_{kn}+uh_{nk})\Big]\right |\\
    \leq& C(n,c_1,C_1)|\nabla u|(|\nabla u||h|+|u||\nabla h|)(1+u|h|)\\
    \leq&C(n,\Sigma,c_1,C_1) \mu(|\nabla u| r^{-1}+|u|r^{-2})(1+u|r|)\leq C(n,\Sigma,c_1,C_1)r^{-1}\mu^2
\end{align*}
if $u$ satisfies \eqref{eq basic assum}. The estimates of other terms in $E(u)$ follows similarly via using \eqref{eq bar gij u} and \eqref{eq inver matrix der}.
\end{proof}
We compute the derivatives of $\bar g^{ij}$ for the purpose of $C^1,C^2$ estimates in the outer region.
\begin{lemma}\label{lem bar g der sec der u}
    The inverse metric $\bar g^{ij}$ satisfies
    \begin{equation}\label{eq bar gij deri u}
       \bar g^{ij}_k=-\bar g^{ip}\bar g^{qj}(u_{pk}u_q+u_pu_{qk})+G_{pqk}(g,h,\bar g,u,\nabla u)
\end{equation}
with $G_{pqk}(g,h,\bar g,u,\nabla u)=-\bar g^{ip}\bar g^{qj}(2u_kh_{pq}+2uh_{pqk}+(u^2h_p^mh_q^n)_kg_{mn})$, and 
 \begin{equation}\label{eq bar gij sec deri u}
    \begin{aligned}
     \bar g^{ij}_{kl}
    =&-\bar g^{ip}\bar g^{qj}(u_{pkl}u_q+u_{pk}u_{ql}+u_{pl}u_{qk}+u_{p}u_{qkl})+(\bar g^{im}\bar g^{np}\bar g^{qj}+\bar g^{ip}\bar g^{qm}\bar g^{nj})\big(u_{ml}u_n+u_mu_{nl}\big) \big(u_{pk}u_q+u_pu_{qk}\big)\\
    &+ G^{ij}_{1kl}(h,\nabla h,\bar g,u,\nabla u)+ G^{ij}_{2kl}(h,\nabla h,\bar g,u,\nabla u)* \nabla^2u
\end{aligned}
\end{equation}
for some tensor $G^{ij}_{mkl}$ $(m=1,2)$ in $h,\nabla h,\bar g,u, \nabla u$. 

Moreover, if $c_1g_{ij}\leq \bar g_{ij}\leq C_1g_{ij}$, for some uniform constants $C_1\geq c_1>0$, $X=\mC=\{(r,\theta)|r\in\mathbb R_+,\theta\in \Sigma\}$ is a regular cone in $\mathbb R^{n+1}$, and $u$ satisfies 
\begin{equation}\label{eq assum u and grad}
    |u|r^{-1}+|\nabla u|\leq \mu,\quad (r,\theta)\in [r_0,\infty)\times \Sigma
\end{equation}
for some $\mu,r_0>0$, then
\begin{equation}\label{eq G Gm structure}
    |G|^2+|G_1| \leq C(n,\Sigma,c_1,C_1)(r^{-1}\mu)^2, \quad |G_2|\leq C(n,\Sigma,c_1,C_1,r_0,\mu);\quad (r,\theta)\in [r_0,\infty)\times \Sigma,
\end{equation}
where $|G|, |G_m|$ are the norm of $G_{pqk}$, $G^{ij}_{mkl}$ ($m=1,2$) respectively.
\end{lemma}
\begin{proof}
    Using the differentiation rule for inverse matrix and differentiating \eqref{eq bar gij u}, we get
\begin{equation}\label{eq bar gij inver der}
\begin{aligned}
   &\bar g^{ij}_k=-\bar g^{ip}\bar g^{qj}\bar g_{pq,k}=-\bar g^{ip}\bar g^{qj}(2u_kh_{pq}+2uh_{pqk}+(u^2h_p^mh_q^n)_kg_{mn}+(u_pu_q)_k).
\end{aligned}
\end{equation}
Then \eqref{eq bar gij deri u} follows by rearranging terms. Differentiating \eqref{eq bar gij inver der} again, and using 
\begin{align*}
    &\bar g^{ij}_{kl}=\bar g^{im}\bar g_{mn,l}\bar g^{tp}\bar g_{pq,k}\bar g^{qj}+\bar g^{ip}\bar g_{pq,k}\bar g^{qm}\bar g_{mn,l}\bar g^{nj}-\bar g^{ip}\bar g_{pq,kl}\bar g^{qj}
\end{align*}
we get \eqref{eq bar gij sec deri u}. The last statement follows the same as the proof of (2) of Lemma \ref{lem E stru}.
\end{proof}
Next, we consider the case when the normal graph $F$ evolves by MCF, i.e. $F$ satisfies \eqref{eq mcf}. Then $u$ evolves by \eqref{eq flow outer}. First, we calculate the equation for the gradient of $u$ along MCF.
\begin{lemma}\label{lem gra diff ine cut app}
   Suppose $u$ satisfies \eqref{eq flow outer}, then $|\nabla u|^2$ satisfies 
    \begin{equation}\label{eq gra diff ine app}
        \partial_t|\nabla u|^2=[-\bar g^{ip}\bar g^{qj}(|\nabla u|^2_pu_q+|\nabla u|^2_qu_p)+2G_{ijk}u_k](u_{ij}+V_{ij})+\bar g^{ij}[(|\nabla u|^2)_{ij}-2u_{pi}u_{pj}]+2\bar g^{ij}u_k\tilde V_{ijk}.
    \end{equation}
    where
    \begin{equation}\label{eq Vij}
        V_{ij}=V_{ij}(h,\bar g,u,\nabla u)=-uh_i^kh_{kj}-h_{ij}-\frac{u_m\bar{g}^{mn}}{1-u_iu_j\bar{g}^{ij}}(u_ih_j^k+u_jh_i^k+u(h_i^k)_j)(g_{kn}+uh_{nk}),
    \end{equation}
    \begin{align*}
        \tilde V_{ijk}=&(h_j^mh_{ki}-h_{ji}h_k^m)u_m+ V_{ijk},
\end{align*}
and $V_{ijl}$ is the covariant derivative of $V_{ij}$. Moreover, for any non-negative smooth function $\eta$ in dependent of $t$, we have
\begin{equation}\label{eq gra diff ine cut app}
    \begin{aligned}
      \partial_t(\eta |\nabla u|^2)=&[-\bar g^{ip}\bar g^{qj}\big((\eta|\nabla u|^2)_pu_q+(\eta|\nabla u|^2)_qu_p-(\eta_p u_q+\eta_q u_p)|\nabla u|^2\big)+2\eta G_{ijl}u_l](u_{ij}+V_{ij})\\
        +&\bar g^{ij}[(\eta |\nabla u|^2)_{ij}-2\eta_i|\nabla u|^2_j-\eta_{ij}|\nabla u|^2-2\eta u_{pi}u_{pj}]+2\eta\bar g^{ij}u_l\tilde V_{ijl}.
    \end{aligned}
\end{equation}
\end{lemma}
\begin{proof}
    Differentiating \eqref{eq flow outer} gives
    \begin{align*}
        u_{tk}=&\bar g^{ij}_l(u_{ij}+V_{ij})+\bar g^{ij}(u_{kij}+\tilde V_{ijk}).
    \end{align*} 
On the other hand, by \eqref{eq bar gij deri u}, 
\begin{align*}
        2u_l\bar g^{ij}_l
        =-\bar g^{ip}\bar g^{qj}(|\nabla u|^2_pu_q+|\nabla u|^2_qu_p)+2G_{ijl}u_l.
    \end{align*}
    Thus,
    \begin{align*}
        &\partial_t|\nabla u|^2=2u_ku_{tk}=2u_l\bar g^{ij}_k(u_{ij}+V_{ij})+\bar g^{ij}[(|\nabla u|^2)_{ij}-2u_{pi}u_{pj}]+2\bar g^{ij}u_k\tilde V_{ijk}\\
        =&[-\bar g^{ip}\bar g^{qj}(|\nabla u|^2_pu_q+|\nabla u|^2_qu_p)+2G_{ijk}u_k](u_{ij}+V_{ij})+\bar g^{ij}[(|\nabla u|^2)_{ij}-2u_{pi}u_{pj}]+2\bar g^{ij}u_k\tilde V_{ijk},
    \end{align*} 
which is the first equation. Using \eqref{eq gra diff ine app}, it's easy to get  \eqref{eq gra diff ine cut app}. 
\end{proof}
Then we calculate the evolution equation for Hessian of $u$.
\begin{lemma}\label{lem sec der diff ine cut app}
    Suppose $u$ satisfies \eqref{eq flow outer}, and if $c_1g_{ij}\leq \bar g_{ij}\leq C_1g_{ij}$, for some uniform constants $C_1\geq c_1>0$, $X=\mC=\{(r,\theta)|r\in\mathbb R_+,\theta\in \Sigma\}$ is a regular cone in $\mathbb R^{n+1}$, and $u$ satisfies \eqref{eq assum u and grad} for some $\mu,r_0>0$. Then for any $\varepsilon>0$, $|\nabla^2 u|^2$ satisfies 
    \begin{equation}\label{eq sec der diff ine app}
    \begin{aligned}
       \partial_t|\nabla^2u(x,\theta,t)|^2\leq& \bar g^{ij}|\nabla^2 u|^2_{ij}+(C(n,\Sigma,c_1,C_1,r_0,\varepsilon)|\nabla u|^2-2c_1^2)|\nabla^2 u|^4-\bar g^{ip}\bar g^{qj}(|\nabla^2u|^2_pu_q+|\nabla^2u|^2_qu_p)\\&(u_{ij}+V_{ij})
       +(\varepsilon+C(n,\Sigma,c_1,C_1,r_0)\mu^2-2c_1) |\nabla^3 u|^2+C(n,\Sigma,c_1,C_1,r_0,\mu,\varepsilon)(|\nabla^2 u|^3+1),
    \end{aligned}
    \end{equation}
    for $(r,\theta)\in [r_0,\infty)\times \Sigma$. Moreover, for any non-negative smooth function $\eta$ independent of $t$, we have
    \begin{equation}\label{eq sec der diff ine cut app}
    \begin{aligned}
          \partial_t(\eta|\nabla^2u|^2)\leq& \bar g^{ij}[(\eta|\nabla^2 u|^2)_{ij}-2\eta_i |\nabla^2 u|^2_j-\eta_{ij}|\nabla^2 u|^2]+\eta[C(n,\Sigma,c_1,C_1,r_0,\varepsilon)|\nabla u|^2-2c_1^2)|\nabla ^2u|^4]\\
       &-\bar g^{ip}\bar g^{qj}[(\eta |\nabla^2u|^2)_pu_q+(\eta|\nabla^2u|^2)_qu_p-(\eta_p u_q+\eta_q u_p)|\nabla^2 u|^2](u_{ij}+V_{ij})\\
       &+\eta[(\varepsilon+C(n,\Sigma,c_1,C_1,r_0)\mu^2-2c_1^2)  |\nabla^3 u|^2+C(n,\Sigma,c_1,C_1,r_0,\mu,\varepsilon)(|\nabla^2 u|^3+1)],
    \end{aligned}
    \end{equation}
     for $(r,\theta)\in [r_0,\infty)\times \Sigma$.
\end{lemma}
\begin{proof}
 In the following of the proof, we consider $(r,\theta)\in [r_0,\infty)\times \Sigma$. Differentiating \eqref{eq flow outer} two times with respect to $x_k,x_l$ and multiplying $u_{kl}$ and summing, we get
    \begin{align*}
        &\partial_t|\nabla^2u(x,\theta,t)|^2=2u_{lk}u_{klt}
        =
        I_1+I_2+II_1+II_2+III_1+III_2
    \end{align*}
       where 
    \begin{align*}
        I_1:=&2u_{kl}\bar g^{ij}_{kl}u_{ij},\quad I_2=2u_{kl}\bar g^{ij}_{kl} V_{ij},\quad II_1:=4u_{kl}\bar g^{ij}_ku_{ijl}\\
    II_2:=&4u_{kl}\bar g^{ij}_kV_{ijl}, \quad III_1:=2u_{kl}\bar g^{ij}u_{ijkl},\quad III_2:=2u_{pq}\bar g^{ij}V_{ijkl}.
    \end{align*}
    and $V_{ij}$ is the tensor in lemma \ref{lem gra diff ine cut app}, and $V_{ijl}, V_{ijkl}$ are the covariant derivatives of $V_{ij}$. By \eqref{eq bar gij sec deri u}, and changing the order of covariant derivative using Ricci identities, we obtain  
    \begin{align*}
        I_1\leq& -2\bar g^{ip}\bar g^{qj}((u_{klp}+R_{plk}^mu_m)u_q+u_p(u_{klq}+R_{qlk}^mu_m))u_{kl}u_{ij}
    -2\bar g^{ip}\bar g^{qj}u_{kl}u_{ij}(u_{pk}u_{ql}+u_{pl}u_{qk})\\
    &+C(n,c_1,C_1)|\nabla u|^2|\nabla ^2u|^4+C(n)|G_1||\nabla^2 u|^2+C(n)|G_2||\nabla^ 2u|^3\\
    \leq& -\bar g^{ip}\bar g^{qj}(|\nabla^2u|^2_pu_q+|\nabla^2u|^2_qu_p)u_{ij} -2\bar g^{ip}\bar g^{qj}u_{kl}u_{ij}(u_{pk}u_{ql}+u_{pl}u_{qk})\\
    &+C(n,\Sigma,c_1,C_1,r_0)|\nabla u|^2(|\nabla^2u|^2+|\nabla^2u|^4)
    +C(n)|G_1||\nabla^2 u|^2+C(n)|G_2||\nabla^ 2u|^3,
    \end{align*}
where $|G_m|$ is the norm of $G^{ij}_{mkl}$ ($m=1,2$) in Lemma \ref{lem bar g der sec der u}. By \eqref{eq G Gm structure}, we get
\begin{align*}
    I_1\leq & -\bar g^{ip}\bar g^{qj}(|\nabla^2u|^2_pu_q+|\nabla^2u|^2_qu_p)u_{ij} -2\bar g^{ip}\bar g^{qj}u_{kl}u_{ij}(u_{pk}u_{ql}+u_{pl}u_{qk})\\
    &+C(n,\Sigma,c_1,C_1,r_0)|\nabla u|^2|\nabla^2u|^4
    +C(n,\Sigma,c_1,C_1,r_0,\mu)(|\nabla^2 u|^2+|\nabla^ 2u|^3).
\end{align*}
Since 
\begin{align*}
    \bar g^{im}\bar g^{nj}u_{pq}u_{ij}(u_{mp}u_{nq}+u_{mq}u_{np})=2\operatorname {Tr}({\bar g^{-1}(\nabla^2u)^3} \bar g^{-1}(\nabla^2u))\geq 2c_1^2(n,\Sigma) |\nabla^2u|^4,
\end{align*}
we get
\begin{equation}\label{eq imp sed outer noncom I1 app}
\begin{aligned}
    I_1\leq & -\bar g^{ip}\bar g^{qj}(|\nabla^2u|^2_pu_q+|\nabla^2u|^2_qu_p)u_{ij} -2c_1^2|\nabla^2u|^2\\
    &+C(n,\Sigma,c_1,C_1,r_0)|\nabla u|^2|\nabla^2u|^4
    +C(n,\Sigma,c_1,C_1,r_0,\mu)(|\nabla^2 u|^2+|\nabla^ 2u|^3).
\end{aligned}
\end{equation}
For $I_2$, we note from the expression of $V_{ij}$ in \eqref{eq Vij} and the assumption, we have
\begin{align*}
    |V_{ij}|\leq C(n,c_1,C_1)(r^{-1}\mu+|h|)\leq C(n,\Sigma,c_1,C_1)(\mu+1)r^{-1}.
\end{align*}
This together with \eqref{eq bar gij sec deri u} and \eqref{eq G Gm structure} yields,
\begin{equation}\label{eq imp sed outer noncom I2 app}
\begin{aligned}
    I_2\leq& C(n,c_1,C_1)|\nabla^2u|(|\nabla^2u|^2+|\nabla u||\nabla^2u|+|G_1|+|G_2|)|V|-2u_{kl}\bar g^{ip}\bar g^{qj}(u_{klp}u_q+u_{klq}u_p) V_{ij}(h,\bar g,u,\nabla u)\\
    =& C(n,c_1,C_1,\mu,r_0)(|\nabla^2u|^3+1)-\bar g^{ip}\bar g^{qj}(|\nabla^2u|^2_pu_q+|\nabla^2u|^2_qu_p)V_{ij}
\end{aligned}
\end{equation} 
for $(r,\theta)\in [r_0,\infty)\times \Sigma$. For $II_1$, by \eqref{eq bar gij deri u}, and Cauchy inequality (assuming $|\nabla^2 u|$ large), we have
\begin{align*}
     II_1=&-4u_{kl}u_{ijl}\bar g^{ip}\bar g^{qj}\big(u_{pk}u_q+u_pu_{qk}+G_{pqk}\big)
     \leq C(n,c_1,C_1)|\nabla^2u||\nabla^3u|(|\nabla^2u||\nabla u|+|G|)\\
    \leq& \varepsilon |\nabla^3 u|^2+C(\varepsilon,c_1,C_1)|\nabla^2 u|^4|\nabla u|^2+C(\varepsilon,c_1,C_1)|\nabla^2u|^2|G|^2.
\end{align*}
By \eqref{eq G Gm structure} and \eqref{eq assum u and grad}, we get
\begin{equation}\label{eq imp sed outer noncom II1 app}
    II_1\leq  \varepsilon |\nabla^3 u|^2+C(\varepsilon,c_1,C_1)|\nabla^2 u|^4|\nabla u|^2+C(\varepsilon,c_1,C_1,r_0)|\nabla^2u|^2\mu^2.
\end{equation}
For $II_2$, we have by \eqref{eq bar gij deri u}, \eqref{eq Vij}, and \eqref{eq assum u and grad} that
\begin{align*}
    II_2=&-4u_{kl}\bar g^{ip}\bar g^{qj}\big(u_{pk}u_q+u_pu_{qk}+G_{pqk}\big)V_{ijl}\leq C(n,\Sigma,c_1,C_1,r_0)(\mu+1)(|\nabla^2u||\nabla u|+|G|)(|\nabla^2u|+1)
\end{align*}
By \eqref{eq G Gm structure}, we get
\begin{equation}\label{eq imp sed outer noncom II2 app}
\begin{aligned}
    II_2\leq C(n,\Sigma,c_1,C_1,r_0,\mu)(|\nabla^2 u|^3+|\nabla u|^2+1)
\end{aligned}
\end{equation}
For $III_1$, 
    \begin{align*}
    III_1=&2u_{kl}\bar g^{ij}(u_{klij}+(R_{ilk}^mu_m)_j+R_{jli}^mu_{mk}+R_{jlk}^mu_{im}+(R_{ikj}^mu_m)_l)\\
    \leq &\bar g^{ij}(|\nabla^2 u|^2_{ij}-2u_{kli}u_{klj})+C(n,c_1,C_1,r_0)(|\nabla^2 u|^2+|\nabla ^2u|).
    \end{align*}
Since 
\begin{align*}
    \bar g^{ij}u_{kli}u_{klj}\geq c_1(n,\Sigma)|\nabla^3u|^3,
\end{align*}
we get
\begin{equation}\label{eq imp sed outer noncom III1 app}
    \begin{aligned}
    III_1\leq &\bar g^{ij}|\nabla^2 u|^2_{ij}-2c_1|\nabla^3u|^3+C(n,c_1,C_1,r_0)(|\nabla^2 u|^2+|\nabla ^2u|).
    \end{aligned}
\end{equation}
At last, by \eqref{eq Vij}, and Cauchy inequality
\begin{equation}\label{eq imp sed outer noncom III2 app}
\begin{aligned}
    III_2\leq& C(n,c_1,C_1)|\nabla^2u||V_{ijkl}|\leq C(n,\Sigma,c_1,C_1,r_0)|\nabla^2u| (|\nabla^2u|+1+|\nabla^2u|^2+\mu |\nabla^3u|)\\
    \leq& C(n,\Sigma,c_1,C_1,r_0)(|\nabla^2 u|^3+\mu^2|\nabla^3u|^2+|\nabla^2 u|^2+1)
\end{aligned}
\end{equation}
Combining \eqref{eq imp sed outer noncom I1 app}-\eqref{eq imp sed outer noncom III2 app}, and another use of Cauchy inequality, we get \eqref{eq sec der diff ine app}. Similar as the proof in Lemma \ref{lem gra diff ine cut app}, we get \eqref{eq sec der diff ine cut app}.
\end{proof}

\section{Calculus on cones}\label{sec cone prea}
\subsection{Geometry on cones}
Let $\mC=\mathbb R_+\times\Sigma=\{(r,\theta)|r\in\mathbb R_+,\theta\in\Sigma\}\subset\mathbb R^{n+1}$ be a regular hypercone in $\mathbb R^{n+1}$, where $\Sigma=\mC\cap\mathbb S^{n}$ is the link of $\mC$, which is a smooth hypersurface in $\mathbb S^{n-1}$. Then we can parametrize $\mC$ by $\mC=r\Psi(\theta_1,\cdots,\theta_{n-1})$, where $\Psi(\theta)\in\Sigma\subset\mathbb S^n$, $r\in\mathbb R_+$. We collect some basic facts about cones here. We use index Lattin letters $i,j,k,l\cdots$ to denote index $1,2,\cdots,n$, and use Greek letters $\alpha,\beta,\gamma,\delta,\cdots$ to denote index $=1,2,\cdots,n-1$. Denote the metric and second fundamental form of $\mC$ by $g_{ij}$ and $h_{ij}$ respectively, we have 
\begin{align*}
    & g_{rr}=\mC_r\cdot \mC_r=1,\quad g_{r\alpha}=\mC_r\cdot \mC_{\alpha}=0, g_{\alpha\beta}=\mC_{\alpha}\cdot \mC_{\beta}=r^2 g_{\Sigma,\alpha\beta}.\\
     &h_{rr}=-\mC_{rr}\cdot\nu=0\cdot\nu=0,h_{r\alpha}=-\mC_{r\alpha}\cdot\nu=-\Psi_{\alpha}\cdot\nu=\frac{1}{r}\mC_{\alpha}\cdot\nu=0;\\
    & h_{\alpha\beta}=-\mC_{\alpha\beta}\cdot\nu=-r\Psi_{\alpha\beta}\cdot\nu=r\bar h_{\alpha\beta},\text{ where }\bar h_{ij}=\Psi_{\theta_i\theta_j}\cdot\nu|_{r=1}.\\
     &h_r^r=g^{ri}h_{i r}=g^{rr}h_{rr}=0,h^r_\alpha=g^{ri}h_{i\alpha}=g^{rr}h_{r\alpha}=0,h^{\alpha}_r=g^{\alpha i}h_{i r}=0,\\
    & h^\alpha_\beta=g^{\alpha i}h_{i \beta}=g^{\alpha \gamma}h_{\gamma \beta}=r^{-1}\bar h^\alpha_\beta,\text{ where }\bar h^\alpha_\beta=h^\alpha_\beta|_{r=1}.\quad \alpha,\beta=1,2,\cdots,n-1.
\end{align*}
Next, we compute the Christoffel symbol $\Gamma_{ij}^k$ on $\mC$:
\begin{align*}
&\Gamma_{\alpha\beta}^\gamma=\Gamma_{\Sigma,\alpha\beta}^\gamma,\,\Gamma_{\alpha\beta}^r=-rg_{\Sigma,\alpha\beta},\,\Gamma_{\alpha r}^\gamma=\frac{1}{r}\delta_{\alpha\gamma},\,\Gamma_{\alpha r}^r=\Gamma_{rr}^\alpha=\Gamma_{rr}^r=0,\\
&r_{\alpha\beta}=r_{,\alpha\beta}-\Gamma_{\alpha\beta}^\gamma r_\gamma-\Gamma_{\alpha\beta}^rr_r=-\Gamma_{\alpha\beta}^rr_r=rg_{\Sigma,\alpha\beta},\\
&r_{\alpha r}=r_{,\alpha r}-\Gamma_{\alpha r}^\gamma r_\gamma=0,\quad r_{rr}=r_{,rr}=0.
\end{align*}
Then we compute the derivatives of the function $r$ on $\mC$.
\begin{equation}\label{eq z C2 norm}
    |\nabla_{\mC} r|^2=g^{ij}r_ir_j=1,\, |\nabla^2_{\mC}r|^2=g^{ik}g^{jl}r_{ij}r_{kl}=r^{-4}g_{\Sigma}^{\alpha\gamma}g_\Sigma^{\beta\delta}r^2 g_{\Sigma,\alpha\beta}g_{\Sigma,\gamma\delta}=\frac{n-1}{r^2}.
\end{equation}
where $\nabla_{\mC}$ denotes the differentiation on $\mC$, $|T|$ denotes the norm of a tensor $T$ on $\mC$. 

\subsection{A Morrey type inequality on cones}\label{sec Morrey}
Recall the definition of $\mathbf H$ in \eqref{eq def H}. We prove a version of Morrey inequality for functions in $\mathbf H$.. 
\begin{lemma}\label{lem Morrey}
    Functions in $\mathbf H$ are actually continuous, i.e. $\mathbf{H}\subset C(\mathbb R_+\times\Sigma)$. Moreover, fo any $v\in\mathbf H$, there holds,
    \begin{equation}\label{eq Morrey}
        |v(y,\theta)|\leq C(n,\Sigma)(\frac{1}{y^{\frac{n}{2}}}+e^{\frac{(y+1)^2}{4}})(\|\nabla v\|_W+\|v\|_W)
    \end{equation}
    for $y>0$.
\end{lemma}

\begin{proof} The proof follows that of Lemma 6.3 of \cite{GS} with a little modification. For simplicity, we omit $W$ in $\|\cdot\|_W$ in the proof. Let's first assume that $v\in C^1(\mathbb R_+\times\Sigma)\cap\mathbf H$. 

For any $(y,\varphi),(z,\theta)\in \mC$, let $\gamma_{(z,\theta)}$ be a geodesic on $\mC$ from $(z,\theta)$ to $(y,\varphi)$ of arclength parameter. By the fundamental theorem of calculus, 
\begin{align*}
v(y,\varphi)=v(z,\theta)+\int_0^{d((y,\varphi),(z,\theta))}\nabla v(\gamma_{(z,\theta)}(t))\cdot\dot{\gamma}_{(z,\theta)}(t)dt\leq v(z,\theta)+\int_0^{d((y,\varphi),(z,\theta))}|\nabla v(\gamma_{(z,\theta)}(t))|dt, 
\end{align*}
where  By H\"older inequality, this implies that
\begin{equation}\label{eq Holde ineq cone}
\begin{aligned}
    v(y,\varphi)^2\leq &C\left(v(z,\theta)^2+y\int_0^{d((y,\varphi),(z,\theta))}|\nabla v(\gamma_{(z,\theta)}(t))|^2dt\right)\\
    \leq &Cv(z,\theta)^2+C(\Sigma)y\int_0^{d((y,\varphi),(z,\theta))}|\nabla v(\gamma_{(z,\theta)}(t))|^2e^{-\frac{r(\gamma_{(z,\theta)}(t))^2}{4}}dt
\end{aligned}
\end{equation}
 For each $0<y\leq 1$, integrate this against $e^{-\frac{z^2}{4}}dv(z,\theta)$ over $(z,\theta)\in [\frac{y}{2},y]\times \Sigma$, then we get
\begin{align*}
v(y,\varphi)^2 \int_{\frac{y}{2}}^y\int_\Sigma z^{n-1}e^{-\frac{z^2}{4}}d\theta dz\leq CI_1+C(\Sigma)I_2
   \end{align*}
   where
\begin{align*}
    I_1:=\int_{\frac{y}{2}}^y\int_\Sigma v(z,\theta)^2z^{n-1}e^{-\frac{z^2
   }{4}}d\theta dz\leq \|v\|^2,
\end{align*}
and
   \begin{align*}
       I_2:=&y\int_{\frac{y}{2}}^y\int_\Sigma \int_0^{d((y,\varphi),(z,\theta))}|\nabla v(\gamma_{(z,\theta)}(t))|^2e^{-\frac{r(\gamma_{(z,\theta)}(t))^2}{4}}dtz^{n-1}e^{-\frac{z^2}{4}}d\theta dz\\
       \leq& y\int_{\frac{y}{2}}^y \int_\Sigma \int_0^{d((y,\varphi),(z,\theta))}|\nabla v(\gamma_{(z,\theta)}(t))|^2e^{-\frac{r(\gamma_{(z,\theta)}(t))^2}{4}}dtd\theta z^{n-1}e^{-\frac{z^2}{4}}dz\\
       \leq &y\int_{\frac{y}{2}}^y\int_\Sigma\int_0^{d((y,\varphi),(z,\theta))}|\nabla v(\gamma_{(z,\theta)}(t))|^2r(\gamma_{(z,\theta)})^{n-1}e^{-\frac{r(\gamma_{(z,\theta)}(t))^2}{4}}d\theta dr e^{-\frac{z^2}{4}}dz\\
   \leq &y\|\nabla v\|^2\int_{\frac{y}{2}}^ye^{-\frac{z^2}{4}}dz\leq y^2\|\nabla v\|^2.
   \end{align*}
Thus
\begin{align*}
    v(y,\varphi)^2\leq &C(n,\Sigma)(\int_{\frac{y}{2}}^{y}\int_\Sigma z^{n-1}e^{-\frac{z^2}{4}}d\theta_2dz)^{-1}[\|v\|^2+y^2\|\nabla v\|^2]\leq C(n,\Sigma)\frac{1}{y^{n}}[\|v\|^2+y^2\|\nabla v\|]^2\\
    \leq &C(n,\Sigma)y^{-n}[\|v\|^2+\|\nabla v\|^2]
\end{align*}
for $0<y\leq 1$. 
 
For $y\geq 1$, integrate \eqref{eq Holde ineq cone} against $e^{-\frac{z^2}{4}}dv(z,\theta)$ over $(z,\theta)\in [y,y+1]\times \Sigma$, then we get
\begin{align*}
   &v(y,\varphi)^2 \int_{y}^{y+1}\int_\Sigma z^{n-1}e^{-\frac{z^2}{4}}d\theta dz\leq CI_3+C(\Sigma)I_4,
   \end{align*}
   where 
   \begin{align*}
       I_3:=\int_{y}^{y+1}\int_\Sigma v(z,\theta)^2z^{n-1}e^{-\frac{z^2
   }{4}}d\theta dz\leq \|v\|^2,
   \end{align*}
   and
   \begin{align*}
       I_4:=&y\int_{y}^{y+1}\int_\Sigma \int_0^{d((y,\varphi),(z,\theta))}|\nabla v(\gamma_{(z,\theta)}(t))|^2e^{-\frac{r(\gamma_{(z,\theta)}(t))^2}{4}}dtz^{n-1}e^{-\frac{z^2}{4}}d\theta dz\\
       \leq& y\int_{y}^{y+1} \int_\Sigma \int_0^{d((y,\varphi),(z,\theta))}|\nabla v(\gamma_{(z,\theta)}(t))|^2e^{-\frac{r(\gamma_{(z,\theta)}(t))^2}{4}}dtd\theta z^{n-1}e^{-\frac{z^2}{4}}dz\\
       \leq& ye^{\frac{(y+1)^2}{4}}\int_{y}^{y+1}\int_\Sigma\int_0^{d((y,\varphi),(z,\theta))}|\nabla v(\gamma_{(z,\theta)}(t))|^2r(\gamma_{(z,\theta)})^{n-1}e^{-\frac{r(\gamma_{(z,\theta)}(t))^2}{4}}d\theta dr e^{-\frac{z^2}{4}}dz\\
       \leq& ye^{\frac{(y+1)^2}{4}}\|\nabla v\|^2\int_{y}^{y+1}e^{-\frac{z^2}{4}}dz\leq  ye^{\frac{(y+1)^2-y^2}{4}}\|\nabla v\|^2.
   \end{align*}
Thus
\begin{align*}
    v(y,\varphi)^2\leq &C(n,\Sigma)(\int_{y}^{y+1}\int_\Sigma z^{n-1}e^{-\frac{z^2}{4}}d\theta_2dz)^{-1}[\|v\|^2+ye^{\frac{(y+1)^2-y^2}{4}}\|\nabla v\|^2]\\
    \leq &C(n,\Sigma)y^{-n}e^{\frac{(y+1)^2}{4}}[\|v\|^2+ye^{\frac{(y+1)^2-y^2}{4}}\|\nabla v\|]^2
    \leq C(n,\Sigma)e^{\frac{(y+1)^2}{2}}[\|v\|^2+\|\nabla v\|^2]
\end{align*}
for $y\geq 1$. 

More generally, \eqref{eq Morrey} holds for $v\in\mathbf H$ since $C_c^1(\mC)$ is dense in $\mathbf H$.
\end{proof}

\section{Hardt-Simon's foliation}\label{sec HS foliation}
Suppose $\mC\subset\mathbb R^{n+1}$ is a regular area minimizing hypercone in $\mathbb R^{n+1}$, $\mathbb R^{n+1}\setminus \mC=E_+\cup E_-$ has two connected components $E_+,E_-$. By the result of \cite{HS},  (Theorem 2.1 of \cite{HS}), there is a smooth area minimizing hypersurface $S_+$ which foliates $E^+$ and has positive distance to the origin. Moreover, $S_+$ can be written as a normal graph over the cone $\mC=\{(r,\theta)|r\in \mathbb R_+,\theta\in \Sigma\}$ outside a big ball $B_{R_s}$, with profile function $\psi(r,\theta)$. That is
\begin{equation}\label{eq S+ para}
    S_+(r,\theta)=\mC(r,\theta)+\psi (r,\theta)\nu_{\mC}(r,\theta),\quad (r,\theta)\in[R_s,\infty)\times \Sigma,
\end{equation}
and $\psi$ has the asymptotics
\begin{equation}\label{eq psi asy}
    \psi(r,\theta)=cr^\alpha+O(r^{\alpha-{\tilde \alpha} }), \quad\text { as }r\to \infty. 
\end{equation}
for some $c>0$, ${\tilde \alpha} (\Sigma)>0$. Moreover, this asymptotic propagates to the derivatives by minimal surface equation. That is,
\begin{equation}\label{eq psi der asy}
    r^{|\gamma|}|\nabla^\gamma \psi(r,\theta)|\leq C(\Sigma,|\gamma|)r^\alpha, \quad |\gamma|\in\mathbb N, \quad(r,\theta)\in[R_s,\infty)\times \Sigma.
\end{equation}
By rescaling, for any $k>0$, 
\begin{equation}
    S_{\kappa,+}:={\kappa}^{\frac{1}{1-\alpha}}S_+
\end{equation}
has profile function
\begin{align*}
     \psi_{\kappa}(r,\theta):={\kappa}^{\frac{1}{1-\alpha}}\psi(k^{\frac{-1}{1-\alpha}}r,\theta),\quad (r,\theta)\in [{\kappa}^{\frac{1}{1-\alpha}}R_s,\infty)\times\Sigma,
\end{align*}
over $\mC$ outside $B_{{\kappa}^{\frac{1}{1-\alpha}}R_s}$. That is,
\begin{equation}\label{eq Sk+ para}
     S_{\kappa,+}(r,\theta)=\mC(r,\theta)+\psi_{\kappa}(r,\theta)\nu_{\mC}(r,\theta),\quad (r,\theta)\in [{\kappa}^{\frac{1}{1-\alpha}}R_s,\infty)\times\Sigma.
\end{equation}
Note that
\begin{equation}\label{eq psik asym}
    \psi_{\kappa}(r,\theta)=ckr^\alpha+O((\frac{r}{1-\alpha})^{\alpha-{\tilde \alpha} })= ckr^\alpha+O(r^{\alpha-{\tilde \alpha} })\text { as }r\to\infty,
\end{equation}
and 
\begin{equation}\label{eq psik der asy}
     r^{|\gamma|}|\nabla^\gamma \psi_{\kappa}(r,\theta)|\leq C(\Sigma,|\gamma|,k)r^\alpha, \quad |\gamma|\in\mathbb N, \quad(r,\theta)\in[R_s,\infty)\times \Sigma.
\end{equation}
Thus, by changing $k>0$, we may assume that $c=1$ with out loss of generality. 

From the parametrization \eqref{eq Sk+ para} , we can use $\{(r,\theta)|r\geq {\kappa}^{\frac{1}{1-\alpha}}R_s,\theta\in \Sigma\}$ for the coordinates of $S_{\kappa,+}\setminus B_{{\kappa}^{\frac{1}{1-\alpha}}R_s}$. We introduce a global coordinates $\{(\tilde r,\theta)|\tilde r\geq r_0,\theta\in \Sigma \}$ on $S_{\kappa,+}$. For any point $p\in S_{\kappa,+}$, $p$ has coordinates $(\tilde r,\theta)$ if $P_{\mC}(p)=(r,\theta)$, where $P_{\mC}:\mathbb R^{n+1}\to \mC$ is the projection from  $\mathbb R^{n+1}$ to $\mC$. If $P_{\mC}(p)$ has more then one point, then we take $\theta$ to be any $\theta_0$ in this projection. This can only happen if $P_{\mC}(p)=(r_0,\theta)$. By this definition, $(\tilde r,\theta)=S_{\kappa,+}(r,\theta)$ if $r\geq {\kappa}^{\frac{1}{1-\alpha}}R_s$.


\begin{thebibliography}{99}
\bibitem{BDG}E. Bombieri, E. De Giorgi and E. Giusti, {\em Minimal cones and the Bernstein problem}, Invent. Math. 7 (1969), 243-268.
\bibitem{CHL}K. Choi, J. Huang, and T. Lee, {\em Ancient mean curvature flows with finite total curvature}, Trans. Amer. Math. Soc. 378 (2025), 6401-6424.
\bibitem{EH}K. Ecker, G. Huisken, {\em Interior estimates for hypersurfaces moving by mean curvature}, Invent. Math. 105 (1991), no. 3, 547-569.
\bibitem{GS}S. Guo, N. Sesum, {\em Analysis of Vel\/azquez's solution to the mean curvature flow with a type II singularity}, Comm. Partial Differential Equations 43 (2018), no.2, 185-285.
\bibitem{HS}R. Hardt, L. Simon, {\em Area minimizing hypersurfaces with isolated singularities}, J. Reine Angew. Math. 362 (1985), 102-129.
\bibitem{HV}M. A. Herrero, J.J. Vel\/azquez, {\em A blow up result for semilinear heat equations in the supercritical case}, 1994, unpublished notes.
\bibitem{I95}T. Ilmanen, {\em Singularities of mean curvature flow of surfaces}, \href{ https://people.math.ethz.ch/~ilmanen/papers/sing.ps}{https://people.math.ethz.ch/~ilmanen/papers/sing.ps} (1995).
\bibitem{La}H. B. Lawson, Jr., {\em The equivariant Plateau problem and interior regularity}, Trans. Amer. Math. Soc. 173 (1972), 231-249. 
\bibitem{Liu}Z. Liu, {\em Blow Up of compact mean curvature flow solutions with bounded mean curvature}, Arxiv, arXiv:2403.16515.
\bibitem{SP}P. Simoes, {\em On a class of minimal cones in $\mathbb R^n$},  Bull. Am. Math. Soc. 80, (1974), 488-489. 
\bibitem{S}M. Stolarski, {\em Existence of mean curvature flow singularities with bounded mean curvature}, Duke Math. J. 172 (2023), no. 7, 1235-1292.
\bibitem{V94}J. J. L. Velázquez, {\em Curvature blow-up in perturbations of minimal cones evolving by mean curvature flow}, Ann. Scuola Norm. Sup. Pisa Cl. Sci. (4) 21 (1994), no. 4, 595-628.
\end{thebibliography}
\end{document}